\documentclass[10pt]{amsart}
\usepackage[utf8]{inputenc}
\usepackage{a4wide,amssymb,mathtools,bbm,bm,esint,dsfont,accents}
\usepackage{comment}
\usepackage{endnotes}
\usepackage[dvipsnames]{xcolor}
\usepackage[pdfborder={0 0 0}]{hyperref}

\newcommand{\entint}[2]{{\zeta}^{{#1}\to{#2}}}

\newcommand{\schrplan}[3]{{\pi}^{{#1}\to{#2},{#3}}}
\newcommand{\EOTplan}[3]{{\pi}^{{#1},{#2}}_{#3}}
\newcommand{\Ent}{\mathsf{Ent}}

\newcommand{\IND}{\mathbbm{1}}

\newcommand{\supp}{\mathrm{supp}}

\newcommand{\De}{\mathrm{d}}
\newcommand{\rmR}{\ensuremath{\mathrm{R}}}
\newcommand{\rmP}{\ensuremath{\mathrm{P}}}

\newcommand{\cC}{\ensuremath{\mathcal C}}

\newcommand{\cH}{\ensuremath{\mathcal H}}
\newcommand{\cI}{\ensuremath{\mathcal I}}

\newcommand{\cL}{\ensuremath{\mathcal L}}

\newcommand{\cN}{\ensuremath{\mathcal N}}

\newcommand{\cP}{\ensuremath{\mathcal P}}
\newcommand{\cQ}{\ensuremath{\mathcal Q}}

\newcommand{\cS}{\ensuremath{\mathcal S}}

\newcommand{\cW}{\ensuremath{\mathcal W}}

\newcommand{\frm}{\ensuremath{\mathfrak m}}

\newcommand{\frp}{\ensuremath{\mathfrak p}}
\newcommand{\frq}{\ensuremath{\mathfrak q}}

\newcommand{\R}{\mathbb{R}}
\newcommand{\RD}{\mathbb{R}^d}
\newcommand{\N}{\mathbb{N}}

\newcommand{\norm}[1]{\left\lVert #1 \right\rVert}

\newcommand{\Id}{\operatorname{Id}}

\newcommand{\abs}[1]{\left\lvert #1 \right\rvert}
\newcommand{\oo}{{\infty}}

\newcommand{\shortmapor}[1]{{\stackrel{#1}{\longrightarrow}}}
\newcommand{\weakto}{\rightharpoonup} 

\newcommand{\CD}{{\sf CD}}
\newcommand{\hp}{{\sf p}}
\newcommand{\sfd}{{\sf d}}

\title[Convergence to the Brenier map and quantitative stability]{Gradient estimates for the Schr\"odinger potentials: convergence to the Brenier map \\and quantitative stability}

\date{\today}

\author{Alberto Chiarini}
\address{Università degli studi di Padova}
\curraddr{Dipartimento di Matematica ``Tullio Levi-Civita'', 35121 Padova, Italy.}
\email{chiarini@math.unipd.it}
\thanks{While this work was written, AC was associated to INdAM and the group GNAMPA}

\author{Giovanni Conforti}
\address{\'Ecole Polytechnique}
\curraddr{D\'epartement de Math\'ematiques Appliqu\'es,
Palaiseau, France.}
\email{giovanni.conforti@polytechnique.edu}
\thanks{GC acknowledges funding from the grant SPOT (ANR-20-CE40-
0014)}

\author{Giacomo Greco}
\address{Eindhoven University of Technology}
\curraddr{Department of Mathematics and Computer Science, 5600MB Eindhoven, Netherlands.}
\email{g.greco@tue.nl}
\thanks{GG acknowledges support from NWO Research Project 613.009. ``Analysis meets Stochastics: Scaling limits in complex systems"}

\author{Luca Tamanini}
\address{Università Bocconi}
\curraddr{Bocconi Institute for Data Science and Analytics (BIDSA), 20136 Milano, Italy.}
\email{luca.tamanini@unibocconi.it}
\thanks{LT acknowledges support from the PRIN 2017 (Prot. 2017TEXA3H) ``Gradient flows, Optimal Transport and Metric Measure Structures''}


\keywords{Optimal transport, Schr\"odinger problem, Schr\"odinger potentials, entropic regularization, quantitative stability, gradient estimates, curvature lower bounds \\
\indent\textit{Data availability.} Data sharing not applicable to this article as no datasets were generated or analysed during the current study.}

\newtheorem{theorem}{Theorem}[section]
\newtheorem{lemma}[theorem]{Lemma}
\newtheorem{prop}[theorem]{Proposition}

\newtheorem{corollary}[theorem]{Corollary}

\theoremstyle{definition}

\theoremstyle{remark}
\newtheorem{remark}[theorem]{Remark}

\definecolor{antiquefuchsia}{rgb}{0.57, 0.36, 0.51}
\numberwithin{equation}{section}

\begin{document}

\maketitle

\begin{abstract}
    We show convergence of the gradients of the Schr\"odinger potentials to the (uniquely determined) gradient of  Kantorovich potentials in the small-time limit under general assumptions on the marginals, which allow for unbounded densities and supports. Furthermore, we provide novel quantitative stability estimates for the optimal values and optimal couplings for the Schr\"odinger problem (SP), that we express in terms of a negative order weighted homogeneous Sobolev norm. The latter encodes the linearized behavior of the $2$-Wasserstein distance between the marginals.
    The proofs of both results highlight for the first time the relevance of gradient bounds for Schr\"odinger potentials, that we establish here in full generality, in the analysis of the short-time behavior of Schr\"odinger bridges.
    Finally, we discuss how our results translate into the framework of quadratic Entropic Optimal Transport, that is a version of SP more suitable for applications in machine learning and data science.
\end{abstract}


\section{Introduction}

The Schr\"odinger problem (SP) as introduced by E.\ Schr\"odinger himself in the seminal work \cite{Schr32} consists in finding the most likely evolution of a cloud of independent Brownian particles between two observed configurations. When the two observation times are very close, SP converges  in a suitable sense to the optimal transport (OT) problem. This connection was first understood by Mikami in \cite{Mikami04} and deeply investigated thereafter. By now, researchers from diverse areas ranging from functional inequalities \cite{gentil2020dynamical} to statistical machine learning \cite{peyre2019computational}, control engineering \cite{chen2021stochastic}, and numerics for PDEs \cite{benamou2021optimal} find common ground on the study of SP. 

In its classical formulation, given two probability measures $\mu,\,\nu$ (i.e.\ the two observed configurations), SP consists in
\begin{equation*}
\text{minimize }  \cH(\pi|\rmR_{0,T}) \text{ under the constraint } \pi\in\Pi(\mu,\nu)\,,
\end{equation*}
where $\rmR_{0,T}\in\cP(\mathbb{R}^{2d})$ denotes the heat kernel 
$\mathrm{R}_{0,T}(\De x\De y)\propto\exp(-|x-y|^2/2T)\De x\De y$, $\Pi(\mu,\nu)$ denotes the set of of couplings between $\mu$ and $\nu$, whereas for any two probability measures $\frp,\,\frq$ the quantity $\cH(\frp|\frq)$ is the relative entropy of $\frp$ w.r.t.\ $\frq$, defined as
\[
\cH(\frp|\frq)\coloneqq\begin{cases}
\displaystyle{\int \log\frac{\De \frp}{\De\frq}\,\De\frp} \quad&\text{if }\frp\ll\frq\,,\\
+\oo\quad&\text{otherwise.}
\end{cases}
\]
We postpone to Section \ref{sec:reg} a precise definition of the relative entropy $\cH(\cdot|\frq)$ when the reference $\frq$ is not a probability measure. The mathematical formulation of SP as an entropy minimization problem comes from the theory of large deviations (Sanov's Theorem), which states that the large deviation rate function for the empirical measure of independent Brownian motions is equal to the relative entropy w.r.t.\ the stationary Wiener measure. Interestingly, an equivalent formulation of SP as a stochastic optimal control problem (via Girsanov's Theorem, \cite[Theorem 2]{Leo2012Girsanov}) is possible, that is, SP is equivalent to
\begin{equation*}
\begin{aligned}
\,&\text{minimize }  \|\alpha^{\rmP}(t,x)\|_{L^2(\De t\times\De\rmP)}\, \text{ among the path-probabilities } \rmP\in\cP(\cC([0,T],\RD)\,\text{ s.t.}\\
\, & \rmP\text{-a.s.} \,\text{ it holds }\quad \begin{cases}
    \De X_t=\alpha^{\rmP}(t,X_t)\De t+\De B_t\,,\\
    X_0\sim \mu,\quad X_T\sim \nu\,.
\end{cases}
\end{aligned}
\end{equation*}
The above control problem has to be understood as choosing the drift $\alpha^{\rmP}$ that steers the system from $\mu$ to the target measure $\nu$ and that at the same time minimizes the cost of such deviation (i.e. $\|\alpha^{\rmP}(t,x)\|_{L^2(\De t\times\De\rmP)}$). For this reason the optimizer in the above problem is usually referred to as optimal control or optimal \emph{corrector}. This dynamical interpretation of SP provides also a tight connection with PDEs. Indeed, the optimal corrector is of gradient type, i.e.\ $\alpha(t,x)=\nabla u_t(x)$, and $u_t$ solves the Hamilton-Jacobi-Bellman equation
\begin{equation}\label{eq:intro:HJB}
    \begin{cases}
       \displaystyle{\partial_t u_t+\frac12\Delta u_t+\frac12|\nabla u_t|^2=0} \\
       u_T=\psi\,,
    \end{cases}
\end{equation}
for some measurable function $\psi$ as final condition (see e.g.\ \cite[Proposition 6]{LeoSch}).

\bigskip

Remarkably, SP is also equivalent to the entropic optimal transport problem (EOT), which is of great interest in data science and machine learning, and which amounts to
\begin{equation*}
\text{minimize }  \int \frac{|x-y|^2}{2}\,\De\pi+\varepsilon\,\cH(\pi|\mu\otimes\nu) \text{ under the constraint } \pi\in\Pi(\mu,\nu)\,,
\end{equation*}
with $\varepsilon>0$ being a regularization parameter. Note indeed that, in the stochastic control interpretation of SP, $T$ plays the role of a time horizon, so that the marginals $\mu,\nu$ are respectively observed at the times $t=0$ and $t=T$. But if the time interval $[0,T]$ is rescaled to $[0,1]$, then $T$ can be naturally interpreted as a diffusion/regularization parameter, as after algebraic manipulations it appears in \eqref{eq:intro:HJB} in front of the Laplacian. Hence, if $\Ent(\cdot)$ denotes the relative entropy w.r.t.\ the Lebesgue measure and we take into account the following identity
\begin{equation*}
T\,\cH(\pi|\rmR_{0,T})-T\,\Ent(\mu)-T\,\Ent(\nu)=\int \frac{|x-y|^2}{2}\,\De\pi+T\,\cH(\pi|\mu\otimes\nu)\quad \text{ for any }\pi\in\Pi(\mu,\nu),
\end{equation*}
we immediately deduce the equivalence between SP and EOT. This allows to translate results from SP to EOT and viceversa.
There are however some small differences between the two and, perhaps more importantly, some hypotheses that are natural from the viewpoint of particle systems, may not be so from a computational standpoint, and viceversa. We shall elaborate more on this point at various places throughout the article.
\vspace{\baselineskip}

The present article aims at bringing some new insights on two basic questions, namely: the convergence of the gradient of the Schr\"odinger potentials (the entropic counterpart to Kantorovich potentials) in the small-time limit; and the stability of optimal solutions with respect to variations in the marginal inputs. In a nutshell, our main contributions are:

\begin{itemize}
\item Theorem \ref{thm:gradient}, where we show under a curvature-dimension condition and a finite Fisher information assumption a strong $L^2$-convergence result for the gradient of the Schr\"odinger potentials (resp.\ Schr\"odinger map) to the uniquely determined gradient of (any) Kantorovich potential (resp.\ to the Brenier map \cite{brenier1991polar}). See~\eqref{phi:decomposition} and Section \ref{intro:grad} for the terminology.
\item Theorem \ref{stab:piani}, where we exhibit explicit upper bounds for the symmetric relative entropy between optimal plans for different sets of marginals. These stability bounds are expressed in terms of a weighted Sobolev norm between marginals and some universal constants depending on curvature lower bounds on the underlying manifold.
\end{itemize}
For both results, we shall provide right after their statements an extensive comparison with the rapidly developing literature on the subject and further comments. An aspect that, however, deserves to be pointed out at this stage is the fact that Theorem \ref{thm:gradient} can be linked with the Monge-Ampère equation. Indeed, it is well known that in $\R^d$ (and under suitable assumptions that we will not detail; the reader may refer to the survey \cite{DePFig2014}) the Brenier map is the gradient of a convex function $u$ solving
\[
\det(D^2 u(x)) = \frac{\rho(x)}{\sigma(\nabla u(x))}, \qquad \textrm{for } \mu\textrm{-a.e. } x,
\]
coupled with the ``boundary condition'' $\nabla u(\supp(\mu))=\supp(\nu)$, where $\rho,\sigma$ are the densities w.r.t.\ the Lebesgue measure of $\mu,\nu$ respectively. This allows to state that $x-2T\varphi^T$ is an ``approximated'' solution to the above Monge-Ampère equation, $\varphi^T$ being the first Schr\"odinger potential in the pair $(\varphi^T,\psi^T)$ associated with $\mu$ and $\nu$ (cf.\ \eqref{phi:decomposition}).

\medskip

The backbone of our proof strategy are gradient estimates for Schr\"odinger potentials, that we call \emph{corrector estimates} in view of the above stochastic control interpretation, since they provide contractive estimates for $\|\alpha^{\rmP}(t,x)\|_{L^2(\De\rmP)}$ for $t\in[0,T]$. These bounds have been put forward in \cite{conforti2019second} to show a quantitative form of convex entropy decay along entropic interpolations and are here proven under much weaker assumptions. So far, corrector estimates have found applications in the proof of new functional inequalities and in the study of the long-time behavior of entropic interpolations, see \cite{conforti2019second,luca2022costa} for example, and are known to be equivalent to the celebrated Bakry-\'Emery condition, see \cite{clerc2020variational}. However, to the best of our knowledge, the results of this paper are the first to show the interest of such bounds in the analysis of the convergence of SP towards OT, in which one has to deal with short-time limits instead of large-time limits. It is worth pointing out that the problem we consider here is the Schr\"odinger problem in its large deviations formulation. 

\subsection*{Outline} We continue this introduction by presenting the main mathematical objects studied in this work. We then move to the statement of the main results, that are Theorem \ref{thm:gradient} and Theorem \ref{stab:piani}. Section \ref{sec:reg} provides some additional background on SP, recalls some known results on the subject, and contains the proof of the aforementioned corrector estimates. Section \ref{sec:grad} is devoted to the proof of the convergence to the Brenier map. The proof of the stability estimates is carried out in Section \ref{stab:section} and in its final part we show how these results can lead to new stability estimates for the quadratic entropic OT problem. Finally, we prove in Appendix \ref{app:log:int} an auxiliary result on the log-integrability of measurable functions under some positive and negative $L^p$-moments, upon which all calculations made in Section \ref{stab:section} rely.

\subsection*{Setting and notation} 

In this work we are going to consider a generalization of the classical SP to a more abstract setting and with a non-trivial dynamics. For this reason, let us consider a weighted manifold $(M,g,\frm)$, where $(M,g)$ is a smooth, connected, complete (possibly non-compact) Riemannian manifold without boundary and with metric tensor $g$, and $\frm$ denotes the $\sigma$-finite invariant measure associated to the SDE
\begin{equation}\label{sde}
\De X_t=-\nabla U(X_t)\,\De t+\sqrt{2}\,\De B_t,
\end{equation}
i.e.\ $\De\frm(x)=e^{-U(x)}{\rm vol}(\De x)$, where $U$ is a $C^2$ potential and $B_t$ denotes a standard Brownian motion on the Riemannian manifold $(M,g)$. By $\sfd$ we shall denote the geodesic distance induced by $g$, so that we will more often refer to $(M,\sfd,\frm)$ rather than to $(M,g,\frm)$.

Let $\rmR_{0,T}$ denote the joint law at time $0$ and $T$ of the solution of the previous SDE started at its equilibrium measure $\frm$ and denote its density by $\hp_T = \frac{\De\rmR_{0,T}}{\De(\frm\otimes\frm)}$. Then, for any two probability measures $\mu,\,\nu\in\cP(M)$ the Schr\"odinger problem and the Schr\"odinger cost are respectively defined as the following variational problem and its optimal value
\begin{equation}\label{SP}\tag{SP}
\cC_T(\mu,\nu) \coloneqq \inf_{\pi\in\Pi(\mu,\nu)}\cH(\pi|\rmR_{0,T})\,.
\end{equation}
As we have already mentioned before, one of the reasons why SP has become more and more popular nowadays is that it can be seen as an entropic regularization for the Optimal Transport problem. This has been particularly studied in the more natural setting of EOT, whose abstract formulation reads as follows
\begin{equation}\label{EOT}\tag{$\varepsilon$EOT}
\cS^\varepsilon(\mu,\nu)\coloneqq \inf_{\pi\in\Pi(\mu,\nu)}\int c(x,y)\,\De\pi+\varepsilon\,\cH(\pi|\mu\otimes\nu)\,,\qquad\varepsilon>0\,,
\end{equation}
for any given cost function $c\in L^1(\mu\otimes\nu)$ and noisy regularisation parameter $\varepsilon>0$.
We refer the reader to the lecture notes \cite{Marcel:notes} for an extensive introduction to abstract EOT and its connections with SP. 
The above problem admits a unique optimizer $\EOTplan{\mu}{\nu}{\varepsilon}\in\Pi(\mu,\nu)$ and its density is of the form  
\begin{equation}\label{entropic:potentials}
\frac{\De\EOTplan{\mu}{\nu}{\varepsilon}}{\De(\mu\otimes \nu)}=\exp\bigg(\frac{\varphi_\varepsilon\oplus\psi_\varepsilon-c}{\varepsilon}\bigg) \qquad \mu\otimes\nu\text{-a.s.}
\end{equation}
where $(\varphi_\varepsilon \oplus \psi_\varepsilon)(x,y) := \varphi_\varepsilon(x) + \psi_\varepsilon(y)$, for two measurable functions $\varphi_\varepsilon\in L^1(\mu)$ and $\psi_\varepsilon\in L^1(\nu)$, unique up to the trivial transformation $(\varphi_\varepsilon,\psi_\varepsilon) \mapsto (\varphi_\varepsilon+c,\psi_\varepsilon-c)$ with $c \in \R$, whom we will refer to as the (entropic) potentials.

Related to SP and the EOT problem above, respectively taking the small-time and small-noise limit, is the  dual formulation of the Monge-Kantorovich problem and the associated potentials. That is,
\begin{equation}\label{W2}
\frac14\,\cW_2^2(\mu,\nu)=\sup_{\varphi\in L^1(\mu),\, \psi\in L^1(\nu)\,\colon\varphi\oplus\psi\leq\frac14 \sfd^2(\cdot,\cdot)}\left(\int\varphi\,\De\mu+\int\psi\,\De\nu\right)\,,
\end{equation}
where $\cW_2^2(\mu,\nu)$ denotes the Wasserstein distance of order two. Two functions $(\varphi_0,\psi_0)$ forming an optimal dual pair are called Kantorovich potentials.

\vspace{\baselineskip}

As concerns the ambient space, throughout the paper we are going to assume that $(M,\, \sfd,\,\frm)$ satisfies a curvature-dimension condition, i.e.\ either one of the following holds 
\begin{equation}\label{CD}\tag{CD}
\begin{minipage}{0.9\textwidth}
    \begin{itemize}
     \item $(M,\, \sfd,\,\frm)$ satisfies $\CD(\kappa,N)$ for some $\kappa\in\R$ and $N<+\oo$, or
     \item  $(M,\, \sfd,\,\frm)$ satisfies $\CD(\kappa,\oo)$ for some $\kappa\in\R$ and $\frm(M)=1$,
    \end{itemize}
    \end{minipage}
\end{equation} 
For the reader who is not familiar with the Bakry-\'Emery theory of lower Ricci bounds for diffusion Markov semigroups, beside suggesting the monograph \cite{bakry2013analysis} for the precise definitions of the $\CD(\kappa,N)$ and $\CD(\kappa,\oo)$ conditions, we would like to give a brief explanation. Recalling that $\frm := e^{-U}{\rm vol}$, the constant $\kappa$ plays the role of a lower bound on the Bakry-\'Emery Ricci curvature tensor ${\rm Ric}_U := {\rm Ric}_g + {\rm Hess}(U)$, while $N$ has to be seen as an upper bound on the \emph{effective} dimension of the generator induced by \eqref{sde}. In a more compact way, $\CD(\kappa,N)$ holds if and only if
\[
{\rm Ric}_U - \frac{\nabla U \otimes \nabla U}{N-n} \geq \kappa g, \qquad n:=\dim(M).
\]
For instance, if $M$ is endowed with $\frm={\rm vol}$ and has Ricci curvature bounded from below by $\kappa$, then the $\CD(\kappa,N)$ condition holds with $N=\dim(M)$, as in this case $U=0$ and the Bakry-\'Emery Ricci curvature tensor coincides with the classical one, while the generator induced by \eqref{sde} coincides with the Laplace-Beltrami operator. As an example of $\CD(\kappa,\oo)$ space, instead, one should think of $\R^d$ endowed with the standard Gaussian $\gamma_d$. In this case, $\CD(1,\oo)$ holds but no $\CD(\kappa,N)$ for $\kappa \in \R$ can be satisfied (for instance, because $\CD(\kappa,N)$ implies that $\frm$ is locally doubling, which is not the case for $\gamma_n$). This shows why $N$ is not a topological dimension but is instead related to the generator induced by \eqref{sde}, which in this case is the Witten Laplacian $\Delta_g - \nabla U \cdot \nabla$.

As concerns the marginals, we are going to assume that they satisfy
\begin{equation}\label{H1}\tag{H1}\mu,\,\nu\in\cP_2(M)\,\quad\text{and}\quad\,\cH(\mu|\frm),\,\cH(\nu|\frm)<+\oo\,;\end{equation}
and that either
\begin{equation}\label{H2}\tag{H2}\hspace{-1.2cm}\begin{minipage}{0.9\textwidth}
     \begin{itemize}
         \item  $\displaystyle\frac{\De\mu}{\De\frm},\frac{\De\nu}{\De\frm}\in L^\oo(\frm)$ and are compactly supported, or
         \item for $\frp=\mu,\nu$, $\displaystyle\frac{\De\frp}{\De\frm}$ is locally bounded away from zero on ${\rm int}(\supp(\frp))$ and $\frp(\partial\supp(\frp))=0$.
     \end{itemize}
    \end{minipage}
    \hfill
\end{equation}
The curvature-dimension condition and \eqref{H1} imply (cf.\ Proposition \ref{prop:fg} below) that SP admits a unique minimizer, which we will refer to as the Schr\"odinger optimal plan $\schrplan{\mu}{\nu}{T}\in\Pi(\mu,\nu)$, and that there exist two measurable functions $\varphi=\varphi^T,\,\psi=\psi^T$, known as the \emph{Schr\"odinger potentials} and unique up to the trivial transformation $(\varphi,\psi) \mapsto (\varphi+c,\psi-c)$ with $c \in \R$, such that 
\begin{equation}\label{phi:decomposition}
\frac{\De\schrplan{\mu}{\nu}{T}}{\De\rmR_{0,T}}=\,e^{\varphi\oplus \psi}\qquad\rmR_{0,T}\textrm{ -a.s.}
\end{equation}
Equivalently, there exist two non-negative measurable functions $f,g$, unique up to the trivial transformation $(f,g) \mapsto (cf,c^{-1}g)$ with $c > 0$, such that
\begin{equation}\label{eq:fg-decomposition}
\frac{\De\schrplan{\mu}{\nu}{T}}{\De\rmR_{0,T}}=\,f \otimes g \qquad\rmR_{0,T}\textrm{ -a.s.}
\end{equation}
where $(f \otimes g)(x,y) := f(x)g(y)$. The pair $(f,g)$ is known in the literature under the name of \emph{$fg$-decomposition}.

We are going to show (cf.\ Proposition \ref{lemma:gia} below) that, solely under assumptions \eqref{H1} and \eqref{H2}, the curvature-dimension condition entails the following estimates for the gradients of the Schr\"odinger potentials
\begin{equation}\label{corr:est}
   \norm{\nabla \varphi_T^T}_{L^2(\nu)}^2\leq \frac{1}{E_{2\kappa}(T)}\biggl[\cC_T(\mu,\nu)-\cH(\nu|\frm)\biggr],\quad\norm{\nabla \psi_0^T}_{L^2(\mu)}^2\leq \frac{1}{E_{2\kappa}(T)}\biggl[\cC_T(\mu,\nu)-\cH(\mu|\frm)\biggr]
\end{equation}
where $\psi_0^T\coloneqq \log P_T e^\psi$, $\varphi^T_T\coloneqq \log P_T e^\varphi$, $P_T$ denotes the semigroup associated to the SDE \eqref{sde} (defined for any non-negative measurable function $u$ as the conditional expected value $P_t u(x)\coloneqq \mathbb{E}[u(X_t)|X_0=x]$, cf. \cite{bakry2013analysis}) and where $E_{2\kappa}$ is defined as
\begin{equation}\label{eq:curvature-factor}
E_{2\kappa}(t)\coloneqq\int_0^t e^{2\kappa s}\De s.
\end{equation}
We refer to this type of bounds as \emph{corrector estimates}. Indeed, it is immediate to observe that $\varphi^T_t := \log P_t e^\varphi$ and $\psi^T_t := \log P_{T-t}e^\psi$ solve the following Hamilton-Jacobi-Bellman (HJB) equations
\[
\begin{cases}
    \partial_t\psi^T_t + |\nabla\psi^T_t|^2 + \Delta\psi^T_t=0\\
    \psi^T_T=\psi
\end{cases}
\qquad \textrm{and} \qquad
\begin{cases}
    \partial_t\varphi^T_t - |\nabla\varphi^T_t|^2 - \Delta\varphi^T_t=0\\
    \varphi^T_0=\varphi
\end{cases} 
\]
on $[0,T] \times M$. The former (up to the scaling parameter $1/2$) corresponds to \eqref{eq:intro:HJB} and its validity comes from the fact that $\nabla \psi^T_t$ is the optimal corrector in the stochastic optimal control formulation. The second HJB equation can be seen as a time-reversal of the former, and indeed $\nabla \varphi^T_t$ provides the optimal corrector for a time-reversed SP. For this reason $\nabla \psi^T_t$ and $\nabla \varphi^T_t$ are also usually referred to respectively as \emph{forward} and \emph{backward} corrector. Let us further observe that above we consider the same semigroup $(P_t)_{t\in [0,T]}$ for both the forward and backward corrector since the underlying dynamics \eqref{sde} is reversible in time.

The aim of this paper is showing that the curvature-dimension condition \eqref{CD} and the corrector estimates \eqref{corr:est} allow to tackle two different aspects of the Schr\"odinger problem. On the one hand, they will be instrumental to show the convergence of the gradients of the potentials in the small-time limit; on the other hand, they will lead to new stability estimates for SP.

\subsection{Convergence of the gradients of the Schr\"odinger potentials}\label{intro:grad}

A cornerstone result in (quadratic) optimal transport \cite{brenier1991polar} states (actually under weaker assumptions than the standing ones on $\mu,\nu$) the existence of a measurable function $\tau\colon M\to M$ such that $\tau_{\#}\mu=\nu$ and 
\begin{equation}\label{tau}
\int\sfd^2(x,\tau(x))\,\De\mu=\cW_2^2(\mu,\nu)=\inf_{\pi\in\Pi(\mu,\nu)}\int\sfd^2(x,y)\,\De\pi\,.
\end{equation}
The function $\tau$ is called the \emph{Brenier map} from $\mu$ to $\nu$. After \cite[Proposition 3.1]{Figalli07}, \cite[Theorem 1.1]{FigGigli2011} (see also the discussion therein) it is well known that $\tau=\Id-2\nabla\varphi^0$, where $\varphi^0$ is a Kantorovich potential  associated to \eqref{W2}. Our first main contribution is the strong convergence of the rescaled \emph{gradients} of the Schr\"odinger potentials to $\nabla\varphi^0$, which is uniquely determined although $\varphi^0$ is not, and thus the \emph{Schr\"odinger map} ${\rm Id}-2T\nabla\varphi^T$ from $\mu$ to $\nu$ to the Brenier map from $\mu$ to $\nu$. To state the result, we introduce the Fisher information, which is defined for any probability measure $\frp\ll\frm$ as the quantity
\begin{align*}
\cI(\frp)\coloneqq\begin{cases}
\norm{\nabla\log\frac{\De\frp}{\De\frm}}_{L^2(\frp)}^2\quad&\text{if }\nabla\log\frac{\De\frp}{\De\frm}\in L^2(\frp)\,,\\
+\oo\quad&\text{otherwise.}
\end{cases}
\end{align*}

\begin{theorem}\label{thm:gradient}
Suppose $(M,\sfd,\frm)$ satisfies \eqref{CD} and that \eqref{H1} holds true. If $\frac{\De\mu}{\De\frm}$ is locally bounded away from $0$ on ${\rm int}(\supp(\mu))$, if $\mu(\partial\,\supp(\mu))=0$ and if the Fisher information $\cI(\mu)$ is finite, then the Schr\"odinger map from $\mu$ to $\nu$ converges strongly in $L^2(\mu)$ to the Brenier map from $\mu$ to $\nu$ in the small-time limit. Equivalently, there are $(\varphi^0,\psi^0)$ Kantorovich potentials such that as $T\downarrow 0$
\[ 
T\nabla\varphi^T\,\to\,\nabla\varphi^0\quad\text{strongly in }L^2(\mu)\,.
\] 
A similar result holds for $T\nabla\psi^T$ under the corresponding assumptions on $\nu$.
\end{theorem}

\subsubsection*{Literature review}  
 For both SP and EOT the small-time and small-noise limits have been extensively studied in the literature. For what concerns the convergence of the rescaled entropic cost to the squared Wasserstein distance, the first result is that of \cite{Mikami04}, eventually generalized in \cite{Leo12}.
 As a further step in the analysis of this convergence, in \cite{mark} (in dimension one), in \cite{ErbarMassRenger} (for the multi-dimensional case) and most recently in \cite{soumik19}, the first-order Taylor expansion of the entropic cost has been investigated showing that the first-order coefficient is given by the sum of the relative entropies of the marginals. Under stronger assumptions on the marginals, the second and fourth authors determined in \cite{conforti2019second} the second-order Taylor expansion of the entropic cost showing that the second-order term is related to the average of the Fisher information along the geodesic between the marginals (an analogous result has also been obtained slightly later in \cite{chizat2020faster}). 
 Most recently, in \cite{CPT22} the fourth author and his coauthors analyze the small-noise limit in the EOT setting under very general assumptions on the cost function $c$ (which allow for non-uniqueness of the optimal transport coupling for instance).
 
 Alongside this line of work, in \cite{Mikami04} and \cite{BerntonGhosalNutz} it is shown that in the small-time and small-noise limits the optimal solutions to SP and EOT respectively converge to the optimal coupling of the OT problem. When it comes to the convergence of dual optimizers,
 in the EOT setting  
  it is proven in \cite{nutz2021entropic} that the entropic potentials $(\varphi_\varepsilon,\psi_\varepsilon)$ converge to the Kantorovich potentials (when the latter are unique) associated to the Monge-Kantorovich problem with cost $c$
 \[
 \sup_{\varphi\in L^1(\mu),\, \psi\in L^1(\nu)\,\colon\varphi\oplus\psi\leq c}\left(\int\varphi\,\De\mu+\int\psi\,\De\nu\right)\,.
 \]
 More precisely, in \cite{nutz2021entropic} the authors prove in the small-noise limit $\varepsilon\downarrow 0$ that the sequences $\{\varphi_\varepsilon\}_{\varepsilon>0}$ and $\{\psi_\varepsilon\}_{\varepsilon>0}$ are (strongly) compact in $L^1(\mu)$ and $L^1(\nu)$ respectively, and that their accumulation points are optimizers of the above Monge-Kantorovich problem. Adapting their proof strategy and relying on \cite{Norris97}, we show under a curvature-dimension condition a similar convergence statement for SP, i.e.\ that the Schr\"odinger potentials  $(T\,\varphi^T,\,T\,\psi^T)$, as defined in \eqref{phi:decomposition}, $L^1$-strongly converge (up to restricting to a subsequence) to a pair of Kantorovich potentials $(\varphi^0,\psi^0)$. This is proven in Lemma \ref{lemma:l1-conv-potentials}. Heuristically speaking, this follows from the connection between EOT and SP given by taking $\varepsilon=T$ and $c=-T\log\frac{\De\rmR_{0,T}}{\De(\mu\otimes\nu)}$ in \eqref{EOT}, which leads to $\cS^T(\mu,\nu)=T\,\cC_T(\mu,\nu)$,  $\EOTplan{\mu}{\nu}{T}\equiv\schrplan{\mu}{\nu}{T}$ and to the equality $T(\varphi^T\oplus\psi^T)=\varphi_T\oplus\psi_T$ between the Schr\"odinger potentials $(\varphi^T,\psi^T)$, as defined in \eqref{phi:decomposition},
 and the entropic potentials, as defined in \eqref{entropic:potentials}.
 
 Regarding the convergence of the gradients of the Schrödinger potentials, to the best of our knowledge Theorem \ref{thm:gradient} is a novelty in the setting we are dealing with. Indeed, closely related results have been obtained in \cite{EstimationEOTmap} and \cite{CPT22}, but in a more restrictive setting (the quadratic Euclidean EOT problem) and under strong regularity assumptions. As regards to \cite{EstimationEOTmap}, uniform bounds on the Hessian of the Kantorovich potential $\varphi^0$ are required to prove a modified version of the aforementioned convergence. Furthermore, the authors work in a very specific setting where the marginals $\mu,\nu$ are compactly supported and with densities globally bounded away from $0$ and $\oo$ on their supports. Moreover they show the convergence of the gradients of the potentials associated to a modified EOT where $\mu$ and $\nu$ are replaced by empirical measures associated to an $n$-sample; therefore they take a regularization parameter $\varepsilon$ depending on the batch size $n$ and then consider the limit $n\to\oo$. A result closer to Theorem \ref{thm:gradient} can be found instead in \cite{CPT22}, but again the marginals $\mu,\nu$ have compact, smooth, and uniformly convex supports and their densities are H\"older continuous and bounded away from 0 therein.

\subsection{Quantitative stability in terms of a weighted homogeneous Sobolev norm}\label{intro:sub:stab}
As a second main contribution, we are interested in explicitly quantifying how 
sensitive are the optimal costs and optimal plans in SP to variations of the marginal constraints.
In this respect, we show (cf.\ Theorem \ref{stab:piani}) that the curvature-dimension condition implies a rather general and explicit stability result in terms of the symmetric relative entropy, i.e.
\[
\cH^\mathrm{sym}(\mu,\bar\mu)\coloneqq \cH(\mu|\bar\mu)+\cH(\bar\mu|\mu)\,,
\]
and in terms of a negative-order weighted homogeneous Sobolev norm, which is defined for any signed measure $\nu$ as follows:
\begin{equation*}
    \norm{\nu}_{\dot{H}^{-1}(\mu)} \coloneqq \sup\biggl\{ \abs{\langle h, \nu \rangle} \,: \norm{h}_{\dot{H}^1(\mu)}\leq 1 \biggr\}\,,\qquad\text{where }\norm{h}_{\dot{H}^1(\mu)}^2 \coloneqq \int \abs{\nabla h}^2 \De \mu\,.
\end{equation*}
This dual norm on the space of signed measures encodes the linearized behavior of the Wasserstein distance $\cW_2$ for infinitesimal perturbations (see \cite{Comp:H-1:W} and references therein). For instance, if $\mu\in\cP_2(\RD)$, $\mu\ll\cL_d$ and $\bar\mu^\varepsilon=(1+\varepsilon h)\mu$ for some $h\in L^\oo(\mu)$ with $\int_{\RD} h\De\mu=0$, then \cite[Theorem 7.26]{Villani:TopicsOT} implies
\[
\norm{\mu-\bar\mu^\varepsilon}_{\dot{H}^{-1}(\mu)}=\varepsilon \norm{h\,\mu}_{\dot{H}^{-1}(\mu)}\quad\text{and}\quad\norm{h\,\mu}_{\dot{H}^{-1}(\mu)}\leq\liminf_{\varepsilon\to0}\frac{\cW_2(\mu,\bar\mu^\varepsilon)}{\varepsilon}\,.
\]
Moreover, \cite{Comp:H-1:W} provides non-asymptotic comparisons between this Sobolev norm and the Wasserstein distance. In particular, it always holds $\cW_2(\mu,\bar\mu)\leq 2\norm{\mu-\bar\mu}_{\dot{H}^{-1}(\mu)}$, and when $(M,\sfd,{\rm vol})$ has non-negative Ricci curvature and if the densities $\frac{\De\mu}{\De{\rm vol}}$ and $\frac{\De\bar\mu}{\De{\rm vol}}$ are bounded away from $0$ and $\oo$, the norm $\norm{\mu-\bar\mu}_{\dot{H}^{-1}(\mu)}$ is equivalent to the Wasserstein distance $\cW_2(\mu,\bar\mu)$.

\bigskip

Apart from the curvature-dimension condition \eqref{CD} and from \eqref{H1} and \eqref{H2}, we assume the following integrability condition on the reference measure:
\begin{equation}\label{I}\tag{I}
     \exists\, r>0\,\colon\quad\int \frm\left(B_{\sqrt{T}}(x)\right)^{r\,T}e^{r\,\sfd^2(x,z)}\De\frm(x)<+\oo\quad\forall z\in M\,.
\end{equation}
The reason why we assume the above condition is that our computations will heavily rely on integrating the Schr\"odinger potentials of one problem against the marginal constraints of the other, which requires enough integrability for the former. 

Notice that if $\frm$ is a probability measure, the previous assumption is met as soon as $e^{r\,\sfd^2(x,z_0)}\in L^1(\frm)$ for some $r>0$ and $z_0\in M$, which is true for instance under a positive curvature condition $\CD(\kappa,\oo)$ with $\kappa>0$. Indeed the latter implies a logarithmic Sobolev inequality with parameter $\kappa^{-1}$ \cite[Corollary 5.7.1]{bakry2013analysis} and then by means of Herbst's argument \cite[Proposition 5.4.1]{bakry2013analysis} we get $e^{r\,\sfd^2(x,z_0)}\in L^1(\frm)$ for any $r<\frac{\kappa}{2}$. Moreover, as soon as $\frm$ is a probability measure the above condition \eqref{I} is time-independent and therefore it allows to consider the small-time limit for SP.

\begin{theorem}\label{stab:piani}
Let $(M,\sfd, \frm)$ satisfy \eqref{CD} and \eqref{I}. For any $(\mu,\nu)$ and $(\bar\mu,\bar\nu)$ satisfying \eqref{H1} and \eqref{H2} it holds 
\begin{equation}\label{eq:stab:plans}
\begin{aligned}
\cH&^\mathrm{sym}(\schrplan{\mu}{\nu}{T},\schrplan{\bar\mu}{\bar\nu}{T})\leq \cH^\mathrm{sym}(\mu,\bar\mu)+ \cH^\mathrm{sym}(\nu,\bar\nu)\\
+&\,\frac{1}{\sqrt{E_{2\kappa}(T)}}\biggl[\sqrt{\cC_T(\mu,\nu)-\cH(\mu|\frm)}\norm{\mu-\bar\mu}_{\dot{H}^{-1}(\mu)}+\sqrt{\cC_T(\mu,\nu)-\cH(\nu|\frm)}\norm{\nu-\bar\nu}_{\dot{H}^{-1}(\nu)}\biggr]\\
+&\,\frac{1}{\sqrt{E_{2\kappa}(T)}}\biggl[\sqrt{\cC_T(\bar\mu,\bar\nu)-\cH(\bar\mu|\frm)}\norm{\bar\mu-\mu}_{\dot{H}^{-1}(\bar\mu)}+\sqrt{\cC_T(\bar\mu,\bar\nu)-\cH(\bar\nu|\frm)}\norm{\bar\nu-\nu}_{\dot{H}^{-1}(\bar\nu)}\biggr]
\end{aligned}
\end{equation}
where $E_{2\kappa}$ is defined as in \eqref{eq:curvature-factor}. Moreover, under the same assumptions it also holds
\begin{equation}\label{eq:fish:stab:plans}\begin{aligned}\cH^\mathrm{sym}(\schrplan{\mu}{\nu}{T},\schrplan{\bar\mu}{\bar\nu}{T})\leq\,& \frac{1}{\sqrt{E_{2\kappa}(T)}}\biggl[\sqrt{\cI(\mu)}+\sqrt{\cC_T(\mu,\nu)-\cH(\mu|\frm)}\biggr]\norm{\mu-\bar\mu}_{\dot{H}^{-1}(\mu)}\\
+&\frac{1}{\sqrt{E_{2\kappa}(T)}}\biggl[\sqrt{\cI(\bar\mu)}+\sqrt{\cC_T(\bar\mu,\bar\nu)-\cH(\bar\mu|\frm)}\biggr]\norm{\bar\mu-\mu}_{\dot{H}^{-1}(\bar\mu)}\\
+& \frac{1}{\sqrt{E_{2\kappa}(T)}}\biggl[\sqrt{\cI(\nu)}+\sqrt{\cC_T(\mu,\nu)-\cH(\nu|\frm)}\biggr]\norm{\nu-\bar\nu}_{\dot{H}^{-1}(\nu)}\\
+&\frac{1}{\sqrt{E_{2\kappa}(T)}}\biggl[\sqrt{\cI(\bar\nu)}+\sqrt{\cC_T(\bar\mu,\bar\nu)-\cH(\bar\nu|\frm)}\biggr]\norm{\bar\nu-\nu}_{\dot{H}^{-1}(\bar\nu)}\,.
\end{aligned}\end{equation}
\end{theorem}

It is worth pointing out that this result and the forthcoming \eqref{eq:stab:cost}, \eqref{eq:fish:stab:cost} require \eqref{H2} only because the latter is required for the corrector estimates. Moreover, as it has already been mentioned, the integrability condition \eqref{I} does not play any role in the constants appearing in the above (and the following) stability bound. Indeed one could, at least formally, perform the computations that lead to \eqref{eq:stab:plans} and \eqref{eq:fish:stab:plans} without this integrability condition.
 
\subsubsection{Literature review} Recently, there has been an increasing interest in the quantitative stability for the EOT problem, which is strongly linked to SP. A first stability estimate for EOT in a general setting has been established in  \cite{ghosal2021stability}. There, the authors, without any integrability assumption, manage to prove a qualitative stability result by relying on a geometric notion inspired by the cyclical monotonicity property in Optimal Transport. 
To the best of our knowledge, the first quantitative stability result is due to Carlier and Laborde in \cite{CarlierLaborde} in the context of multi-marginal EOT. By considering bounded marginals equivalent to a common reference probability measure and a bounded cost, they show that the potentials are Lipschitz-continuous in $L^2$ and $L^\oo$ w.r.t.\ the densities of the marginals. As far as concerns quantitative stability for the primal optimizers, the first result appeared in the work \cite{Deligiannidis2021}. There the authors prove on compact metric spaces a quantitative uniform stability along Sinkhorn's algorithm which implies stability for EOT and more precisely an explicit $\cW_1$-Lipschitzianity of the optimizers w.r.t.\ the marginals. More recently, in \cite{Marcel:quant:stability} a quantitative stability result that holds on general metric spaces is shown, thus removing the compactness assumption at the cost of requiring some exponential integrability with respect to the marginals (condition met for instance when the marginals are sub-Gaussians). More precisely the authors prove a $\frac{1}{2p}$-H\"older continuity for $\cW_p$ provided that the cost function satisfies an abstract condition, introduced there in order to bound the optimal values of EOT by means of the Wasserstein distance between the corresponding optimizers.  
Such condition is met for a wide enough class of cost functions such as $c(x,y)=\abs{x-y}^p$ on the Euclidean space, with $p\in(1,\oo)$. Nevertheless, it is not easily verifiable in general, and in particular when considering the case  $c=-T\,\log \hp_T$ (which allows to translate results from EOT into results for SP). Lastly, it is worth mentioning that in the most recent \cite{Marcel:stab:potentials} the authors study the qualitative stability of the Schr\"odinger potentials associated to EOT in a general setting assuming the cost function $c$ satisfies $e^{\beta c}\in L^1(\mu\otimes\nu)$ for some $\beta>0$, which leads to the convergence of Sinkhorn's algorithm. The \eqref{CD} condition does not have a natural counterpart in the EOT setting, where the results we have just discussed have been established; for this reason our stability result is not implied by any of the stability bounds mentioned above.

\begin{remark}[On the integrability condition \eqref{I}]
The condition
\[
e^{r\,\sfd^2(x,z_0)}\in L^1(\frm)\quad\text{for some }r>0\text{ and }z_0\in M,\quad\frm(M)=1
\]
should be compared with (6.8) and (6.9) in \cite{Marcel:notes} in the EOT setting, by considering the cost function $c(x,y)=-T\log \hp_T$, combined with a Gaussian heat kernel lower bound (cf.\ \eqref{infinity:log:bound} and \eqref{N:log:bound} below). There, the authors are interested in getting uniform bounds on the Schr\"odinger potentials along Sinkhorn's iterates.  However let us stress out that \cite{Marcel:notes} requires $e^{r\,\sfd^2(x,y)}\in L^1(\mu\otimes\nu)$ for some $r>\frac1T$, which in particular does not suit the most interesting EOT regime, i.e.\ the small-noise (or equivalently small-time) limit. On the contrary, our condition only requires $r>0$ independently from the time-window $[0,T]$ and moreover we are able to pass the integrability assumption on the equilibrium measure $\frm$. This suits more the \emph{stability setting} since we would like to keep the assumptions on the marginals as light as possible.
As we have already mentioned before, in the most recent work \cite{Marcel:stab:potentials} the authors manage to prove the uniform integrability of the potentials along Sinkhorn's algorithm by solely requiring $e^{r\,\sfd^2(x,y)}\in L^1(\mu\otimes\nu)$ for some $r>0$ overcoming the small-noise issue present in \cite{Marcel:notes}.  
In conclusion, it is not surprising that our stability results require 
\eqref{I}, since we need enough integrability for the potentials, which is analogous to requiring $e^{r\,\sfd^2(x,y)}\in L^1(\mu\otimes\nu)$ in the EOT setting. 
\end{remark}

\begin{remark}
The cost $\cC_T(\mu,\nu)$ appearing on the right-hand sides of \eqref{eq:stab:plans} and \eqref{eq:fish:stab:plans} can be bounded by the independent coupling $\cH(\mu\otimes\nu|\rmR_{0,T})$ and henceforth (by combining \eqref{remappo}, \eqref{N:log:bound}, and \eqref{remappo2}) by the quantity \begin{equation*}T\,\cC_T(\mu,\nu)\leq  T\,\cH(\mu|\frm)+T\,\cH(\nu|\frm)+C(T+T^2)+M_2(\mu)+M_2(\nu)+C\,T\biggl[M_1(\mu)+M_1(\nu)\biggr]\,,\end{equation*}
where $M_1(\mu),\,M_1(\nu)$ and $M_2(\mu),\,M_2(\nu)$ denote the first and second moments of $\mu$ and $\nu$ and the constant $C\geq 0$ is independent of the marginals (and it is equal to $0$ if $(M,\sfd,\frm)$ satisfies a $\CD(\kappa,\oo)$ condition with $\frm(M)=1$). We have kept the explicit dependence on the cost in \eqref{eq:stab:plans} because this gives a bound which is sharper compared to the one with the independent coupling $\cH(\mu\otimes\nu|\rmR_{0,T})$, especially in the small-time limit $T\downarrow 0$.
  \end{remark}

\begin{remark}[Sharpness in the long-time regime]
Under a non-negative curvature condition we know that in the long-time limit $\schrplan{\mu}{\nu}{T}\weakto \mu\otimes\nu$ (cf.\ \cite[Lemma 3.1]{conforti2021formula}) and hence we expect 
\[\cH^\mathrm{sym}(\schrplan{\mu}{\nu}{T},\schrplan{\bar\mu}{\bar\nu}{T})\to\cH^\mathrm{sym}(\mu\otimes\nu,\bar\mu\otimes\bar\nu)=\cH^\mathrm{sym}(\mu,\bar\mu)+\cH^\mathrm{sym}(\nu,\bar\nu)\,.\]
This shows that the previous bound \eqref{eq:stab:plans} under a non-negative curvature condition is \emph{sharp}. Indeed, if $\kappa\geq 0$ \eqref{eq:stab:plans} implies
\[ 
\limsup_{T\to\oo} \cH^\mathrm{sym}(\schrplan{\mu}{\nu}{T},\schrplan{\bar\mu}{\bar\nu}{T})\leq\cH^\mathrm{sym}(\mu,\bar\mu)+\cH^\mathrm{sym}(\nu,\bar\nu)
\] 
and the above convergence is exponentially fast (of order $\approx e^{-\kappa T}/\sqrt{T}$). This is another confirmation that under the corrector estimates and a non-negative curvature condition we are able to efficiently describe the exact behavior of the Schr\"odinger problem, as it has already been shown in the context of the entropic turnpike estimates (cf.\ \cite[Theorem 1.4]{conforti2019second} in the classic setting, \cite[Theorem 1.5]{backhoff2020mean} for the mean field SP and \cite[Theorem 1.7]{CCG21} for the kinetic SP).
\end{remark}

\begin{remark} A stability result can also be stated at the level of the optimal value of the Schr\"odinger problem, after renormalizing it as follows
\[ 
\cS_T(\mu,\nu)=T\,\cC_T(\mu,\nu)-T\,\cH(\mu|\frm)-T\,\cH(\nu|\frm)\,.
\]
More precisely, our computations show that the quantity $\cS_T$ (hereafter called \emph{entropic cost}) satisfies (cf.\ Proposition \ref{stab:entropic:cost})
\begin{equation}\label{eq:stab:cost} \begin{aligned}
&\abs{\cS_T(\bar\mu,\bar\nu)-\cS_T(\mu,\nu)}\leq T\biggl[\cH^\mathrm{sym}(\mu,\bar\mu)\wedge \cH^\mathrm{sym}(\nu,\bar\nu)\biggr]\\
&+\frac{T}{\sqrt{E_{2\kappa}(T)}}\biggl[\sqrt{\cC_T(\mu,\nu)-\cH(\mu|\frm)}\norm{\mu-\bar\mu}_{\dot{H}^{-1}(\mu)}+\sqrt{\cC_T(\mu,\nu)-\cH(\nu|\frm)}\norm{\nu-\bar\nu}_{\dot{H}^{-1}(\nu)}\biggr]\\
&+\frac{T}{\sqrt{E_{2\kappa}(T)}}\biggl[\sqrt{\cC_T(\bar\mu,\bar\nu)-\cH(\bar\mu|\frm)}\norm{\bar\mu-\mu}_{\dot{H}^{-1}(\bar\mu)}+\sqrt{\cC_T(\bar\mu,\bar\nu)-\cH(\bar\nu|\frm)}\norm{\bar\nu-\nu}_{\dot{H}^{-1}(\bar\nu)}\biggr]
\end{aligned}
\end{equation}
and, under a finite Fisher information assumption, we can deduce that the Schr\"odinger cost enjoys certain Lipschitz-type estimates w.r.t.\ the $\dot{H}^{-1}$  norm  (cf.\ Proposition \ref{stab:fish:finite})
\begin{equation}\label{eq:fish:stab:cost}
\begin{aligned}
\abs{\cC_T(\bar\mu,\bar\nu)-\cC_T(\mu,\nu)} & \leq \frac{1}{\sqrt{E_{2\kappa}(T)}}\biggl[\sqrt{\cI(\mu)}+\sqrt{\cC_T(\mu,\nu)-\cH(\mu|\frm)}\biggr]\norm{\mu-\bar\mu}_{\dot{H}^{-1}(\mu)}\\
& + \frac{1}{\sqrt{E_{2\kappa}(T)}}\biggl[\sqrt{\cI(\bar\mu)}+\sqrt{\cC_T(\bar\mu,\bar\nu)-\cH(\bar\mu|\frm)}\biggr]\norm{\bar\mu-\mu}_{\dot{H}^{-1}(\bar\mu)}\\
& + \frac{1}{\sqrt{E_{2\kappa}(T)}}\biggl[\sqrt{\cI(\nu)}+\sqrt{\cC_T(\mu,\nu)-\cH(\nu|\frm)}\biggr]\norm{\nu-\bar\nu}_{\dot{H}^{-1}(\nu)}\\
& + \frac{1}{\sqrt{E_{2\kappa}(T)}}\biggl[\sqrt{\cI(\bar\nu)}+\sqrt{\cC_T(\bar\mu,\bar\nu)-\cH(\bar\nu|\frm)}\biggr]\norm{\bar\nu-\nu}_{\dot{H}^{-1}(\bar\nu)}\,.
\end{aligned}
\end{equation}

In the EOT setting, the Lipschitzianity of the costs is a well-studied problem. The most general and recent result is the one in \cite{Marcel:quant:stability} where the authors manage to prove Lipschitzianity  with respect to the $p$-Wasserstein distance provided that the cost function satisfies a certain abstract condition introduced therein. The above estimates \eqref{eq:stab:cost} and \eqref{eq:fish:stab:cost} are less tight than theirs, however as we have already mentioned our setting does not satisfy their abstract condition in general. Moreover, since the negative Sobolev norm is intimately related to the $\cW_2$-distance, the above bounds may lead to a $\cW_2$-Lipschitz estimate. We will address this in a future line of work. 
\end{remark}

\begin{remark}\label{small:time:rem}  Let us mention that both Theorem \ref{stab:piani} and the Lipschitz type bounds above are well behaved in the small-time limit $T\downarrow 0$. Indeed, since we have 
\begin{equation}\label{eq:small:cost}\lim_{T\to 0}T\cC_T(\mu,\nu)=\frac{\cW^2_2(\mu,\nu)}{4}\,,\end{equation}
  (cf.\ \cite[Remark 5.11]{GigTam21} for $\CD(\kappa,N)$, or next section and Remark \ref{limite:costo} for $\CD(\kappa,\oo)$), if \eqref{CD}, \eqref{H1}, \eqref{H2} and \eqref{I} (for some fixed time window $T_0>0$) hold true and if  $\cH^\mathrm{sym}(\mu,\bar\mu),\, \cH^\mathrm{sym}(\nu,\bar\nu)$ are both finite, then after rescaling it holds
\begin{align*}
    \limsup_{T\to0} T\,\cH^\mathrm{sym}(\schrplan{\mu}{\nu}{T},\schrplan{\bar\mu}{\bar\nu}{T})\leq&\, \frac{\cW_2(\mu,\nu)}{2}\,\biggl[\norm{\mu-\bar\mu}_{\dot{H}^{-1}(\mu)}+\norm{\nu-\bar\nu}_{\dot{H}^{-1}(\nu)}\biggr]\\
&\qquad\qquad\qquad+\frac{\cW_2(\bar\mu,\bar\nu)}{2}\,\biggl[\norm{\bar\mu-\mu}_{\dot{H}^{-1}(\bar\mu)}+\norm{\bar\nu-\nu}_{\dot{H}^{-1}(\bar\nu)}\biggr]\,,\\
\limsup_{T\to 0}\abs{\cS_T(\bar\mu,\bar\nu)-\cS_T(\mu,\nu)}\leq&\, \frac{\cW_2(\mu,\nu)}{2}\,\biggl[\norm{\mu-\bar\mu}_{\dot{H}^{-1}(\mu)}+\norm{\nu-\bar\nu}_{\dot{H}^{-1}(\nu)}\biggr]\\
&\qquad\qquad\qquad+\frac{\cW_2(\bar\mu,\bar\nu)}{2}\,\biggl[\norm{\bar\mu-\mu}_{\dot{H}^{-1}(\bar\mu)}+\norm{\bar\nu-\nu}_{\dot{H}^{-1}(\bar\nu)}\biggr]\,.
\end{align*}
The $T$ prefactor on the left-hand side of the first displacement should not be surprising. For instance, it also appears in the quantitative stability bounds for the optimizers of EOT (cf.\ \cite[Theorem 3.11]{Marcel:quant:stability}). In fact, we cannot expect in general to bound the symmetric entropy of two optimal couplings since, even assuming the weak convergence of the Schr\"odinger optimizer $\schrplan{\mu}{\nu}{T}$ to the optimal $\cW_2$-coupling $\EOTplan{\mu}{\nu}{0}$, it may happen that $\cH^\mathrm{sym}(\EOTplan{\mu}{\nu}{0},\EOTplan{\bar\mu}{\bar\nu}{0})=+\oo$ while the right-hand side stays finite. For instance, consider $M=(0,1)$ with the marginals $\mu,\,\nu,\,\bar\mu$ equal to the Lebesgue measure on $(0,1)$ and $\De\bar\nu(x)=
(\log 2)2^{x}\,\De x$. Then $\cH^\mathrm{sym}(\nu,\bar\nu)$ and the left-hand sides of our bounds are finite but $\cH^\mathrm{sym}(\EOTplan{\mu}{\nu}{0},\EOTplan{\mu}{\bar\nu}{0})=+\oo$ since the two optimal couplings  $\EOTplan{\mu}{\nu}{0}$ and $\EOTplan{\mu}{\bar\nu}{0}$ have disjoint supports (the first is supported on the diagonal while the latter is supported on the graph of $L(x)=\log_2 x$). For sake of completeness, the finiteness of $\norm{\nu-\bar\nu}_{\dot{H}^{-1}(\nu)}$ is due to \cite[Theorem 5]{Comp:H-1:W}, while the finiteness of $\norm{\bar\nu-\nu}_{\dot{H}^{-1}(\bar\nu)}$ follows from the former and \cite[Lemma 2]{Comp:H-1:W}.
\end{remark}

\medskip

\begin{remark}[An application to the quadratic EOT] For the interested reader, in Section \ref{app:EOT} we provide an explicit connection between SP and the Euclidean EOT problem with quadratic cost
\begin{equation*}
\cS^\varepsilon(\mu,\nu)\coloneqq \inf_{\pi\in\Pi(\mu,\nu)}\int |x-y|^2\,\De\pi+\varepsilon\,\cH(\pi|\mu\otimes\nu)\,,\qquad\varepsilon>0\,,
\end{equation*}
and we show how our estimates and results can be translated to the EOT setting. Here we just mention the fact that our results work for any choice of marginals in $\cP_2(\RD)$ just satisfying a modified version of \eqref{H1} and \eqref{H2}, where the reference measure $\frm$ should be replaced by the Lebesgue measure, without requiring any sub-Gaussianity property. Indeed in Section \ref{app:EOT} we will show that if $\cH^\mathrm{sym}(\mu,\bar\mu),\, \cH^\mathrm{sym}(\nu,\bar\nu)$ are both finite then, by introducing the constant $C_\varepsilon=\frac{d\varepsilon}{2}\log(4\pi\varepsilon)$ we may write
\begin{equation*}\begin{aligned}
    \abs{\cS^\varepsilon(\bar\mu,\bar\nu)-\cS^\varepsilon(\mu,\nu)}\leq \varepsilon\biggl[\cH^\mathrm{sym}(\mu,\bar\mu)\wedge \cH^\mathrm{sym}(\nu,\bar\nu)&\biggr]\\
    +2 \biggl[\sqrt{\cS^\varepsilon(\mu,\nu)+\varepsilon\,\cH(\nu|\cL_d)+C_\varepsilon}\norm{\mu-\bar\mu}_{\dot{H}^{-1}(\mu)}&+\sqrt{\cS^\varepsilon(\mu,\nu)+\varepsilon\,\cH(\mu|\cL_d)+C_\varepsilon}\norm{\nu-\bar\nu}_{\dot{H}^{-1}(\nu)}\biggr]\\
    +2\biggl[\sqrt{\cS^\varepsilon(\bar\mu,\bar\nu)+\varepsilon\,\cH(\bar\nu|\cL_d)+C_\varepsilon}\norm{\bar\mu-\mu}_{\dot{H}^{-1}(\bar\mu)}&+\sqrt{\cS^\varepsilon(\bar\mu,\bar\nu)+\varepsilon\,\cH(\bar\mu|\cL_d)+C_\varepsilon}\norm{\bar\nu-\nu}_{\dot{H}^{-1}(\bar\nu)}\biggr]
\end{aligned}\end{equation*}
and
\begin{equation*}\begin{aligned}
    \varepsilon\,\cH^\mathrm{sym}(\EOTplan{\mu}{\nu}{\varepsilon},\EOTplan{\bar\mu}{\bar\nu}{\varepsilon})\leq\, \varepsilon\,\cH^\mathrm{sym}(\mu,\bar\mu)+ \varepsilon\,\cH^\mathrm{sym}(\nu,\bar\nu)&\,\\
    +2 \biggl[\sqrt{\cS^\varepsilon(\mu,\nu)+\varepsilon\,\cH(\nu|\cL_d)+C_\varepsilon}\norm{\mu-\bar\mu}_{\dot{H}^{-1}(\mu)}&+\sqrt{\cS^\varepsilon(\mu,\nu)+\varepsilon\,\cH(\mu|\cL_d)+C_\varepsilon}\norm{\nu-\bar\nu}_{\dot{H}^{-1}(\nu)}\biggr]\\
    +2\biggl[\sqrt{\cS^\varepsilon(\bar\mu,\bar\nu)+\varepsilon\,\cH(\bar\nu|\cL_d)+C_\varepsilon}\norm{\bar\mu-\mu}_{\dot{H}^{-1}(\bar\mu)}&+\sqrt{\cS^\varepsilon(\bar\mu,\bar\nu)+\varepsilon\,\cH(\bar\mu|\cL_d)+C_\varepsilon}\norm{\bar\nu-\nu}_{\dot{H}^{-1}(\bar\nu)}\biggr]\,.
\end{aligned}\end{equation*}
Similarly to what happens in the small-time limit for SP (cf.\ Remark \ref{small:time:rem} above), the above stability estimates behave well in the small-noise limit (cf.\ \eqref{small:noise:EOT} below).
\end{remark}

\section{Preliminaries}\label{sec:reg}

This section is divided in two halves. In the first part, after recalling the setting, auxiliary results, and introducing further notation, in Proposition \ref{prop:fg} we prove the existence of the $fg$-decomposition and of the Schr\"odinger potentials. The second half is dedicated to the proof of the corrector estimates \eqref{corr:est} which is given in Proposition \ref{lemma:gia}. 

\subsection{Setting}
We start by recalling the assumptions on the triplet $(M,\sfd,\frm)$ that will hold true throughout the whole paper: $M$ is a smooth, connected, complete (possibly non-compact) Riemannian manifold without boundary and with metric tensor $g$; $\sfd$ is the distance induced by $g$; $\frm$ denotes the $\sigma$-finite measure $\De\frm(x)=e^{-U(x)}{\rm vol}(\De x)$, where ${\rm vol}$ is the volume measure of $M$ and $U$ is a $C^2$ potential. We denote by $\cP_p(M)$, $p\geq 1$, the set of probability measures on $M$ with finite $p$-moment, that is such that for a fixed $z_0\in M$ it holds
\[
M_p(\frp)\coloneqq\int\sfd^p(y,z_0)\De\frp(y)<+\oo\,.
\] 
Let us also notice that $\cP_2(M)\subseteq\cP_1(M)$ and in particular $M_1(\frp)\leq 1+M_2(\frp)$ for any $\frp\in\cP_2(M)$.

The curvature assumption \eqref{CD} will also be always satisfied by the triplet $(M,\sfd,\frm)$ and it yields many important consequences, both analytical and geometric. First of all, the heat kernel $\hp_t = \frac{\De\rmR_{0,t}}{\De(\frm\otimes\frm)}$ enjoys a Gaussian lower bound. More precisely, under the $\CD(\kappa,N)$ condition with $N<\infty$, from \cite[Theorem 1.1 and Theorem 1.2]{JLZ14} we know that for every $\delta > 0$ there exist $C_1,C_2$ depending only on $\kappa,N,\delta$ such that
\begin{equation}\label{eq:gaussian-estimates}
\hp_t(x,y) \geq \frac{1}{C_1\frm(B_{\sqrt{t}}(x))}\exp\Big(-\frac{\sfd^2(x,y)}{(4-\delta)t}-C_2t\Big)
\end{equation}
for all $x,y \in M$ and all $t>0$. On the other hand, if we assume that $(M,\sfd,\frm)$ satisfies $\CD(\kappa,\oo)$ with $\frm(M)=1$, then by \cite[Corollary 1.3]{Wang11} it holds
\begin{equation}\label{eq:gaussian-lower}
\hp_t(x,y) \geq \exp\Big(-\frac{\kappa\sfd^2(x,y)}{2(1-e^{-\kappa t})}\Big), \qquad \forall x,y \in M,\, \forall t > 0.
\end{equation}
As a further consequence of the $\CD(\kappa,\infty)$ condition, the semigroup $P_t$ satisfies the $L^\infty$-Lipschitz regularization \cite[Theorem 6.8]{AGS14b}
\begin{equation}\label{eq:linfty-lip-reg}
\sqrt{2E_{2\kappa}(t)}{\rm Lip}(P_t u) \leq \|u\|_{L^\infty(\frm)}, \qquad \forall t > 0,\, u \in L^\infty(\frm),
\end{equation}
where $E_{2\kappa}$ is defined as in \eqref{eq:curvature-factor}. Let us recall that also Hamilton's gradient estimate \cite{Hamilton93, Kotschwar07} (see also \cite{JiangZhang16}) holds true under $\CD(\kappa,\infty)$: for any positive $u \in L^p \cap L^\infty(\frm)$ for some $p \in [1,\infty)$ it holds
\begin{equation}\label{eq:hamilton}
t|\nabla\log P_t u|^2 \leq (1+2\kappa^-t)\log\Big(\frac{\|u\|_{L^\infty(\frm)}}{P_t u}\Big).
\end{equation}
Finally, let us recall that the $\CD(\kappa,N)$ condition with $N<\infty$ entails the Bishop-Gromov inequality (see for instance \cite[Theorem 2.3]{Sturm06II}): for any $x \in \supp(\frm) = M$ and $0 < r \leq R \leq \pi\sqrt{(N-1)/\kappa^+}$, where $(N-1)/\kappa^+=+\infty$ if $\kappa \leq 0$, it holds
\begin{equation}\label{eq:bishop-gromov}
\frac{\frm(B_r(x))}{\frm(B_R(x))} \geq 
\left\{\begin{array}{ll}
\displaystyle{\frac{\int_0^r \sin(t\sqrt{\kappa/(N-1)})^{N-1}\De t}{\int_0^R \sin(t\sqrt{\kappa/(N-1)})^{N-1}\De t}}\,, &\qquad \text{if } \kappa > 0   \,, \\
\\
\bigg(\displaystyle{\frac{r}{R}}\bigg)^N\,, &\qquad \text{if } \kappa = 0   \,, \\
\\
\displaystyle{\frac{\int_0^r \sinh(t\sqrt{-\kappa/(N-1)})^{N-1}\De t}{\int_0^R \sinh(t\sqrt{-\kappa/(N-1)})^{N-1}\De t}}\,, &\qquad \text{if } \kappa < 0\,.
\end{array}\right.
\end{equation}
As a consequence of \eqref{eq:bishop-gromov}, in the following lemma we point out a technical integrability property of the reference measure $\frm$ that in Proposition \ref{prop:fg} will guarantee the existence and uniqueness of the solution to SP as well as the existence of the $fg$-decomposition.

\begin{lemma}\label{lemma:log:balls}
Let $(M,\sfd,\frm)$ satisfy $\CD(\kappa, N)$ with $N<+\oo$. Then for any $r>0$ it holds
\[
\left(\log\frm\left(B_r(\cdot)\right)\right)^+\in L^1(\mu),\quad\forall\mu\in\cP_1(M)\,.
\]
\end{lemma}

\begin{proof}
Assume $\kappa < 0$. If we fix $\bar{x} \in M$ and observe that $B_r(x) \subset B_{\sfd(x,\bar{x})+r}(\bar{x})$, then by \eqref{eq:bishop-gromov}
\[
\frm(B_r(x)) \leq C_1\frm(B_r(\bar{x}))\int_0^{\sfd(x,\bar{x})+r}\sinh(t\sqrt{-\kappa/(N-1)})^{N-1}\De t \leq C_2\,e^{C_3\sfd(x,\bar{x})},
\]
where $C_1,C_3$ only depend on $\kappa,N,r$ and $C_2$ also depends on $\bar{x}$. As a consequence,
\begin{equation}\label{remappo2}
\log\frm(B_r(x)) \leq \log C_2 + C_3\sfd(x,\bar{x})
\end{equation}
and therefore $\left(\log\frm\left(B_r(\cdot)\right)\right)^+\in L^1(\mu)$ for any $\mu\in\cP_1(M)$.  
In the case $\kappa \geq 0$, the argument works verbatim, the only difference being in the application of the Bishop-Gromov inequality.
\end{proof}

At this stage, let us also recall the definition of the relative entropy when the reference $\frm$ is not a probability measure (see \cite[Section 2]{conforti2021formula} and \cite[Appendix A]{LeoSch} for more details). As $\frm$ is $\sigma$-finite, there exists a measurable function $W\colon M\to [0,+\oo)$ such that
\[z_W\coloneqq \int e^{-W}\,\De\frm<+\oo\,.\]
Then, for any $\frp\in\cP(M)$ such that $W\in L^1(\frp)$ we can define the relative entropy with respect to $\frm$ as
\begin{equation}\label{def:entropy}
\cH(\frp|\frm)\coloneqq \cH(\frp|\frm_W)-\int W\,\De\frp-\log z_W\,,
\end{equation}
where $\frm_W\coloneqq z_W^{-1}\,e^{-W}\,\frm\in\cP(M)$ and $\cH(\frp|\frm_W)$ is defined via the standard definition of relative entropy between probabilities, given in the introduction. Moreover, it is immediate to see that the above definition is independent of the choice of the measurable function $W\colon M\to [0,+\oo)$ satisfying $z_W<+\oo$. In particular, thanks to \cite[Theorem 4.24]{Sturm06I}, $W$ can be chosen as $C\sfd^2(\cdot,x_0)$ for any $x_0 \in M$ and $C>0$ sufficiently large, whence the lower semicontinuity of $\mathcal{H}(\cdot|\frm)$ w.r.t.\ $W_2$-convergence. 
We are now ready to prove the main result of this subsection, about the existence of an $fg$-decomposition as in \eqref{eq:fg-decomposition} for the unique minimizer in \eqref{SP}.

\begin{prop}\label{prop:fg}
Let $(M,\sfd,\frm)$ satisfy \eqref{CD} and suppose the marginals $\mu,\,\nu$ satisfy \eqref{H1}. Then \eqref{SP} admits a unique minimizer $\schrplan{\mu}{\nu}{T}\in\Pi(\mu,\nu)$ and there are two non-negative measurable functions $f,\,g$ satisfying \eqref{eq:fg-decomposition}, i.e.,
\[
\frac{\De\schrplan{\mu}{\nu}{T}}{\De\rmR_{0,T}}(x,y)=f(x)\,g(y)\quad\rmR_{0,T}\text{-a.s.}
\]
They are $\frm$-a.e.\ unique up to the trivial transformation $(f,g) \mapsto (cf,g/c)$ for some $c>0$
and moreover $\varphi=\log f\in L^1(\mu)$, $\psi=\log g \in L^1(\nu)$. Finally, $f$ and $g$ satisfy the Schr\"odinger system
\begin{equation}\label{schr:system}\begin{cases}
\mu=f\,P_T g\,\frm\,,\\
\nu=P_Tf\, g\,\frm\,.
\end{cases}\end{equation}
\end{prop}

Since the $fg$-decomposition is unique up to a scalar multiplicative factor, through the whole paper we are going to assume the following symmetric normalization
\begin{equation}\label{gia:renormalization}
\begin{aligned}\int\log f\,\De\mu-\cH(\mu|\frm)=&\int\log g\,\De\nu-\cH(\nu|\frm)\\
 =&\frac12\left[\cC_T(\mu,\nu)-\cH(\mu|\frm)-\cH(\nu|\frm)\right]=\frac{\cS_T(\mu,\nu)}{2T}\,.
\end{aligned}
\end{equation}

\begin{proof}
The existence and uniqueness of $\schrplan{\mu}{\nu}{T}\in\Pi(\mu,\nu)$ and $f,g$, as well as the validity of the Schr\"odinger system \eqref{schr:system} follow from \cite[Proposition 2.1]{GigTam21} once we establish that $\cH(\mu\otimes\nu|\rmR_{0,T})$ is finite.
In order to prove that, let us derive a lower bound for $\log \hp_T$.
First, assume that $(M,\,\sfd,\,\frm)$ satisfies $\CD(\kappa,\oo)$ for some $\kappa\in\R$ and $\frm(M)=1$. From \eqref{eq:gaussian-lower} we get the lower bound 
\begin{equation}\label{eq:gaussian-lower-bound}
\log \hp_T(x,y)\geq -\frac{\kappa \,\sfd^2(x,y)}{2(e^{\kappa \,T}-1)}\,,
\end{equation}
which, combined with the trivial inequalities $\sfd(x,y)^2\leq 2(\sfd^2(x,z)+\sfd^2(y,z))$ valid for any $z\in M$ and $e^{\kappa\,T}-1\geq \kappa \,T$, implies
\begin{equation}\label{infinity:log:bound}
    \log \hp_T(x,y)\geq -\frac{\sfd^2(x,z)}{T}-\frac{\sfd^2(y,z)}{T}\,.
\end{equation}
On the other hand, if $\CD(\kappa,N)$ holds with $N<+\oo$, then from the heat kernel lower bound \eqref{eq:gaussian-estimates} we obtain
 \begin{equation}\label{N:log:bound}
     \log \hp_T(x,y)\geq -\log\left[C_1\, \frm\left(B_{\sqrt{T}}(x)\right)\right]-C_2\,T-\frac{\sfd^2(x,z)}{T}-\frac{\sfd^2(y,z)}{T}\,.
 \end{equation}
In both cases, thanks to Assumption \eqref{H1} and Lemma \ref{lemma:log:balls} with $r=\sqrt{T}$ we conclude that $-\int\log \hp_T\,\De(\mu\otimes\nu)\in[-\oo,+\oo)$, which leads to the well-defined summation
 \begin{equation}\label{remappo}\cH(\mu\otimes\nu|\rmR_{0,T})=\cH(\mu|\frm)+\cH(\nu|\frm)-\int\log \hp_T\,\De(\mu\otimes\nu)<+\oo\,.\end{equation}
 
Finally, by arguing as in \eqref{remappo}, we know that $\cH(\mu\otimes\nu|\rmR_{0,T}^\mu)$ is also finite, where $\rmR_{0,T}^\mu\in\cP(\mathbb{R}^{2d})$ is the probability measure defined as $\rmR_{0,T}^\mu(\De x,\De y)\coloneqq \hp_T(x,y)\mu(\De x)\frm(\De y)$. Since for any $\pi\in\Pi(\mu,\nu)$ it holds
\[\cH(\pi|\rmR_{0,T})=\cH(\mu|\frm)+\cH(\pi|\rmR_{0,T}^\mu)\,,\]
clearly $\schrplan{\mu}{\nu}{T}$ minimizes $\cH(\cdot|\rmR_{0,T}^\mu)$ as well, over the set of couplings $\Pi(\mu,\nu)$. Given the above premises, the integrability of $\varphi$ and $\psi$ follows from \cite[Corollary 1.13 and Remark 2.22]{Marcel:notes} applied to the independent coupling $\mu\otimes\nu$ and to the probability reference $\rmR_{0,T}^\mu$.
\end{proof}

\subsection{Corrector estimates}\label{sec:corr:est}

We are going to prove the corrector estimates~\eqref{corr:est} via an approximation argument. First we will assume that the marginals have compact support and bounded densities, then we will extend it to the case in which $f,g\in L^\oo(\frm)$, and finally to the case when the second condition in \eqref{H2} is satisfied by relying on a finite-dimensional approximation argument where the second marginal is fixed while the first marginal constraint is replaced by a finite-dimensional one. More precisely, since $(M,\sfd)$ is separable we know that there exists a countable dense family of bounded measurable functions $\{\phi_i\}_{i\in\N}$ such that
\[
(\mathrm{proj}_{x_1})_{\#}\pi=\mu \quad \Leftrightarrow \quad \int \phi_i(x)\De\pi(x,y)=\int \phi_i\De\mu \qquad\forall\,i\in\N\,.
\]
Therefore for any fixed $K\in\N$ we may define the convex and closed (in total variation) set
\[\cQ^\nu_K\coloneqq \biggl\{\pi\in\cP(M\times M)\,\colon (\mathrm{proj}_{x_2})_{\#}\pi=\nu\text{ and  } \int \phi_i(x)\De\pi(x,y)=\int\phi_i\De\mu\quad\forall \,i=1,\dots, K  \biggr\}\]
and then consider the associated minimization problem
\begin{equation}\label{half:schr}
\inf_{\pi\in\cQ^\nu_K}\cH(\pi|\rmR_{0,T})\,.
\end{equation}
The next lemma, whose proof is postponed to the end of this section, shows that the above problem gives an approximation of SP.

\begin{lemma}\label{lemma:half:schr}
Let $(M,\sfd, \frm)$ satisfy \eqref{CD} and suppose that the marginals $\mu,\,\nu$ satisfy \eqref{H1}. Then the above minimiziation problem \eqref{half:schr} admits a unique optimizer $\pi^K\in \cQ^\nu_K$, whose density is given by
\begin{equation}\label{pi:K:fg}
\frac{\De\pi^K}{\De\rmR_{0,T}}(x,y)=f_K(x)\,g_K(y)\qquad\text{with }\,f_K=C\,\exp\left(\sum_{i=1}^K\lambda_i\,\phi_i\right)\in L^\oo(\frm)
\end{equation}
for some constant $C>0$ and  multipliers  $\lambda_i\in\R$.
Moreover, $g_K$ converges $\frm$-a.e.\ to $g$ and $ P_T f_K$ converges $\frm$-a.e.\ to $ P_T f$ on $\supp(\nu)$. Finally, the optimal value $\cH(\pi^K|\rmR_{0,T})$ converges to $\cC_T(\mu,\nu)$ as $K \to \infty$.
\end{lemma}

\begin{prop}\label{lemma:gia}
Let $(M,\sfd, \frm)$ satisfy \eqref{CD}, suppose the marginals $\mu,\,\nu$ satisfy \eqref{H1} and let $f,\,g$ be as in Proposition \ref{prop:fg} and satisfying the normalization \eqref{gia:renormalization}. Then, if $\frac{\De\nu}{\De\frm}$ is locally bounded away from $0$ on ${\rm int}(\supp(\nu))$ and $\nu(\partial\supp(\nu))=0$ it holds
\begin{equation}\label{eq:corrector:nu}
  \norm{\nabla  \log P_T f}_{L^2(\nu)}^2\leq \frac{1}{E_{2\kappa}(T)}\biggl[\cC_T(\mu,\nu)-\cH(\nu|\frm)\biggr]\,,
\end{equation}
where $E_{2\kappa}$ is defined as in \eqref{eq:curvature-factor}. In particular, it is part of the statement the fact that $\log P_T f \in W^{1,2}_{loc}({\rm int}(\supp(\nu)))$ with $|\nabla\log P_T f| \in L^2(\nu)$.

Similarly, if $\frac{\De\mu}{\De\frm}$ is locally bounded away from $0$ on ${\rm int}(\supp(\mu))$ and $\mu(\partial\supp(\mu))=0$, then $\log P_T g \in W^{1,2}_{loc}({\rm int}(\supp(\mu)))$ with $|\nabla\log P_T g| \in L^2(\mu)$ and it holds
\begin{equation}\label{eq:corrector:mu}
\norm{\nabla\log P_T g}_{L^2(\mu)}^2\leq \frac{1}{E_{2\kappa}(T)}\biggl[\cC_T(\mu,\nu)-\cH(\mu|\frm)\biggr]\,.
\end{equation}
In particular, both \eqref{eq:corrector:nu} and \eqref{eq:corrector:mu} hold true if the marginals satisfy \eqref{H2}.
\end{prop}

\begin{proof}
We will only prove inequality \eqref{eq:corrector:nu}, as~\eqref{eq:corrector:mu} follows by completely analogous techniques.

As a preliminary step, let us assume that $\mu,\nu$ have bounded supports and bounded densities $\frac{\De\mu}{\De\frm},\,\frac{\De\nu}{\De\frm}$. Under this extra hypothesis, it holds
\begin{equation}\label{eq:gronwall}
\int\abs{\nabla\log P_Tf}^2\De\nu \leq e^{-2\kappa(T- t)}\int\abs{\nabla\log P_{t}f}^2\De\entint{\mu}{\nu,T}_t,
\end{equation}
where $\entint{\mu}{\nu,T}_t = P_t f\,P_{T-t}g\,\frm$ is the $T$-entropic interpolation from $\mu$ to $\nu$ at time $t$. If $(M,\sfd,\frm)$ satisfies $\CD(\kappa,N)$ with $N<\infty$, this is a consequence of Gr\"onwall lemma applied to the functions $\alpha(t) := \int\abs{\nabla\log P_{t}f}^2\De\entint{\mu}{\nu,T}_t$ and $\beta(t) := e^{2\kappa(T-t)}\int\abs{\nabla\log P_Tf}^2\De\nu$, as on the one hand by \cite[Lemmas 3.6 and 3.7]{conforti2019second} together with \cite[Proposition 4.8]{GigTam21} (which justifies the computations of \cite[Lemmas 3.6 and 3.7]{conforti2019second} in the non-compact and possibly negatively curved setting) we have
\begin{equation}\label{eq:gronwall2}
\alpha \in C^1((0,T]) \qquad \textrm{and} \qquad \alpha'(t) \leq -2\kappa\alpha(t),\,\forall t \in (0,T]
\end{equation}
while on the other hand it is readily verified that $\beta' = -2\kappa\beta$; since $\alpha(T)=\beta(T)$, it must hold $\alpha(t) \geq \beta(t)$ for all $t \in (0,T]$, namely \eqref{eq:gronwall}.

If instead $(M,\sfd,\frm)$ satisfies $\CD(\kappa,\infty)$ and $\frm(M)=1$, the reader should refer to \cite[Lemma 2.2]{conforti2021formula}: this grants the validity of \eqref{eq:gronwall2} for
\[
\alpha_\delta(t) := c_\delta\int |\nabla\log (P_t f+\delta)|^2\,(P_tf + \delta)(P_{T-t}g + \delta)\,\De\frm,
\]
where $\delta>0$ is any positive number and $c_\delta$ is a normalization constant, so that $c_\delta(P_tf + \delta)(P_{T-t}g + \delta)\frm$ is a probability measure. By considering $\beta_\delta(t) := e^{2\kappa(T-t)}\int\abs{\nabla\log (P_Tf+\delta)}^2(P_Tf + \delta)g\,\De\frm$, by the same argument as before we obtain $\alpha_\delta(t) \geq \beta_\delta(t)$ for all $t \in (0,T]$ and $\delta>0$, whence \eqref{eq:gronwall} by passing to the limit as $\delta \downarrow 0$ by the Dominated Convergence Theorem. Indeed, note first that
\[
\begin{split}
|\nabla\log (P_t f+\delta)|^2\,(P_tf + \delta)(P_{T-t}g + \delta) & = \frac{|\nabla P_t f|^2}{P_t f +\delta}(P_{T-t}g+\delta) \leq |\nabla\log P_t f|^2 P_t f(P_{T-t}g+\delta) \\
& \leq C(\kappa,t)(P_t f)|\log P_t f|(P_{T-t}g+1) 
\end{split}
\]
where the last inequality is due to Hamilton's gradient estimate \eqref{eq:hamilton}, $C(\kappa,t)$ being some constant that only depends on $\kappa,t$. Then, thanks to the standing assumptions on $\mu,\nu$ and the Gaussian heat kernel lower bound \eqref{eq:gaussian-lower-bound}, by \cite[Proposition 2.1]{GigTam21} we know that $f,g \in L^\infty(\frm)$. By the maximum principle this implies that the previous right-hand side is bounded, hence integrable as $\frm(M)=1$, and thus provides an admissible dominating function.

Now, by multiplying by $e^{2\kappa(T-t)}$ and integrating over $t\in[0,T]$ in \eqref{eq:gronwall} we deduce that
\[
\begin{split}
E_{2\kappa}(T)\int\abs{\nabla\log P_Tf}^2\De\nu 
& \leq \int_0^T\int \abs{\nabla\log P_{t}f}^2 \, \De\entint{\mu}{\nu,T}_t \De t = \cC_T(\mu,\nu)-\cH(\nu|\frm)\,,
\end{split}
\]
where the identity is justified by the following  representation of the entropic cost
\begin{equation}\label{eq:bbs}
\cC_T(\mu,\nu) = \cH(\nu|\frm) + \int_0^T\int \abs{\nabla\log P_{t}f}^2\,\De\entint{\mu}{\nu,T}_t \De t\,
\end{equation}
 (see \cite[Proposition 4.1]{GigTam19} paying attention to the different rescaling and Remark \ref{rmk:bbs} below).
 Inequality \eqref{eq:corrector:nu} is thus proved, as well as the last part of the statement.

Now let us remove the additional assumptions on $\mu,\nu$ in a two-steps procedure. 

\medskip

\noindent\textbf{1st step.} In addition to \eqref{H1}, assume that $\mu,\,\nu$ are such that the associated $f,g \in L^\infty(\frm)$. Then fix $\bar{x} \in M$ and introduce $f_n \coloneqq \mathds{1}_{B_n(\bar{x})}f$, $g_n \coloneqq \mathds{1}_{B_n(\bar{x})}g$, so that \eqref{eq:corrector:nu} holds true for
\[
\mu_n \coloneqq c_n f_n P_T g_n\,\frm \qquad \textrm{and} \qquad \nu_n \coloneqq c_n g_n P_T f_n\,\frm,
\]
where $c_n$ is a normalization constant (note that by self-adjointness of $P_T$ it is the same for both measures $\mu_n$ and $\nu_n$). Namely
\[
\int\abs{\nabla\log P_Tf_n}^2\De\nu_n \leq \frac{1}{E_{2\kappa}(T)}\,\biggl[\cC_T(\mu_n,\nu_n)-\cH(\nu_n|\frm)\biggr]\,.
\]
Observing that $\abs{\nabla\log P_Tf_n} \geq \abs{\nabla\log(P_Tf_n+\delta)}$ for any $\delta>0$, the inequality above implies in particular that
\begin{equation}\label{eq:approx-corrector}
\int\abs{\nabla \log(P_Tf_n+\delta)}^2\De\nu_n \leq \frac{1}{E_{2\kappa}(T)}\,\biggl[\cC_T(\mu_n,\nu_n)-\cH(\nu_n|\frm)\biggr]
\end{equation}
for all $\delta>0$. Let us now pass to the limit as $n \to \infty$. 

To this end, observe first that $P_t f_n \to P_t f$ and $P_T g_n \to P_T g$ pointwise as $n \to \infty$. Indeed
\[
\begin{split}
|P_T f_n(x)-P_T f(x)| & = \bigg|\int\big(\mathds{1}_{B_n(\bar{x})}(y)-1\big)f(y)\hp_T(x,y)\,\De\frm(y)\bigg| \\
& \leq \|f\|_{L^\infty}\int|\mathds{1}_{B_n(\bar{x})}(y)-1|\hp_T(x,y)\,\De\frm(y)
\end{split}
\]
and the right-hand side vanishes as $n \to \infty$ by dominated convergence. This allows to handle the right-hand side of \eqref{eq:approx-corrector} in the following way. As concerns the term $\cC_T(\mu_n,\nu_n)$, on the one hand, 
\[
\begin{split}
\int \log f_n\,\De\mu_n \leq \int \log f\,\De\mu_n & = \int(\log f)^+\,\De\mu_n - \int(\log f)^-\,\De\mu_n \\
& \leq c_n\int(\log f)^+\,\De\mu - \int(\log f)^-\,\De\mu_n,
\end{split}
\]
where we have used $f_n \leq f$ and, as a consequence of the maximum principle, also $P_T g_n \leq P_T g$. This implies that
\[
\begin{split}
\limsup_{n \to \infty}\int \log f_n\,\De\mu_n 
& \leq \int(\log f)^+\,\De\mu - \liminf_{n \to \infty}\int(\log f)^-\,\De\mu_n \\
& \leq \int(\log f)^+\,\De\mu - \int(\log f)^-\,\De\mu = \int\log f\,\De\mu,
\end{split}
\]
since $c_n \to 1$ and $f_n\,P_T g_n \to f\,P_T g$ pointwise (as already discussed before) together with Fatou lemma. On the other hand, one can show in an analogous fashion that
\[
\limsup_{n \to \infty}\int \log g_n\,\De\nu_n \leq \int \log g\,\De\nu.
\]
Since
\[
\cC_T(\mu_n,\nu_n) = \int\log f_n\,\De\mu_n + \int\log g_n\,\De\nu_n + \log c_n,
\]
we conclude that
\begin{equation}\label{eq:limsup-cost}
\limsup_{n \to \infty} \cC_T(\mu_n,\nu_n) \leq \cC_T(\mu,\nu).
\end{equation}
Secondly, as regards $\cH(\nu_n|\frm)$, note that $g_n P_T f_n \to g P_T f$ $\frm$-a.e.\ together with $g_n P_T f_n \leq g P_T f \in L^1(\frm)$ entails $g_n P_T f_n \to g P_T f$ in $L^1(\frm)$ by dominated convergence, whence $\nu_n \rightharpoonup \nu$. Moreover,
\[
\begin{split}
|M_2(\nu_n)-M_2(\nu)| & \leq \bigg|\int\sfd^2(\cdot,\bar{x})\,\De\nu_n - \int\sfd^2(\cdot,\bar{x})g_n P_T f_n\,\De\frm\bigg| \\
& \qquad + \bigg|\int\sfd^2(\cdot,\bar{x})g_n P_T f_n\,\De\frm - \int\sfd^2(\cdot,\bar{x})\,\De\nu\bigg| \\
& \leq (c_n-1)M_2(\nu) + \int\sfd^2(\cdot,\bar{x})\bigg|\frac{g_n P_T f_n}{g P_T f}-1\bigg|\,\De\nu \to 0, \quad \textrm{as } n \to \infty,
\end{split}
\]
again by $c_n \to 1$ and dominated convergence (recall indeed that $\nu \in \cP_2(M)$ by \eqref{H1} and $g_n P_T f_n \leq g P_T f$, so that $\frac{g_n P_T f_n}{g P_T f} \in [0,1]$). Therefore $\cW_2(\nu_n,\nu) \to 0$ and, by lower semicontinuity of the entropy w.r.t.\ $\cW_2$-convergence, 
\begin{equation}\label{eq:limsup-entropy}
\limsup_{n \to \infty} \Big(- \cH(\nu_n|\frm)\Big) \leq - \cH(\nu|\frm).
\end{equation}
Passing to the left-hand side of \eqref{eq:approx-corrector}, fix $k \in \N$ and note that for all $n \geq k$ we have
\[
\int\abs{\nabla \log(P_Tf_n+\delta)}^2\De\nu_n \geq \frac{1}{c_k}\int\abs{\nabla \log(P_Tf_n+\delta)}^2\De\nu_k = \frac{1}{c_k}\int_{\bar B_k(\bar{x})}\abs{\nabla \log(P_Tf_n+\delta)}^2\De\nu_k,
\]
since $c_n \geq 1$. We now claim that $|\nabla\log(P_Tf_n+\delta)| \rightharpoonup G$ in $L^2(\bar B_k(\bar{x}),\nu_k)$ for some $G$ such that $|\nabla\log(P_Tf+\delta)| \leq G$. To this end, observe that by the $L^\infty$-Lipschitz regularization \eqref{eq:linfty-lip-reg} it holds  $\abs{\nabla P_Tf_n} \leq C_{T,\kappa}\|f_n\|_{L^\infty(\frm)}$, where $C_{T,\kappa}$ only depends on $T$ and $\kappa$ given by \eqref{CD}. Since $f_n \leq f$, this implies
\[
\int_{\bar B_k(\bar{x})}\abs{\nabla \log(P_Tf_n+\delta)}^2\,\De\frm \leq \frac{C_{T,\kappa}^2}{\delta^2}\frm(B_k(\bar{x}))\|f\|^2_{L^\infty(\frm)}.
\]
The functions $(|\nabla\log(P_Tf_n+\delta)|)_{n \in \N}$ are thus equi-bounded in $L^2(\bar B_k(\bar{x}),\frm)$ and this implies that, up to subsequences, they converge to some function $G \in L^2(\bar B_k(\bar{x}),\frm)$ weakly in $L^2(\bar B_k(\bar{x}),\frm)$. Since $\log(P_Tf_n+\delta)$ converges $\frm$-a.e.\ to $\log(P_Tf+\delta)$, by \cite[Lemma 4.3(b)]{AGS14} $|\nabla\log(P_Tf+\delta)| \leq G$. As for $|\nabla\log(P_Tf_n+\delta)| \rightharpoonup G$ in $L^2(\bar B_k(\bar{x}),\nu_k)$, this directly follows from the weak convergence in $L^2(\bar B_k(\bar{x}),\frm)$ and the fact that, under the current assumptions, $\frac{\De\nu}{\De\frm} \in L^\infty(\frm)$. Combining the claim with the lower semicontinuity of the $L^2(\nu_k)$-norm w.r.t.\ weak convergence we obtain
\[
\liminf_{n \to \infty}\int_{\bar B_k(\bar{x})}\abs{\nabla \log(P_Tf_n+\delta)}^2\De\nu_k\, \geq \int_{\bar B_k(\bar{x})}G^2\De\nu_k \geq \int_{\bar B_k(\bar{x})}\abs{\nabla \log(P_Tf+\delta)}^2\De\nu_k.
\]
From this inequality, \eqref{eq:limsup-cost} and \eqref{eq:limsup-entropy} we end up with
\[
\frac{1}{c_k}\int_{\bar B_k(\bar{x})}\abs{\nabla \log(P_Tf+\delta)}^2\De\nu_k \leq \frac{1}{E_{2\kappa}(T)}\,\biggl[\cC_T(\mu,\nu)-\cH(\nu|\frm)\biggr]
\]
and it is now sufficient to let first $k \to \infty$ and then $\delta \downarrow 0$. In both cases, the monotone convergence theorem (and the fact that $c_k \to 1$) allows to handle the left-hand side and finally get the validity of \eqref{eq:corrector:nu} for $\mu,\nu$.

\medskip

\noindent\textbf{2nd step.} Now let $\mu,\nu$ be as in \eqref{H1} and assume that $\frac{\De\nu}{\De\frm}$ is locally bounded away from $0$ on ${\rm int}(\supp(\nu))$ and $\nu(\partial\supp(\nu))=0$. Fix $K\in\N$, let $f_K,g_K$ be defined as in Lemma \ref{lemma:half:schr} and let $\mu_K\coloneqq (\mathrm{proj}_{x_1})_{\#}\pi^K$ be the first marginal of the optimizer $\pi^K$ associated to \eqref{half:schr}. Since this approximation guarantees us only that $f_K\in L^\oo(\frm)$, let us fix $n\in\N$ and define $g_K^n\coloneqq\min\{g_K,n\}\in L^\oo(\frm)$ so that the previous step applies to the marginals
\[
\mu^K_n \coloneqq c_{K,n} f_K P_T g_K^n\,\frm \qquad \textrm{and} \qquad \nu^K_n \coloneqq c_{K,n} g_K^n P_T f_K\,\frm,
\]
where $c_{K,n}$ is a normalization constant (again, this is the same for both $\mu_n^K$ and $\nu_n^K$). Henceforth, in this case \eqref{eq:corrector:nu} reads as
\[
\int\abs{\nabla\log P_T f_K}^2\,\De\nu_n^K\leq \frac{1}{E_{2\kappa}(T)}\biggl[\cC_T(\mu^K_n,\nu^K_n)-\cH(\nu_n^K|\frm)\biggr]\,.
\]
Owing to algebraic manipulations, the normalizing constant $c_{K,n}$ can be neglected and therefore it holds
\begin{align*}
\int\abs{\nabla\log P_T f_K}^2\,g^n_K P_Tf_K\,\De\frm\leq \frac{1}{E_{2\kappa}(T)}\biggl[\int f_K\log f_K\,P_T g^n_K\,\De\frm-\int P_T f_K\log (P_T f_K)\,g_K^n\,\De\frm\biggr]\,.
\end{align*}
Since by Lemma \ref{lemma:half:schr} $f_K$ is bounded away from 0 and $\infty$, it follows that $\log f_K\in L^1(\mu^K)$ and $\log(P_T f_K)\in L^1(\nu)$. From this, by applying Fatou lemma to the left-hand side and the Dominated Convergence Theorem to the right-hand one, in the limit as $n\to\oo$ we have
\begin{equation*}
    \int\abs{\nabla\log P_T f_K}^2\,g_K P_Tf_K\,\De\frm
\leq \frac{1}{E_{2\kappa}(T)}\biggl[\int f_K\log f_K\,P_T g_K\,\De\frm-\int P_T f_K\log( P_T f_K)\,g_K\,\De\frm\biggr]\,,
\end{equation*}
which is the corrector estimate \eqref{eq:corrector:nu} associated to $\cC_T(\mu^K,\nu)$:
\begin{equation}\label{eq:seattle}
    \int\abs{\nabla\log P_T f_K}^2\,\De\nu
\leq \frac{1}{E_{2\kappa}(T)}\biggl[\cC_T(\mu^K,\nu)-\cH(\nu|\frm)\biggr]\,.
\end{equation}
We wish now to take the limit as $K\to\oo$. However this point is more subtle, since we are interested in proving that $\nabla\log P_T f$ is well defined also in the limit. At this stage, let us stress out that the extra-assumption on the density $\frac{\De\nu}{\De\frm}$ and on $\supp(\nu)$ has never been used and it is just needed in the following discussion. First, notice that the right-hand side above is uniformly bounded in $K\in\N$ since Lemma \ref{lemma:half:schr} ensures that $\cC_T(\mu^K,\nu)=\cH(\pi^K|\rmR_{0,T})$ converges to $\cC_T(\mu,\nu)$. 
As concerns the left-hand side, fix $\bar{x} \in \mathrm{int} (\supp(\nu))$. By assumption there exist $r,\alpha>0$ such that $\bar B_r(\bar{x}) \subset \supp(\nu)$ and $\frac{\De\nu}{\De\frm} \geq \alpha$ $\frm$-a.e.\ in $\bar B_r(\bar{x})$, so that, for any fixed $\delta>0$ it holds
\[
\int\abs{\nabla\log P_Tf_K}^2\De\nu \geq \int\abs{\nabla \log(P_Tf_K+\delta)}^2\De\nu \geq \alpha \int_{\bar B_r(\bar{x})}\abs{\nabla \log(P_Tf_K+\delta)}^2\,\De\frm\,.
\]
As a byproduct of these bounds we have
\[
\limsup_{K \to \infty} \int_{\bar B_r(\bar{x})}\abs{\nabla \log(P_Tf_K+\delta)}^2\,\De\frm \leq \alpha^{-1}\,E_{2\kappa}(T)^{-1}\,\biggl[\cC_T(\mu,\nu)-\cH(\nu|\frm)\biggr] < +\infty\,.
\]
The functions $(|\nabla\log(P_Tf_K+\delta)|)_{K \in \N}$ are thus equi-bounded in $L^2(\bar B_r(\bar{x}),\frm)$ and this implies that, up to subsequences, they converge to some function $G_{\bar{x}} \in L^2(\bar B_r(\bar{x}),\frm)$ weakly in $L^2(\bar B_r(\bar{x}),\frm)$. Moreover, from Lemma \ref{lemma:half:schr} we also have that $\log(P_Tf_K+\delta)$ converges $\frm$-a.e.\ to $\log(P_Tf+\delta)$ in $\bar B_r(\bar{x})$. Therefore, relying again on \cite[Lemma 4.3(b)]{AGS14} we conclude that $|\nabla\log(P_Tf+\delta)| \leq G_{\bar{x}}$ on $\bar B_r(\bar{x})$. In particular, it is worth stressing that in this case \cite[Lemma 4.3(b)]{AGS14} also ensures that $\log(P_Tf+\delta) \in W^{1,2}(\bar B_r(\bar{x}))$, which does not follow from the regularizing effect of $P_T$, because of the possible lack of integrability of $f$.

To replicate the proof given in the previous step we also need that $|\nabla\log(P_Tf_K+\delta)| \rightharpoonup G_{\bar{x}}$ in $L^2(\bar B_r(\bar{x}),\nu)$, but this may fail, as $\frac{\De\nu}{\De\frm}$ needs not belong to $L^\infty(\frm)$. For this reason, let us introduce $[\nu]_N \coloneqq \min\{g_K P_T f_K,N\}\frm$ for $N \in \N$, so that the left-hand side of \eqref{eq:seattle} can be trivially estimated from below as
\[
\int\abs{\nabla\log P_T f_K}^2\,\De[\nu]_N
\leq \int\abs{\nabla\log P_T f_K}^2\,\De\nu \leq \frac{1}{E_{2\kappa}(T)}\biggl[\cC_T(\mu^K,\nu)-\cH(\nu|\frm)\biggr]\,.
\]
Since now $\frac{\De[\nu]_N}{\De\frm} \in L^\infty(\frm)$ by construction, $|\nabla\log(P_Tf_K+\delta)| \rightharpoonup G_{\bar{x}}$ in $L^2(\bar B_r(\bar{x}),\frm)$ implies the weak convergence towards the same limit in $L^2(\bar B_r(\bar{x}),[\nu]_N)$. Combining these considerations with the lower semicontinuity of the $L^2([\nu]_N)$-norm, we obtain
\[
\liminf_{K \to \infty}\int_{\bar B_r(\bar{x})}\abs{\nabla \log(P_Tf_K+\delta)}^2\De[\nu]_N\, \geq \int_{\bar B_r(\bar{x})}G^2_{\bar{x}}\De[\nu]_N \geq \int_{\bar B_r(\bar{x})}\abs{\nabla \log(P_Tf+\delta)}^2\De[\nu]_N
\]
whence
\[
\int_{\bar B_r(\bar{x})}\abs{\nabla \log(P_Tf+\delta)}^2\De[\nu]_N \leq \frac{1}{E_{2\kappa}(T)}\,\biggl[\cC_T(\mu,\nu)-\cH(\nu|\frm)\biggr]\,.
\]
Choosing now $(x_k)_{k \in \N} \subset {\rm int}(\supp(\nu))$ dense and denoting by $r_k$ the radii associated to each $x_k$ according to the previous construction, so that ${\rm int}(\supp(\nu)) = \cup_k \bar B_{r_k}(x_k)$,  by a diagonal argument there exists $G \in L^2_{loc}(\frm)$ such that $|\nabla\log(P_Tf_K+\delta)| \rightharpoonup G$ in $L^2(\bar B_{r_k}(x_k),\frm)$ and $L^2(\bar B_{r_k}(x_k),[\nu]_N)$ for every $k \in \N$ and, by \cite[Lemma 4.3(b)]{AGS14}, $|\nabla\log(P_Tf+\delta)| \leq G$ $\frm$-a.e\ since $\log(P_Tf_K+\delta) \to \log(P_Tf+\delta)$ $\frm$-a.e.\ in $\supp(\nu)$. Setting $B_k \coloneqq \cup_{i=1}^k \bar B_{r_i}(x_i)$, by the same reasoning as above (noting that the choice of $N$ does not depend on the point $x_k$) we obtain
\[
\int_{B_k}\abs{\nabla \log(P_Tf+\delta)}^2\De[\nu]_N \leq \frac{1}{E_{2\kappa}(T)}\,\biggl[\cC_T(\mu,\nu)-\cH(\nu|\frm)\biggr]\,, \qquad \forall k \in \N.
\]
Taking the limit as $k \to \infty$, by $\nu(\partial\supp(\nu))=0$ we infer that $|\nabla\log(P_T f + \delta)|$ actually belongs to $L^2([\nu]_N)$. Passing then to the limit as $N \to \infty$ and $\delta \downarrow 0$, by monotonicity and again the fact that $\nu(\partial\supp(\nu))=0$ we precisely obtain \eqref{eq:corrector:nu} for $\mu,\nu$ as in the statement.
\end{proof}

\begin{remark}\label{rmk:bbs}
In \cite[Proposition 4.1]{GigTam19}, the dynamical representation formula \eqref{eq:bbs} for the entropic cost is actually proved under a $\CD(\kappa,N)$ assumption with $N < \infty$. However, the very same argument works also if $(M,\sfd,\frm)$ satisfies $\CD(\kappa,\infty)$ and $\frm(M)=1$. Indeed, the proof of \cite[Proposition 4.1]{GigTam19} essentially relies on the regularity and integrability of $t \mapsto \rho_t \coloneqq P_t f\,P_{T-t} g$, $t \mapsto \log P_t f$, $t \mapsto \log P_{T-t} g$ (and of their gradients) and these properties can be extended to the case $N=\infty$ as follows.

As concerns the regularity, given that $\mu,\nu$ have bounded densities and supports, the lower bound \eqref{eq:gaussian-lower-bound} on the heat kernel allows to deduce that $f,g \in L^\infty(\frm)$ with bounded supports too, exactly as in \cite[Proposition 2.1 and Theorem 2.2]{GigTam21}. Then the smoothing property of $P_T$ entails $C^\infty$-regularity, see \cite[Theorem 3.1]{Grigoryan09}.

As regards the integrability, what is needed is the existence of an $L^1(\De t \otimes \frm)$-function dominating, locally in $t$, $t \mapsto \log(P_t f)\rho_t$, $t \mapsto \log(P_{T-t} g)\rho_t$ and $t \mapsto |\nabla\log P_t f|^2\rho_t$, $t \mapsto |\nabla\log P_{T-t} g|^2\rho_t$. By the maximum principle, $t \mapsto \log(P_t f)\rho_t$, $t \mapsto \log(P_{T-t} g)\rho_t$ are dominated (up to multiplicative constants) by $t \mapsto P_{T-t}g$ and $t \mapsto P_t f$, respectively. As for $t \mapsto |\nabla\log P_t f|^2$, $t \mapsto |\nabla\log P_{T-t} g|^2$, the desired domination follows from Hamilton's gradient estimate under $\CD(\kappa,\infty$ condition \cite{JiangZhang16} and the bounds on $t \mapsto \log(P_t f)\rho_t$, $t \mapsto \log(P_{T-t} g)\rho_t$.
\end{remark}

We conclude the section with the proof of Lemma \ref{lemma:half:schr}.

\begin{proof}[Proof of Lemma \ref{lemma:half:schr}]
The existence and uniqueness of the minimizer $\pi^K\in\cQ^\nu_K$ is guaranteed by \cite[Proposition 1.17]{Marcel:notes}. Moreover, from the same reference we have $\schrplan{\mu}{\nu}{T}\ll\pi^K$ and 
\begin{equation}\label{prop:da:marcel}
\pi^K\to\schrplan{\mu}{\nu}{T}\,\text{ in total variation}\quad\text{ and  }\quad\cH(\pi^K|\rmR_{0,T})\to\cH(\schrplan{\mu}{\nu}{T}|\rmR_{0,T})=\cC_T(\mu,\nu)\,.
\end{equation}
Let us now prove \eqref{pi:K:fg}. Firstly, notice that if we introduce the marginal $\mu^K\coloneqq (\mathrm{proj}_{x_1})_{\#}\pi^K$, then it immediately follows that $\pi^K\in\Pi(\mu^K,\nu)\subseteq\cQ^\nu_K$ is the unique optimizer of the Schr\"odinger problem with marginals $\mu^K,\,\nu$ and $\cC_T(\mu^K,\nu)=\cH(\pi^K|\rmR_{0,T})$. Owing to \eqref{prop:da:marcel}, $\cH(\mu^K|\frm)\leq \cH(\pi^K|\rmR_{0,T})$ is eventually finite for $K$ large enough and therefore Proposition \ref{prop:fg} yields the existence of two positive measurable functions $f_K,\,g_K$ such that $\rmR_{0,T}$-a.s.\ it holds $\frac{\De\pi^K}{\De\rmR_{0,T}}(x,y)=f_K(x)\,g_K(y)$ with $\log f_K\in L^1(\mu^K)$ and $\log g_K\in L^1(\nu)$. Then, in view of the additive property of the relative entropy \cite[Equation (72) in Appendix A]{LeoSch} we may write 
\[
\cH(\pi|\rmR_{0,T})=\cH(\nu|\frm)+\int \cH\left(\pi(\cdot|\mathrm{proj}_{x_2}=y)|\frm P_T(y)\right)\,\De\nu(y)\quad\forall\pi\in\cQ^\nu_K\,,
\]
where $\frm P_T(y)=\rmR_{0,T}(\cdot|\mathrm{proj}_{x_2}=y)\in\cP(M)$ is the probability measure whose density is given by $\hp_T(x,y)\De\frm(x)$. 
Therefore minimizing \eqref{half:schr} is equivalent to solving for any fixed $y\in\supp(\nu)$ the problem
\[
\inf_{\frq(y)\in\cQ_K}\cH(\frq(y)|\frm P_T(y))\quad \text{with }\cQ_K \coloneqq \left\{\frq\in\cP(M)\,\colon\, \int\phi_i\De\frq=\int\phi_i\De\mu\quad\forall \,i=1,\dots,K\right\}\,,
\]
indeed if $\frq^*\colon M\to\cP(M)$ denotes the stochastic kernel associated to the minimizers of the above problem (whose existence and uniqueness, for any fixed $y\in\supp(\nu)$, are once again ensured by \cite[Proposition 1.17]{Marcel:notes}) then $\pi^K=\frq^*\otimes\nu$. Moreover, by arguing as in \cite[Example 1.18]{Marcel:notes} for any fixed $y\in\supp(\nu)$ we can write
\[
\frac{\De\frq^*(y)}{\De(\frm P_T(y))}(x)=c(y)\exp\left(\sum_{i=1}^Kb_i(y)\phi_i(x)\right)
\]
for some constants $c(y)>0$ and $b_i(y)\in\R$, possibly depending on $y$.
By combining the above expression, the $f_K,g_K$-decomposition and the Schr\"odinger system associated to $\cC_T(\mu^K,\nu)$ (cf.\ \eqref{schr:system} in Proposition \ref{prop:fg}) we deduce that
\begin{align*}
f_K(x)g_K(y)=&\,\frac{\De\pi^K}{\De\rmR_{0,T}}(x,y)=\frac{\De(\frq^*\otimes\nu)}{\De(\frm P_T\otimes\frm)}(x,y)=\frac{\De\nu}{\De\frm}(y)\,\frac{\De\frq^*(y)}{\De(\frm P_T(y))}(x)\\
=&\,g_K(y)P_Tf_K(y)\,c(y)\exp\left(\sum_{i=1}^Kb_i(y)\phi_i(x)\right)\,.
\end{align*}
Henceforth, for any $y\in \supp(\nu)$ it holds
\[
f_K(x)=P_Tf_K(y)\,c(y)\exp\left(\sum_{i=1}^Kb_i(y)\phi_i(x)\right)
\] 
and since the above left-hand side does not depend on the choice of $y$, we may choose a fixed  $y^*\in\supp(\nu)$ and write 
\[
f_K(x)=P_Tf_K(y^*)\,c(y^*)\exp\left(\sum_{i=1}^Kb_i(y^*)\phi_i(x)\right)\in(0,+\oo)\,.
\]
This proves \eqref{pi:K:fg} with  $C\coloneqq P_Tf_K(y^*)\,c(y^*)$ and $\lambda_i\coloneqq b_i(y^*)$. 

Finally, we claim that we can fix a renormalization for the decomposition $f_K,g_K$ (which is unique up to a multiplicative constant) such that $g_K$ converges $\frm$-a.e.\ to $g$. In order to do so, note that the convergence in total variation in \eqref{prop:da:marcel} implies that (along a non-relabeled subsequence)
\[
f_K(x)g_K(y) = \frac{\De\pi^K}{\De\rmR_{0,T}}(x,y) \to \frac{\De\schrplan{\mu}{\nu}{T}}{\De\rmR_{0,T}}(x,y)=f(x)g(y), \qquad \rmR_{0,T}\textrm{-a.s.}
\]
and a fortiori $\frm\otimes\frm$-a.e., thanks to the \eqref{CD} assumption and the Gaussian lower bounds \eqref{eq:gaussian-lower-bound} and \eqref{N:log:bound}. If $A\subseteq M\times M$ denotes the subset where the latter limit holds pointwise and $A^c$ its complement, then $\frm\otimes\frm(A^c)=0$ and Fubini's Theorem implies that for $\frm$-a.e.\ $x\in M$ it holds $\frm(A^c_x)=0$, where the section is defined as $A_x\coloneqq\{y\colon (x,y)\in A\}$. Combining this with $\mu\ll\frm$ we deduce that there exists an element $x^{*}\in\mathrm{int}(\supp(\mu))$ such that $f(x^{*})\neq 0$ and $\frm(A^c_{x^{*}})=0$. We have therefore proven that
\begin{equation}\label{quasi:ok}
f_K(x^{*}) g_K(y)\to f(x^{*}) g(y) \quad \textrm{for }\frm\text{-a.e. } y.
\end{equation}
By renormalizing the $f_K,g_K$-decomposition in such a way that $f_K(x^{*})=f(x^{*})\in (0,+\oo)$, \eqref{quasi:ok} reads as
\begin{equation}\label{m.a.s.conv}
    g_K\to g \quad \frm\text{-a.e.}
\end{equation}
As a direct consequence of this and \eqref{schr:system} we get the $\frm$-a.e.\ convergence of $P_Tf_K$ to $ P_Tf$ on $\supp(\nu)$, since the marginal $\nu$ is always the same at each step $K\in\N$.
\end{proof}

\section{Proof of the convergence of the gradients}\label{sec:grad}

As already mentioned in the Introduction, the first step in the proof of Theorem \ref{thm:gradient} is to show the convergence of the Schr\"odinger potentials to the Kantorovich ones as the parameter $T$ vanishes. To achieve this, we will rely on \cite[Proposition 5.1]{nutz2021entropic} and on \cite[Theorem 1.1]{Norris97}, which states that uniformly on compact subsets of $M\times M$ it holds
\[
T\log \hp_T(x,y)\,\shortmapor{T\downarrow 0}\,-\frac14\,\sfd^2(x,y)\,.
\]
Before starting with the proof, let us fix some notations. First, let $\rho \coloneqq \frac{\De\mu}{\De\frm}$ and $\sigma \coloneqq \frac{\De\nu}{\De\frm}$ denote the densities of the marginals. Then let $(\varphi^T,\psi^T)$ denote the unique Schr\"odinger potentials (as introduced in \eqref{phi:decomposition}) satisfying the symmetric normalization \eqref{gia:renormalization}, i.e. 
\[
\int\varphi^T\,\De\mu-\cH(\mu|\frm)=\int \psi^T\,\De\nu-\cH(\nu|\frm)\,;
\]
let $(\Phi_T,\,\Psi_T)$ be the unique entropic potentials for \eqref{EOT} (as defined in \eqref{entropic:potentials}) with $\varepsilon=T$ and cost $c_T(x,y) \coloneqq -T\log\hp_T(x,y)$ satisfying the symmetric normalization 
\[
\int\Phi_T\,\De\mu=\int\Psi_T\,\De\nu
\]
and, similarly, let $(\varphi^0,\psi^0)$ be a pair of Kantorovich potentials (defined in \eqref{W2}) satisfying the symmetric normalization
\begin{equation}\label{phi0:sym}
\int\varphi^0\,\De\mu=\int \psi^0\,\De\nu\,.
\end{equation}
Let us stress that, when rescaled by a factor $T$, the potentials $(\varphi^T,\psi^T)$ coincide with the entropic potentials for \eqref{EOT} with $\varepsilon=T$ and $\tilde c_T \coloneqq -T\log \frac{\De\rmR_{0,T}}{\De(\mu\otimes\nu)}$. Although similar, this EOT problem is different from the one associated with $c_T$; in particular, $(\varphi^T,\psi^T)$ is different from $(\Phi_T,\,\Psi_T)$, whence the choice of the uppercase notation to avoid confusion.

Indeed
\[
\tilde{c}_T(x,y) = c_T(x,y) + T\log\rho(x) + T\log\sigma(y),
\]
and therefore, even if for $T$ fixed the two EOT problems have the same minimizers and the same optimal values, the corresponding potentials differ. More precisely, it holds
\begin{equation}\label{rel:potentials}
T\varphi^T=\Phi_T+T\log\rho\quad\text{ and }\quad T\psi^T=\Psi_T +T\log\sigma\,.
\end{equation}
Hence, even if $(\varphi^T,\psi^T)$ are the Schr\"odinger potentials, from the EOT viewpoint it may be more convenient to work with $c_T$, and thus $(\Phi_T,\,\Psi_T)$, because of the dynamical interpretation of the cost as transition kernel of the associated SDE~\eqref{sde}.

As concerns $(\varphi^0,\psi^0)$, due to the lack of regularity of the OT problem~\eqref{W2}, even after the normalization \eqref{phi0:sym} uniqueness of the Kantorovich potentials may fail. Some results in this direction are known for some specific examples, for instance in the Euclidean case if at least one marginal is absolutely continuous with respect to the Lebesgue measure and if its support is connected then uniqueness holds \cite[Appendix B]{Marcel:EOT-deviations}. However here we are interested in the convergence of the Schr\"odinger map $\Id-2T\nabla\varphi^T$ to the Brenier map $\tau$ (as introduced above in \eqref{tau}), whose uniqueness instead holds true in our setting. More precisely, $\tau=\Id-2\nabla\varphi^0$ with $\varphi^0$ a Kantorovich potential as fixed above (cf.\ \cite[Proposition 3.1]{Figalli07}, \cite[Theorem 1.1]{FigGigli2011} and discussion therein).

\medskip

After this premise, let us show the strong convergence of the Schr\"odinger potentials (possibly along a subsequence), which will be pivotal in the subsequent proof of Theorem \ref{thm:gradient}.

\begin{lemma}[$L^1$-convergence of the potentials]\label{lemma:l1-conv-potentials}
Assume that $(M,\sfd,\frm)$ satisfies \eqref{CD} and that $\mu,\nu$ satisfy \eqref{H1}.
Then it holds:
\begin{itemize}
    \item $\{\Phi_T\}_{T\in (0,1]}$ and $\{\Psi_T\}_{T\in (0,1]}$ are strongly precompact in $L^1(\mu)$ and $L^1(\nu)$ respectively and their accumulation points are Kantorovich potentials $(\varphi^0,\psi^0)$ for \eqref{W2};
    \item or equivalently, $\{T\varphi^T\}_{T\in (0,1]}$ and $\{T\psi^T\}_{T\in (0,1]}$ are strongly precompact in $L^1(\mu)$ and $L^1(\nu)$ respectively and their accumulation points are Kantorovich potentials $(\varphi^0,\psi^0)$ for \eqref{W2}.
\end{itemize} 
If the Kantorovich potentials $(\varphi^0,\psi^0)$ associated to \eqref{W2} are unique, then
it holds
\[\Phi_T \to \varphi^0\text{ strongly in }L^1(\mu)\quad\text{ and }\quad \Psi_T \to \psi^0\text{ strongly in }L^1(\nu)\,,\]
and equivalently
\[T\varphi^T \to \varphi^0\text{ strongly in }L^1(\mu)\quad\text{ and }\quad T\psi^T \to \psi^0\text{ strongly in }L^1(\nu)\,.\] 
\end{lemma}

\begin{proof}
Firstly, notice that the equivalence between the statements follows from \eqref{rel:potentials} and the finite entropy condition in \eqref{H1}.
By \cite[Proposition 5.1]{nutz2021entropic}, to obtain the desired convergence for $\Phi_T$ and $\Psi_T$ it is sufficient to show that $c_T \coloneqq -T\log\hp_T \to \sfd^2/4$ uniformly on compact subsets as $T \to 0$ and that there exists a function $\overline{c}(x,y) = \overline{c}_1(x) + \overline{c}_2(y)$ with $\overline{c}_1 \in L^1(\mu)$ and $\overline{c}_2 \in L^1(\nu)$ such that $c_T \leq \overline{c}$ for all $T$ sufficiently small, say $T \leq 1$.

Under the $\CD(\kappa,N)$ condition with $N<\infty$, using the heat kernel lower bound \eqref{eq:gaussian-estimates} and assuming without loss of generality that $T \leq 1$, so that $\log\frm(B_{\sqrt{T}}(x)) \leq \log\frm(B_1(x))$, we see that
\begin{equation}\label{eq:upper-bound-cost}
\begin{split}
c_T(x,y) & \leq T\log C_1 + T\log\frm(B_{\sqrt{T}}(x)) + \frac{\sfd^2(x,y)}{4-\delta} + C_2T^2 \\
& \leq T\log C_1 + \left(T\log\frm(B_1(x))\right)^{+} + \frac{\sfd^2(x,y)}{4-\delta} + C_2T^2.
\end{split}
\end{equation}
By the trivial inequality $\sfd^2(x,y) \leq 2(\sfd^2(x,z) + \sfd^2(y,z))$ valid for any $z \in M$ we conclude that $c_T(x,y) \leq c'_T(x) + c''_T(y)$ with
\[
c'_T(x) \coloneqq \left(T\log\frm(B_1(x))\right)^{+} + \frac{2}{4-\delta}\sfd^2(x,z), \qquad 
c''_T(y) \coloneqq T\log C_1 + \frac{2}{4-\delta}\sfd^2(y,z) + C_2T^2.
\]
By the fact that $\nu \in \mathcal{P}_2(M)$, it is clear that $c''_T$ can be dominated by a $\nu$-integrable function not depending on $T$. As regards $c'_T$ it follows from $\mu\in\cP_2(M)$ and  Lemma \ref{lemma:log:balls} which gives the $\mu$-integrability of the positive part of $\log\frm(B_1(\cdot))$.  

On the other hand, if we assume that $(M,\sfd,\frm)$ satisfies $\CD(\kappa,\oo)$ with $\frm(M)=1$, then leveraging on the heat kernel lower bound \eqref{eq:gaussian-lower} we see that
\[
c_T(x,y) \leq \frac{\kappa T}{2(1-e^{-\kappa T})}\sfd^2(x,y)
\]
and in this case the trivial inequality $\sfd^2(x,y) \leq 2(\sfd^2(x,z) + \sfd^2(y,z))$, valid for any $z \in M$, readily provides us with functions $c_T',c_T''$ which are $\mu$- and $\nu$-integrable respectively.

Finally, in both cases, the fact that $c_T \to \sfd^2/4$ uniformly on compact subsets of $M \times M$ as $T \downarrow 0$ is ensured by \cite[Theorem 1.1]{Norris97}. Therefore we can apply \cite[Proposition 5.1]{nutz2021entropic}, which concludes our proof. 
\end{proof}

\begin{remark}\label{limite:costo}
Let us mention that the previous result gives as a direct consequence a new proof of the small-time limit of the renormalized Schr\"odinger cost, that is, under the same assumptions of Lemma \ref{lemma:l1-conv-potentials} we have proven
\[\lim_{T\to 0}T\cC_T(\mu,\nu)=\frac{\cW_2^2(\mu,\nu)}{4}\,.\]
\end{remark}

We are now ready to prove the convergence of the gradients.

\begin{proof}[Proof of Theorem \ref{thm:gradient}]
For the readers' convenience we divide our proof into four steps.

\medskip

\noindent\textbf{1st step: weak compactness of the gradients in $\mathbf{L^2(\boldsymbol{\mu})}$.} From the identity $\mu=f P_Tg\,\frm$ we deduce that $\log\rho = \varphi^T + \psi^T_0$, where $\psi^T_0\coloneqq \log P_T e^{\psi^T}$ and hence
\begin{equation*}
    \begin{split}
        \norm{T\nabla \varphi^T}_{L^2(TM,\mu)} & = \norm{T\nabla \log\rho -T\nabla \psi^T_0}_{L^2(TM,\mu)}  \\ &\leq T \norm{\nabla \log\rho}_{L^2(TM,\mu)} +T \norm{\nabla \psi^T_0 }_{L^2(TM,\mu)}.   \end{split}
\end{equation*} 
The correctors estimate \eqref{corr:est} allows to control the last term as
\[ 
\norm{\nabla \psi^T_0}^2_{L^2(TM,\mu)} \leq \frac{1}{E_{2\kappa}(T)}\biggl[\cC_T(\mu,\nu)-\cH(\mu|\frm)\biggr]
\]
whence 
\begin{equation*}
     \|T\nabla \varphi^T\|_{L^2(TM,\mu)} \leq T \sqrt{\cI(\mu)} + \sqrt{\frac{T}{E_{2\kappa}(T)}}\sqrt{T\,\cC_T(\mu,\nu) - T\,\cH(\mu|\frm)}.
\end{equation*}
This inequality together with the fact that $\lim_{T \downarrow 0}\frac{E_{2\kappa}(T)}{T}=1$, $\lim_{T\downarrow 0}T \cC_T(\mu,\nu)=\frac{1}{4}\cW_2^2(\mu,\nu)$ (cf.\ Remark \ref{limite:costo}) and  $\cI(\mu)<+\infty$ implies that 
\begin{equation}\label{eq:l2_conv_1bis}
    \limsup_{T\rightarrow 0}\norm{T\nabla \varphi^T}_{L^2(TM,\mu)}  \leq \frac{\cW_2(\mu,\nu)}{2}.
\end{equation}
Therefore, the sequence $(T\nabla\varphi^T)$ is weakly compact in $L^2(\mu)$.

\medskip

\noindent\textbf{2nd step: weak convergence of the gradients in $\mathbf{L^2(\boldsymbol{\mu_k})}$.} Fix $k \in \N$ and define $\rho_k := \min\{\rho,k\}$, $\mu_k := \rho_k\frm$ (note that $\mu_k$ needs not be a probability measure). We claim that for any $k$ (large enough) and for any weakly convergent subsequence of $\{T\nabla\varphi^T\}_{T>0}$ there exist a Kantorovich potential $\varphi^0$ and a subsubsequence $T_n\downarrow 0$ such that $T_n\nabla\varphi^{T_n}\weakto\nabla\varphi^0$ in $L^2(TM,\mu_k)$.

Firstly, note that for any subsequence of $\{T\nabla\varphi^T\}_{T>0}$, we can consider the corresponding subsequence of $\{T\varphi^T\}_{T>0}$ and deduce from Lemma \ref{lemma:l1-conv-potentials} that there exist a Kantorovich potential $\varphi^0\in L^1(\mu)$ and a subsubsequence $T_n\downarrow0$ such that $T_n\varphi^{T_n}\to \varphi^0$ strongly in $L^1(\mu)$, and a fortiori in $L^1(\mu_k)$ because $\mu_k \leq \mu$ trivially. Moreover, the family $\{T\nabla\varphi^T\}_{T>0}$ is weakly compact also in $L^2(TM,\mu_k)$. Hence, given a subsequence of $\{T\nabla\varphi^T\}_{T>0}$ weakly converging to some $\zeta_k \in L^2(TM,\mu_k)$ there exists $T_n\downarrow 0$ such that
\begin{equation}\label{appo:sub}
T_n\varphi^{T_n}\to\varphi^0\text{ in }L^1(\mu_k)\qquad\text{and}\qquad T_n\nabla\varphi^{T_n}\weakto \zeta_k\text{ in }L^2(TM,\mu_k)\,.
\end{equation}
Our claim is proven once we show that the weak limit $\zeta_k$ actually does not depend on $k$ and coincides with the weak gradient $\nabla \varphi^0$ $\frm$-a.e.

To this end, fix $x \in {\rm int}(\supp(\mu))$. By assumption there exists an open neighbourhood $B \subset \supp(\mu)$ of $x$ and a constant $c>0$ such that $\rho \geq c$ $\frm$-a.e.\ on $B$. Without loss of generality, we can assume that $B=B(x,r)$ for some radius $r>0$. In this way compactness in $L^2(B,\mu)$ implies compactness in $L^2(B,\frm)$ and, a fortiori, in $L^2(B,{\rm vol})$, since $\frm=e^{-U} {\rm vol}$ and $c' \leq e^{-U} \leq C'$ on $B$ for some constants $c',\,C'>0$. More explicitly,
\begin{equation}\label{eq:dai-che-ci-siamo}
c'T^2\int_B \abs{\nabla\varphi^T}^2\,\De{\rm vol} \leq T^2\int_B \abs{\nabla\varphi^T}^2\,\De\frm \leq \frac{T^2}{c}\int_B \abs{\nabla\varphi^T}^2\,\De\mu \leq \frac{1}{c}\,\norm{T\nabla\varphi^T}^2_{L^2(\mu)}
\end{equation}
and the right-hand side is uniformly bounded in $T$ by \eqref{eq:l2_conv_1bis}.

Now, inspired by the proof of \cite[Proposition 2.14]{barthe2008}, we observe that the Neumann Laplacian on the smooth compact Riemannian manifold with boundary $(\overline{B},g,{\rm vol}|_{\overline{B}})$ has a spectral gap (see e.g.\ \cite{greene1989s}). This is equivalent to the fact that ${\rm vol}|_{\overline{B}}$ satisfies a Poincar\'e inequality, whence in particular
\[
\int_B \abs{T\varphi^T}^2\,\De{\rm vol} \leq \bigg(\int_B T\varphi^T\,\De{\rm vol}\bigg)^2 + C_P \int_B \abs{T\nabla\varphi^T}^2\,\De{\rm vol}
\]
for some constant $C_P > 0$. Note that the first term on the right-hand side is uniformly bounded in $T_n$, since by the fact that $\rho_k \geq c$ (provided $k \geq c$) and $c'{\rm vol} \leq \frm \leq C'{\rm vol}$ $\frm$-a.e.\ in $B$, \eqref{appo:sub} implies $T_n\varphi^{T_n} \to \varphi^0$ in $L^1(B,{\rm vol})$, so that $(\int_B T_n\varphi^{T_n}\,\De{\rm vol})^2$ converges to $(\int_B\varphi^0\,\De{\rm vol})^2$ as $T_n \downarrow 0$.
The second one is bounded as well by \eqref{eq:dai-che-ci-siamo}. Hence we obtain compactness in $L^2(B,{\rm vol})$ for the subsequence $\{T_n\varphi^{T_n}\}_{n\in\N}$.

As a consequence, there exist a non-relabeled subsequence and $\xi \in L^2(B,{\rm vol})$ such that $T_n\varphi^{T_n} \rightharpoonup \xi$ in $L^2(B,{\rm vol})$. On the other hand, by \eqref{appo:sub} we know that $T_n\varphi^{T_n} \to \varphi^0$ in $L^1(\mu_k)$. This implies that $\xi = \varphi^0$ $\mu_k$-a.s.\ on $B$  and, as a byproduct, that 
\begin{equation}\label{appo:weak:sub}
T_n\varphi^{T_n} \weakto \varphi^0\text{ in }L^2(B,{\rm vol})
\end{equation}
without passing to a subsequence. Indeed let us consider an arbitrary non-negative $\phi \in C^\infty_c(B)$, so that $\phi \in L^1 \cap L^\infty(B,{\rm vol})$, and start observing that
\[
\abs{\int_{B} \phi(\varphi^0-\xi)\,\De{\rm vol}} \leq \abs{\int_{B} \phi(\varphi^0-T_n\varphi^{T_n})\,\De{\rm vol}} + \abs{\int_{B} \phi(T_n\varphi^{T_n}-\xi)\,\De{\rm vol}}\,.
\]
The second term vanishes since $T_n\varphi^{T_n} \rightharpoonup \xi$ in $L^2(B,{\rm vol})$. As for the first one, note that
\[
\begin{split}
\abs{\int_{B} \phi(\varphi^0-T_n\varphi^{T_n})\,\De{\rm vol}} & \leq \norm{\phi}_\infty\int_{B} \abs{\varphi^0-T_n\varphi^{T_n}} \, \De{\rm vol} \leq \frac{1}{c'}\norm{\phi}_\infty\int_{B} \abs{\varphi^0-T_n\varphi^{T_n}}\,\De\frm \\
& \leq \frac{1}{cc'}\norm{\phi}_\infty\int_{B} \abs{\varphi^0-T_n\varphi^{T_n}}\,\De\mu_k \leq \frac{1}{cc'}\norm{\phi}_\infty\norm{\varphi^0-T_n\varphi^{T_n}}_{L^1(\mu_k)} 
\end{split}
\]
where for the second and third inequality we have used once more that $\frm \geq c'{\rm vol}$ and $\rho_k \geq c$ $\frm$-a.e.\ on $B$ respectively. Hence, from $T_n\varphi^{T_n} \to \varphi^0$ in $L^1(\mu_k)$ it follows the weak convergence $T_n\varphi^{T_n} \rightharpoonup \varphi^0$ in $L^2(B,{\rm vol})$, without passing to any subsubsequence.

Now we are ready to prove the claim of this second step. Take $\beta \in C^\infty_c(B,TM)$ and observe that the assumption $\cI(\mu) < \infty$ can be equivalently restated as $\nabla\sqrt{\rho} \in L^2(TM,\frm)$. As a consequence, and by definition of $\rho_k$, $\nabla\sqrt{\rho_k} \in L^2(TM,\frm)$ too and this fact together with $\rho_k \leq k$ justifies the validity of the chain rule $\nabla\rho_k = \nabla((\sqrt{\rho_k})^2) = 2\sqrt{\rho_k}\nabla\sqrt{\rho_k}$ and proves that $|\nabla\rho_k| \in L^2(B,\frm)$, whence the applicability of the integration by parts formula, so that
\[
T\int \langle \nabla\varphi^T,\beta\rangle\,\De\mu_k = -T\int\varphi^T{\rm div}(\beta)\,\De\mu_k -  T\int\varphi^T\langle\beta,\nabla\rho_k\rangle \,\De\frm\,.
\]
On the one hand, for the left-hand side we know that
\[
\lim_{T_n \downarrow 0}T_n\int \langle \nabla\varphi^{T_n},\beta\rangle\,\De\mu_k = \int \langle \zeta_k,\beta\rangle\,\De\mu_k = \int \langle \rho_k\,\zeta_k,\beta\rangle\,\De\frm\,.
\]
As concerns the right-hand side, the fact that ${\rm div}(\beta)$ is bounded in $B$ and the strong convergence in \eqref{appo:sub} ensure that
\[
\lim_{T_n \downarrow 0}T_n\int\varphi^{T_n}{\rm div}(\beta)\,\De\mu_k = \int\varphi^0{\rm div}(\beta)\,\De\mu_k = \int\varphi^0\rho_k\,{\rm div}(\beta)\,\De\frm\,.
\]
For the limit of the second term, $\langle\beta, \nabla\rho_k\rangle \in L^2(\frm)$ together with $\supp(\beta) \subset B$, $e^{-U/2} \leq \sqrt{C'}$ in $B$ and $T\varphi^T \rightharpoonup \varphi^0$ in $L^2(B,{\rm vol})$, proved in \eqref{appo:weak:sub} above, allow us to deduce that
\[
\lim_{T_n \downarrow 0}\,{T_n}\int\varphi^{T_n}\langle\beta,\nabla\rho_k\rangle \,\De\frm = \int\varphi^0 \langle\beta,\nabla\rho_k\rangle \,\De\frm.
\]
Hence, after rearrangement, 
\[
-\int\varphi^0\rho_k\,{\rm div}(\beta)\,\De\frm = \int \langle \rho_k\,\zeta_k,\beta\rangle\,\De\frm + \int\varphi^0\langle\beta,\nabla\rho_k\rangle\,\De\frm = \int\langle\beta,\rho_k\,\zeta_k + \varphi^0\,\nabla\rho_k\rangle\,\De\frm,
\]
which means that $\nabla(\varphi^0\rho_k) = \rho_k\,\zeta_k + \varphi^0\nabla\rho_k$ on $B$ by definition of weak gradient. 

To conclude that $\zeta_k = \nabla\varphi^0$ on $B$, recall that $\rho_k \in W^{1,2}(B)$. Using the lower bound $\rho_k \geq c$ $\frm$-a.e.\ on $B$, by chain rule we get that also $\rho_k^{-1} \in W^{1,2}(B)$ and therefore, by Leibniz rule,
\[
\nabla\varphi^0 = \nabla(\varphi^0\rho_k \cdot \rho_k^{-1}) = \frac{\nabla(\varphi^0\rho_k)}{\rho_k} + \varphi^0\rho_k\nabla(\rho_k^{-1}) = \zeta_k \qquad \frm\textrm{-a.e. on } B.
\]
Since $B$ was obtained starting from an arbitrary $x \in {\rm int}(\supp(\mu))$, we conclude that $\zeta_k = \nabla\varphi^0$ $\frm$-a.e.\ in ${\rm int}(\supp(\mu))$ and therefore, up to changing the weak limit on a $\mu$-null set we have proven that 
\begin{equation}\label{eq:conv-approx}
T_n\nabla\varphi^{T_n}\weakto \nabla\varphi^0\text{ in }L^2(TM,\mu_k)\,.
\end{equation}

\medskip

\noindent\textbf{3rd step: weak convergence of the gradients in $\mathbf{L^2(\boldsymbol{\mu})}$.} We now claim that the result above can be improved to
\[
T_n\nabla\varphi^{T_n}\weakto \nabla\varphi^0\text{ in }L^2(TM,\mu)\,.
\]
To this end, fix $\beta \in L^2 \cap L^\infty(TM,\mu)$ and start observing
\[
\begin{split}
\bigg|\int\langle T_n\nabla\varphi^{T_n}-\nabla\varphi^0,\beta\rangle\,\De\mu\bigg| & \leq \bigg|T_n\int\langle \nabla\varphi^{T_n},\beta\rangle(\rho-\rho_k)\,\De\frm\bigg| + \bigg|\int\langle T_n\nabla\varphi^{T_n}-\nabla\varphi^0,\beta\rangle\,\De\mu_k\bigg| \\ 
& \quad + \bigg|\int\langle \nabla\varphi^0,\beta\rangle(\rho-\rho_k)\,\De\frm\bigg|\,.
\end{split}
\]
Recalling that $\cW_2(\mu,\nu) = 2\norm{ \nabla\varphi^0}_{L^2(TM,\mu)}$ (see \cite[Theorem 10.3]{AGS14}, paying attention to the different rescaling), we see that the third term on the right-hand side vanishes as $k \to \infty$ by dominated convergence theorem. As concerns the first one, let us first estimate it as follows
\[
\begin{split}
\bigg|T_n\int\langle \nabla\varphi^{T_n},\beta\rangle(\rho-\rho_k)\,\De\frm\bigg| & \leq T_n\int|\langle \nabla\varphi^{T_n},\beta\rangle|\Big(1-\frac{\rho_k}{\rho}\Big)\,\De\mu \leq \|\beta\|_\infty\int T_n|\nabla\varphi^{T_n}|\Big(1-\frac{\rho_k}{\rho}\Big)\,\De\mu \\
& \leq \|\beta\|_\infty\|T_n\nabla\varphi^{T_n}\|_{L^2(TM,\mu)}\Big\|1-\frac{\rho_k}{\rho}\Big\|_{L^2(\mu)}
\end{split}
\]
and then observe that $\|T_n\nabla\varphi^{T_n}\|_{L^2(TM,\mu)}$ is bounded as $T_n \downarrow 0$ by \eqref{eq:l2_conv_1bis}, while $\Big\|1-\frac{\rho_k}{\rho}\Big\|_{L^2(\mu)} \to 0$ as $k \to \infty$ by dominated convergence. This allows to deduce that, given $\varepsilon>0$ there exists $k'$ independent of $T_n$ so that
\[
\bigg|\int\langle T_n\nabla\varphi^{T_n}-\nabla\varphi^0,\beta\rangle\,\De\mu\bigg| \leq \bigg|\int\langle T_n\nabla\varphi^{T_n}-\nabla\varphi^0,\beta\rangle\,\De\mu_{k'}\bigg| + 2\varepsilon.
\]
Taking now the limit as $T_n \downarrow 0$, the right-hand side converges to $2\varepsilon$ by \eqref{eq:conv-approx} and the arbitrariness of $\varepsilon$ together with the density of $L^2 \cap L^\infty(TM,\mu)$ in $L^2(TM,\mu)$ allows to conclude that $T_n\nabla\varphi^{T_n} \weakto \nabla\varphi^0$ in $L^2(TM,\mu)$ as claimed.

\medskip

\noindent\textbf{4th step: strong convergence of the gradients in $\mathbf{L^2(\boldsymbol{\mu})}$.}
The previous step and the uniqueness of the Brenier map $\tau=\Id-2\nabla \varphi^0$ grant that the whole sequence of gradients converges 
\[
T\nabla\varphi^T\weakto \nabla\varphi^0\text{ in }L^2(TM,\mu_k)\,, \qquad \forall k \in \N,
\]
which automatically implies
\[
\liminf_{T\rightarrow 0}\norm{T\nabla \varphi^T_0}_{L^2(TM,\mu_k)}  \geq \norm{\nabla \varphi^0}_{L^2(TM,\mu_k)}\,.
\]
Recalling again that $\cW_2(\mu,\nu) = 2\norm{ \nabla\varphi^0}_{L^2(TM,\mu)}$, we deduce from \eqref{eq:l2_conv_1bis} that 
 \begin{equation*}
     \lim_{T\rightarrow 0}\norm{T\nabla \varphi^T_0}_{L^2(TM,\mu)}  = \norm{\nabla \varphi^0}_{L^2(TM,\mu)}.
 \end{equation*}
Since in a Hilbert space such as $L^2(TM,\mu)$ weak convergence plus convergence of the norm implies strong convergence, we obtain the desired conclusion.
\end{proof}

\begin{remark}
The previous proof runs exactly in the same way if we consider the entropic potentials $(\Phi_T,\Psi_T)$ instead of the rescaled Schr\"odinger potentials $(T\varphi^T,T\psi^T)$. Moreover, the weak compactness in the first step can be proven without assuming finite Fisher information $\cI(\mu)$, since
\[
\norm{T\nabla\Phi_T}_{L^2(\mu)}=\norm{T\nabla\varphi^T-T\nabla\log\rho}_{L^2(\mu)}=T\,\norm{\psi^T_0}_{L^2(\mu)}\,.
\]
However, even though we are considering the entropic potentials $(\Phi_T,\Psi_T)$, the assumption of finite Fisher information $\cI(\mu)<+\oo$ is needed in the second step where we identify the weak limit of $\{\Phi_{T_n}\}_{n\in\N}$ with $\nabla\varphi^0$. 
\end{remark}

\begin{corollary}
Under the same assumptions of Theorem \ref{thm:gradient}, if $\mu$ satisfies a Poincaré inequality, then the convergence of the potentials in Lemma \ref{lemma:l1-conv-potentials} holds true also in $L^2(\mu)$.
\end{corollary}

 \section{Proofs of the quantitative stability estimates}\label{stab:section}

In this section we show how from the corrector estimates \eqref{corr:est} it is possible to derive novel quantitative stability estimates for SP. Firstly we prove Theorem \ref{stab:piani} and as a byproduct we deduce in Proposition \ref{stab:entropic:cost} a certain stability estimate for the entropic cost. Alongside we show how a finite Fisher information assumption gives similar bounds independent of the symmetric relative entropy between the marginals. 
Finally, we conclude the section with an application of our results to the EOT setting with quadratic cost function.

Throughout the whole section $(f,g)$ denotes the $fg$-decomposition for $\cC_T(\mu,\nu)$ while $(\bar f,\bar g)$ denotes the corresponding one for $\cC_T(\bar \mu,\bar\nu)$ and we implicitly assume that both are renormalized according to \eqref{gia:renormalization}. 
Before moving to the proof of Theorem \ref{stab:piani} let us give some $L^p$-bounds (with respect to the reference measure $\frm$) for $(P_T g)^{-1}$ and $(P_T f)^{-1}$,  that combined with Lemma \ref{log:integrability} will be pivotal for the validity of all the computations performed hereafter.

\begin{lemma}\label{LemmaPg-1}
Assume \eqref{CD}, \eqref{I} and \eqref{H1}. Then $(P_Tg)^{-1}\in L^p(\frm)$ and $(P_T f)^{-1}\in L^p(\frm)$ with $p= r\,T$. More precisely, it holds
\begin{equation}\begin{aligned}\label{eq:Pg-1}
     \norm{(P_T g)^{-1}}_p\leq C_1\, \exp\left[\frac{M_2(\nu)}{T}-\frac{\cS_T(\mu,\nu)}{2T}+C_2\,T\right]\,,\\
     \norm{(P_T f)^{-1}}_p\leq C_1\, \exp\left[\frac{M_2(\mu)}{T}-\frac{\cS_T(\mu,\nu)}{2T}+C_2\,T\right]\,,
\end{aligned}\end{equation}
where the constants $C_1>0$ and $C_2\geq 0$ do not depend on the marginals $\mu,\nu$ and neither on $f$ or $g$. Moreover if $\kappa\geq 0$ or if $\CD(\kappa,\oo)$ with $\frm(M)=1$ holds, then $C_2=0$.
\end{lemma}
 
\begin{proof}
We will only prove the first inequality since the proof of the second runs exactly in the same fashion. Firstly, notice that from Jensen's inequality it follows
\begin{align*}
     \log P_T g(x)&=\log \int g(y) \hp_T(x,y)\De\frm(y)=\log \int g(y) \hp_T(x,y)\sigma(y)^{-1}\De\nu(y)\\
     &\geq\int\log g\,\De\nu-\int\log\sigma\De\nu+\int\log \hp_T(x,y)\,\De\nu(y)\\
     &=\frac12\left[\cC_T(\mu,\nu)-\cH(\mu|\frm)-\cH(\nu|\frm)\right]+\int\log \hp_T(x,y)\,\De\nu(y)\\
     &=\frac{\cS_T(\mu,\nu)}{2T}+\int\log \hp_T(x,y)\,\De\nu(y)\,,
\end{align*}
where the last step follows from \eqref{gia:renormalization}.
On the one hand, if $(M,\,\sfd,\,\frm)$ satisfies $\CD(\kappa,\oo)$ for some $\kappa\in\R$ and if $\frm(M)=1$, then from \eqref{infinity:log:bound} we deduce the lower bound
\begin{align*}\log P_T g(x)
\geq  \frac{\cS_T(\mu,\nu)}{2T}-\frac{\sfd^2(x,z_0)}{T}-\frac{M_2(\nu)}{T} \,.\end{align*}
 Therefore if we take $p=r\,T$ we have, owing to \eqref{I},
 \begin{align*}
     \norm{(P_T g)^{-1}}_p^p &=\int (P_T g)^{-p}\De\frm=\int e^{-p\log P_Tg}\De\frm\\
     &\leq \exp\left[p\left(\frac{M_2(\nu)}{T}-\frac{\cS_T(\mu,\nu)}{2T}\right)\right]\int e^{\frac{p}{T}\sfd^2(x,z_0)}\De\frm(x)
     \\&\leq C_1^p\,\exp\left[p\left(\frac{M_2(\nu)}{T}-\frac{\cS_T(\mu,\nu)}{2T}\right)\right]\,.
 \end{align*}
 On the other hand if $\CD(\kappa,N)$ holds with $N<+\oo$, the lower bound \eqref{N:log:bound} leads to 
 \begin{align*}\log P_T g(x)\geq \frac{\cS_T(\mu,\nu)}{2T}-\log\left[C_1\, \frm\left(B_{\sqrt{T}}(x)\right)\right]-C_2\,T-\frac{\sfd^2(x,z_0)}{T}-\frac{M_2(\nu)}{T}\,.\end{align*}
  Therefore if we consider again $p=r\,T$, up to relabelling the positive constant $C_1$ at each step, we have
 \begin{align*}
     \norm{(P_T g)^{-1}}_p^p &=\int (P_T g)^{-p}\De\frm=\int e^{-p\log P_Tg}\De\frm\\
     &\leq C_1^p\,\exp\left[p\left(\frac{M_2(\nu)}{T}-\frac{\cS_T(\mu,\nu)}{2T}+C_2\,T\right)\right]\int\frm\left(B_{\sqrt{T}}(x)\right)^p e^{\frac{p}{T}\,\sfd^2(x,z_0)}\De\frm(x)\\
     &\leq C_1^p\,\exp\left[p\left(\frac{M_2(\nu)}{T}-\frac{\cS_T(\mu,\nu)}{2T}+C_2\,T\right)\right]\,.
 \end{align*}
 This concludes our proof.
 \end{proof}

We are now ready to prove the quantitative stability estimate for the optimizers of SP.

\begin{proof}[Proof of Theorem \ref{stab:piani}]
Let us start with proving the first estimate \eqref{eq:stab:plans}. Let us assume that $\cH^\mathrm{sym}(\mu,\bar\mu),\,\cH^\mathrm{sym}(\nu,\bar\nu)<\oo$, otherwise our claim is trivial. Firstly, notice that we can write 
\begin{equation*}
\begin{aligned}
    \cH(\schrplan{\mu}{\nu}{T}|\schrplan{\bar\mu}{\bar\nu}{T})&-\cH(\nu|\bar\nu)-\cH(\mu|\bar\mu)\\
    &=\int\log \frac{f(x)g(y)}{\bar f(x)\bar g(y)}\,\De\schrplan{\mu}{\nu}{T}(x,y)-\int\log\frac{g P_Tf}{\bar g P_T\bar f}\De\nu-\int\log\frac{f P_Tg}{\bar f P_T \bar g}\De\mu\\
   &= \int \log\frac{P_T\bar g(x)\,P_T\bar f(y)}{P_Tg(x)\,P_Tf(y)}\,\De\schrplan{\mu}{\nu}{T}(x,y)\\
   &= \int\biggl[\log P_T\bar g(x)+\log P_T\bar f(y)-\log P_T g(x)-\log P_T f(y)\biggr]\De\schrplan{\mu}{\nu}{T}(x,y)\,.
\end{aligned}
\end{equation*}
We claim that the above right-hand side is equal to
\begin{equation}
\label{sum:well:def}\int\log P_T \bar g\De\mu+ \int\log P_T\bar f\De\nu-\int\log P_T g\De\mu-\int \log P_T f\De\nu\,.
\end{equation}
Since $\cH(\mu|\frm),\,\cH(\nu|\frm)$ are both finite, thanks to Proposition \ref{prop:fg} we already know that $\log P_Tg\in L^1(\mu)$ and $\log P_Tf\in L^1(\nu)$ and that under our normalization \eqref{gia:renormalization} they read as 
\[-\int\log P_T g\De\mu=-\int \log P_T f\De\nu=\frac{\cS_T(\mu,\nu)}{2T}\in(-\oo,+\oo)\,.\]
Hence for the claim to be true it only remains to prove that 
\begin{equation}\label{menoL1}
\left(\log P_T\bar g\right)^{-}\in L^1(\mu)\quad\text{ and }\quad\left(\log P_T\bar f\right)^-\in L^1(\nu)\,,
\end{equation}
because this implies that $\int\log P_T \bar g\De\mu,\, \int\log P_T\bar f\De\nu>-\oo$, whence the fact that \eqref{sum:well:def} is a well-defined summation. We will prove \eqref{menoL1} by relying on Lemmata \ref{LemmaPg-1} and \ref{log:integrability}. We  consider two different cases.

\medskip

\noindent\textbf{1st case: $\frm\in\cP(M)$.} Consider the positive measurable function 
\[h=\IND_{\{P_T\bar g<1\}}\,P_T \bar g+\IND_{\{P_T\bar g\geq 1\}}\]
and notice that $(\log P_T\bar g)^-=-\log h$. Since $\frm(M)=1$ and $h\leq 1$ we have $h\in L^1(\frm)$; moreover $h^{-1}\in L^{rT}(\frm)$ since
\[
\norm{h^{-1}}_{rT}^{rT}=\int h^{-rT}\,\De\frm=\int\left[ \IND_{\{P_T\bar g<1\}}(P_T \bar g)^{-rT}+\IND_{\{P_T\bar g\geq 1\}}\right]\,\De\frm\leq \norm{(P_T\bar g)^{-1}}_{rT}^{rT}+\frm\{P_T\bar g\geq 1\}
\]
and the right-hand side is finite because $\frm\left\{P_T\bar g\geq 1\right\}\leq 1$ and $(P_T\bar g)^{-1}\in L^{rT}(\frm)$ by Lemma \ref{LemmaPg-1}.
Therefore we may apply Lemma \ref{log:integrability} with $\frq=\frm$ and $\frp=\mu$ and deduce 
\begin{align*}
   \int \abs{\log h}\,\De\mu\leq 2\,e^{\cH(\mu|\frm)-1}\,\left\{ \frac{1}{rT}\vee \log_2\left[1+\left(1+\norm{(P_T\bar g)^{-1}}_{rT}^{rT}\right)^{\frac{1}{rT}}\right]\right\}<+\oo\,.
\end{align*}
Hence, we have proven that
\begin{align*}
      0\leq \int(\log P_T\bar g)^-\,\De\mu=\int \abs{\log h}\,\De\mu<+\oo\,,\quad\text{ which is equivalent to }\left(\log P_T\bar g\right)^-\in L^1(\mu)\,.
\end{align*}
\medskip

\noindent\textbf{2nd case: $\frm\not\in\cP(M)$.} The proof is similar to the previous case, however we work with the probability measure $\frm_W$ (defined via \eqref{def:entropy}) instead of $\frm$. Indeed, since we are assuming $\cH(\mu|\frm)<+\oo$, there exists a positive measurable function $W\colon M\to [0,+\oo)$ such that 
\[z_W=\int e^{-W}\De\frm<+\oo\,,\quad W\in L^1(\mu)\quad\text{ and }\quad\cH(\mu|\frm_W)<+\oo\,,\]
where $\frm_W\coloneqq z_W^{-1} e^{-W}\frm\in\cP(M)$. By considering the same positive measurable function 
\[h=\IND_{\{P_T\bar g<1\}}\,P_T \bar g+\IND_{\{P_T\bar g\geq 1\}}\]
 we are guaranteed once again that $\norm{h}_{ L^1(\frm_W)}\leq 1$ and that
\begin{align*}
\norm{h^{-1}}_{L^{rT}(\frm_W)}^{rT} &=\int h^{-rT}\,\De\frm_W=\int\left[ \IND_{\{P_T\bar g<1\}}\,(P_T \bar g)^{-rT}+\IND_{\{P_T\bar g\geq 1\}}\right]\,\De\frm_W\\
&\leq \norm{(P_T\bar g)^{-1}}^{rT}_{L^{rT}(\frm_W)}+\frm_W\left\{P_T\bar g\geq 1\right\}\leq  \norm{(P_T\bar g)^{-1}}^{rT}_{L^{rT}(\frm_W)}+1\,,
\end{align*}
By arguing in the same way as in the second half of Lemma \ref{LemmaPg-1} (since $\frm$ is not a probability, recall that we are under the $\CD(\kappa, N)$ condition with $N<+\oo$) we see that
\begin{align*}
     &\norm{(P_T \bar g)^{-1}}_{L^{rT}(\frm_W)}^{rT}=\int (P_T \bar g)^{-rT}\De\frm_W\\
      &\quad\leq C_1^{rT}\,\exp\left[rT\left(\frac{M_2(\bar\nu)}{T}-\frac{\cS_T(\bar\mu,\bar\nu)}{2T}+C_2\,T\right)\right]\int\frm\left(B_{\sqrt{T}}(x)\right)^{rT} e^{r\,\sfd^2(x,z_0)}\De\frm_W(x)<+\oo\,,
\end{align*}
where we have used the fact that $W\geq 0$ and the integrability condition \eqref{I}. Therefore we have shown that $\norm{(P_T \bar g)^{-1}}_{L^{rT}(\frm_W)}^{rT}$ is finite and once again from Lemma \ref{log:integrability} (this time with $\frq=\frm_W$ and $\frp=\mu$) we deduce 
\begin{align*}
   \int \abs{\log h}\,\De\mu\leq 2\,e^{\cH(\mu|\frm_W)-1}\,\left\{ \frac{1}{rT}\vee \log_2\left[1+\left(1+\norm{(P_T\bar g)^{-1}}_{L^{rT}(\frm_W)}^{rT}\right)^{\frac{1}{rT}}\right]\right\}<+\oo\,.
\end{align*}
Hence, we have proven that
\begin{align*}
      0\leq \int(\log P_T\bar g)^-\,\De\mu=\int \abs{\log h}\,\De\mu<+\oo,\,\quad\text{ which is equivalent to }\left(\log P_T\bar g\right)^-\in L^1(\mu)\,.
\end{align*}
By arguing in the same way we can also prove that $\left(\log P_T\bar f\right)^-\in L^1(\nu)$, whence the validity of \eqref{menoL1}. From this, we deduce that $\int\log P_T \bar g\De\mu,\,\int\log P_T\bar f\De\nu\in (-\oo,+\oo]$ and that \eqref{sum:well:def} is a well-defined summation. As a consequence it follows
\begin{equation}\label{new:giappo1}
\begin{aligned}
\cH(\schrplan{\mu}{\nu}{T}|\schrplan{\bar\mu}{\bar\nu}{T}) & -\cH(\nu|\bar\nu) - \cH(\mu|\bar\mu)\\ 
&=\int\log P_T \bar g\,\De\mu+ \int\log P_T\bar f\,\De\nu-\int\log P_T g\,\De\mu-\int \log P_T f\,\De\nu\,. 
\end{aligned}
\end{equation}
By exchanging the roles of $\mu,\nu$ and $\bar\mu,\bar\nu$ we can also write  \begin{equation}\label{new:giappo2}
\begin{aligned}
\cH(\schrplan{\bar\mu}{\bar\nu}{T}|\schrplan{\mu}{\nu}{T})&-\cH(\bar\nu|\nu)-\cH(\bar\mu|\mu) \\ 
&=\int\log P_T g\,\De\bar\mu+ \int\log P_T f\,\De\bar\nu-\int\log P_T \bar g\,\De\bar\mu-\int \log P_T \bar f\,\De\bar\nu\,,
\end{aligned}
\end{equation}
which added to \eqref{new:giappo1} gives
\begin{align*}
    \cH^\mathrm{sym}(\schrplan{\mu}{\nu}{T},\schrplan{\bar\mu}{\bar\nu}{T})=\cH^\mathrm{sym}(\mu,\bar\mu)+\cH^\mathrm{sym}(\nu,\bar\nu)-\int\log P_T g\,\De(\mu-\bar\mu)-\int\log P_T\bar g\,\De(\bar\mu-\mu)\\
    -\int\log P_T f\,\De(\nu-\bar\nu)-\int\log P_T\bar f\,\De(\bar\nu-\nu)\\
    \leq \cH^\mathrm{sym}(\mu,\bar\mu)+\cH^\mathrm{sym}(\nu,\bar\nu)+\norm{\nabla\log P_T g}_{L^2(\mu)}\norm{\mu-\bar\mu}_{\dot{H}^{-1}(\mu)}+\norm{\nabla\log P_T \bar g}_{L^2(\bar\mu)}\norm{\bar\mu-\mu}_{\dot{H}^{-1}(\bar\mu)}\\
    +\norm{\nabla\log P_T f}_{L^2(\nu)}\norm{\nu-\bar\nu}_{\dot{H}^{-1}(\nu)}+\norm{\nabla\log P_T \bar f}_{L^2(\bar\nu)}\norm{\bar\nu-\nu}_{\dot{H}^{-1}(\bar\nu)}\,.
\end{align*}
Let us point out that so far we have not relied on the assumption \eqref{H2}, which is needed now when applying the corrector estimates. Indeed, given the above bound, inequality \eqref{eq:stab:plans} follows from the corrector estimates (cf.\ Proposition \ref{lemma:gia}).

The proof of \eqref{eq:fish:stab:plans} under the assumption of finite Fisher information is postponed after the next \emph{mixed integrability} result, which complements what we have proven so far and which is necessary for the integral computations that will follow.
\end{proof}

\begin{corollary}\label{cor:svolta}
Let $(M,\sfd, \frm)$ satisfy \eqref{CD} and \eqref{I}. For any $(\mu,\nu)$ and $(\bar\mu,\bar\nu)$ satisfying \eqref{H1}, \eqref{H2} and such that $\norm{\mu-\bar\mu}_{\dot{H}^{-1}(\mu)},\,\norm{\bar\mu-\mu}_{\dot{H}^{-1}(\bar\mu)},\,\norm{\nu-\bar\nu}_{\dot{H}^{-1}(\nu)}$, and $\norm{\bar\nu-\nu}_{\dot{H}^{-1}(\bar\nu)}$ are all finite and with  $\cH^\mathrm{sym}(\mu,\bar\mu),\,\cH^\mathrm{sym}(\nu,\bar\nu)<\oo$, we have
\[
\log P_T g,\,\log P_T\bar g\in L^1(\mu)\cap L^1(\bar\mu)\quad\text{ and }\quad\log P_T f,\,\log P_T\bar f\in L^1(\nu)\cap L^1(\bar\nu)\,.
\]
As a consequence we deduce that also
\[
\log  f,\,\log \bar f\in L^1(\mu)\cap L^1(\bar\mu)\quad\text{ and }\quad\log g,\,\log \bar g\in L^1(\nu)\cap L^1(\bar\nu)\,.
\]
\end{corollary}

\begin{proof}
As mentioned in the previous proof we already know, thanks to Proposition \ref{prop:fg}, that
\[
\log P_Tg\in L^1(\mu),\,\log P_Tf\in L^1(\nu),\,\log P_T\bar g\in L^1(\bar\mu),\text{ and }\log P_T\bar f\in L^1(\bar\nu)\,.
\] 
Now, given our assumptions, the left-hand side of \eqref{new:giappo1} is finite (thanks to Theorem \ref{stab:piani}) and since we have $\log P_Tg\in L^1(\mu),\,\log P_Tf\in L^1(\nu)$ and $\int\log P_T \bar g\De\mu,\,\int\log P_T\bar f\De\nu\in (-\oo,+\oo]$ (from the proof of Theorem \ref{stab:piani}), we deduce that 
\[
\log P_T \bar g\in L^1(\mu)\quad\text{ and } \quad\log P_T\bar f\in L^1(\nu)\,.
\]
Similarly, working with \eqref{new:giappo2} we get $\log P_T g\in L^1(\bar{\mu})$ and $\log P_T f\in L^1(\bar{\nu})$.

Finally, from what we have just proven above and thanks to $\cH^\mathrm{sym}(\mu,\bar\mu),\,\cH^\mathrm{sym}(\nu,\bar\nu)$ being finite we deduce 
\[
\log\frac{f}{\bar f}\in L^1(\mu)\cap L^1(\bar\mu)\quad\text{ and }\quad\log\frac{g}{\bar g}\in L^1(\nu)\cap L^1(\bar\nu)\,,
\]
which combined with $\log f\in L^1(\mu)$, $\log g\in L^1(\nu)$, $\log\bar f\in L^1(\bar\mu)$ and $\log \bar g\in L^1(\bar\nu)$ gives our final assertion.
\end{proof}

\begin{proof}[Resuming the proof of Theorem \ref{stab:piani}]
Let us now prove \eqref{eq:fish:stab:plans}. Without loss of generality we may assume that the Fisher information of the marginals and the negative Sobolev norms are all finite, otherwise the claimed estimate is trivial. Let us start by showing that this guarantees the finiteness of the symmetric relative entropies $\cH^\mathrm{sym}(\mu,\bar\mu),\, \cH^\mathrm{sym}(\nu,\bar\nu)$. Indeed, since $\cI(\mu)<+\oo$ we have $\nabla\log\frac{\De\mu}{\De\frm}\in L^2(\mu)$ and therefore by the definition of $\norm{\mu-\bar\mu}_{\dot{H}^{-1}(\mu)}$ we deduce
\[
\abs{\int\log\frac{\De\mu}{\De\frm}\,\De(\mu-\bar\mu)}\leq \sqrt{\cI(\mu)}\,\norm{\mu-\bar\mu}_{\dot{H}^{-1}(\mu)} <+\oo\,.
\]
Thanks to the above bound and to the finiteness of $\cH(\mu|\frm)$ and $\cH(\bar\mu|\frm)$ we are allowed to write
\begin{align*}
\int\log\frac{\De\bar\mu}{\De\mu}\,\De\bar\mu=\int\left[\frac{\De\bar\mu}{\De\frm}\log\frac{\De\bar\mu}{\De\frm}+\left(\frac{\De\mu}{\De\frm}-\frac{\De\bar\mu}{\De\frm}\right)\log\frac{\De\mu}{\De\frm}-\frac{\De\mu}{\De\frm}\log\frac{\De\mu}{\De\frm}\right]\De\frm\\
=\cH(\bar\mu|\frm)-\cH(\mu|\frm)+\int \log\frac{\De\mu}{\De\frm}\,\De(\mu-\bar\mu)<+\oo\,,
\end{align*}
which reads as $\cH(\bar\mu|\mu)<+\oo$. Let us just mention that above the Radon-Nikodym derivative is well defined since, under our assumption, $\bar\mu\sim\mu$ are equivalent. Indeed, if there were a Borel subset $A\subset M$ such that $\mu(A)=0$ and $\bar\mu(A)>0$, then by choosing as test function $h$ (a mollified version of) $\IND_A$ we would get $\norm{\mu-\bar\mu}_{\dot{H}^{-1}(\mu)}=+\oo$ which we are assuming to be finite; therefore $\bar\mu\ll\mu$. Similarly one can prove $\mu\ll\bar\mu$, and hence $\bar\mu\sim\mu$. We have thus proven that $\cH(\bar\mu|\mu)$ is finite. By reasoning in the same fashion one can prove that
$\cH(\mu|\bar\mu)$, $\cH(\bar\nu|\nu)$ and $\cH(\nu|\bar\nu)$ are also finite, whence the finiteness of the symmetric entropies.

At this stage, the proof is similar to the one already presented above. It is indeed enough to notice that
\begin{equation*}
\begin{aligned}
    \cH^\mathrm{sym}&(\schrplan{\mu}{\nu}{T},\schrplan{\bar\mu}{\bar\nu}{T})\\
    &=\int \log f\,\De(\mu-\bar\mu)+\int\log\bar f\,\De(\bar\mu-\mu)+\int\log g\,\De(\nu-\bar\nu)+\int\log\bar g\,\De(\bar\nu-\nu)\\
    & \leq \norm{\nabla\log f}_{L^2(\mu)}\norm{\mu-\bar\mu}_{\dot{H}^{-1}(\mu)}+\norm{\nabla\log\bar f }_{L^2(\bar\mu)}\norm{\bar\mu-\mu}_{\dot{H}^{-1}(\bar\mu)} \\
    & +\norm{\nabla\log  g}_{L^2(\nu)}\norm{\nu-\bar\nu}_{\dot{H}^{-1}(\nu)} + \norm{\nabla\log \bar g}_{L^2(\bar\nu)}\norm{\bar\nu-\nu}_{\dot{H}^{-1}(\bar\nu)}\,,
\end{aligned}
\end{equation*}
and that we can bound the $L^2$-norm of the gradient of the potentials in terms of the Fisher information, e.g.\
\[
\norm{\nabla\log f}_{L^2(\mu)}\leq \norm{\nabla\log\frac{\De\mu}{\De\frm}}_{L^2(\mu)}+\norm{\nabla\log P_T g}_{L^2(\mu)}=\sqrt{\cI(\mu)}+\norm{\nabla\log P_T g}_{L^2(\mu)}
\]
and analogously for the other three summands.
\end{proof}

Before moving to the proof of the stability estimates presented in \eqref{eq:stab:cost} and \eqref{eq:fish:stab:cost}, we need another technical result given by the following 

\begin{lemma}\label{disslemma}
Let $(M,\sfd, \frm)$ satisfy \eqref{CD} and \eqref{I}. For any $(\mu,\nu)$ and $(\bar\mu,\bar\nu)$ satisfying \eqref{H1}, \eqref{H2} and such that $\norm{\mu-\bar\mu}_{\dot{H}^{-1}(\mu)},\,\norm{\bar\mu-\mu}_{\dot{H}^{-1}(\bar\mu)},\,\norm{\nu-\bar\nu}_{\dot{H}^{-1}(\nu)}$, and $\norm{\bar\nu-\nu}_{\dot{H}^{-1}(\bar\nu)}$ are all finite and with  $\cH^\mathrm{sym}(\mu,\bar\mu),\,\cH^\mathrm{sym}(\nu,\bar\nu)<\oo$, it holds
\begin{align*}
    \int \log\frac{P_T  f}{P_T \bar f}\,\De\nu\leq \int\log\frac{ f}{\bar f}\,\De\mu\qquad\text{and}\qquad\int \log\frac{P_T g}{P_T \bar g}\,\De\mu\leq \int\log\frac{g}{\bar g}\,\De\nu\,.
\end{align*}
Analogous bounds hold when we exchange the roles between $(\mu,\nu,f,g)$ and $(\bar\mu,\bar\nu,\bar f,\bar g)$.
\end{lemma}

\begin{proof}
Since $\cH(\nu|\bar\nu)\leq \cH(\schrplan{\mu}{\nu}{T}|\schrplan{\mu}{\bar\nu}{T})$, by means of the $fg$-decomposition the inequality reads as
\[\int \log\frac{P_T  f}{P_T \bar f}\De \nu+\int\log\frac{g}{\bar g}\De\nu\leq \int\log\frac{ f}{\bar f}\De \mu+\int\log\frac{ g}{\bar g}\De \nu\,,\]
which yields the first inequality. The other bound can be proven in the same way.
\end{proof}

By following the same line of reasoning applied for the stability of the optimal plans, we can now deduce the following

\begin{prop}\label{stab:entropic:cost}
Let $(M,\sfd, \frm)$ satisfy \eqref{CD} and \eqref{I}. For any $(\mu,\nu)$ and $(\bar\mu,\bar\nu)$ satisfying \eqref{H1}, \eqref{H2}, and such that $\norm{\mu-\bar\mu}_{\dot{H}^{-1}(\mu)},\,\norm{\bar\mu-\mu}_{\dot{H}^{-1}(\bar\mu)},\,\norm{\nu-\bar\nu}_{\dot{H}^{-1}(\nu)}$, and $\norm{\bar\nu-\nu}_{\dot{H}^{-1}(\bar\nu)}$ are all finite and with  $\cH^\mathrm{sym}(\mu,\bar\mu),\,\cH^\mathrm{sym}(\nu,\bar\nu)<\oo$, it holds
\begin{align*}
\cS_T(\mu,\nu)-\cS_T(\bar\mu,\bar\nu)\leq T\biggl[\cH(\bar\mu|\mu)\wedge \cH(\bar\nu|\nu)\biggr]
&+\frac{T}{\sqrt{E_{2\kappa}(T)}}\,\sqrt{\cC_T(\mu,\nu)-\cH(\mu|\frm)}\norm{\mu-\bar\mu}_{\dot{H}^{-1}(\mu)}\\
&+\frac{T}{\sqrt{E_{2\kappa}(T)}}\,\sqrt{\cC_T(\mu,\nu)-\cH(\nu|\frm)}\norm{\nu-\bar\nu}_{\dot{H}^{-1}(\nu)},
\end{align*}
where $E_{2\kappa}$ is defined as in \eqref{eq:curvature-factor}, from which it follows
\begin{equation*} 
\begin{aligned}
     &\abs{\cS_T(\bar\mu,\bar\nu)-\cS_T(\mu,\nu)}\leq T\biggl[\cH^\mathrm{sym}(\mu,\bar\mu)\wedge \cH^\mathrm{sym}(\nu,\bar\nu)\biggr]\\
    &\qquad+\frac{T}{\sqrt{E_{2\kappa}(T)}}\biggl[\sqrt{\cC_T(\mu,\nu)-\cH(\mu|\frm)}\norm{\mu-\bar\mu}_{\dot{H}^{-1}(\mu)}+\sqrt{\cC_T(\mu,\nu)-\cH(\nu|\frm)}\norm{\nu-\bar\nu}_{\dot{H}^{-1}(\nu)}\biggr]\\
&\qquad+\frac{T}{\sqrt{E_{2\kappa}(T)}}\biggl[\sqrt{\cC_T(\bar\mu,\bar\nu)-\cH(\bar\mu|\frm)}\norm{\bar\mu-\mu}_{\dot{H}^{-1}(\bar\mu)}+\sqrt{\cC_T(\bar\mu,\bar\nu)-\cH(\bar\nu|\frm)}\norm{\bar\nu-\nu}_{\dot{H}^{-1}(\bar\nu)}\biggr]\,.
\end{aligned}
\end{equation*}
\end{prop}

\begin{proof}
First of all, since the assumptions of Corollary \ref{cor:svolta} are met, the following computations are all well defined. With this said, let us start by noticing that
\[
\cS_T(\mu,\nu)=T\,\cC_T(\mu,\nu)-T\,\cH(\mu|\frm)-T\,\cH(\nu|\frm)=-T\int\log P_T g\,\De\mu-T\int\log P_T f\,\De\nu\,,
\]
and therefore 
\begin{align*}
    \cS_T(\bar\mu,\bar\nu)-\cS_T(\mu,\nu)=T\int\log P_T g\,\De\mu+T\int\log P_T f\,\De\nu-T\int\log P_T \bar g\,\De\bar\mu-T\int\log P_T \bar f\,\De\bar\nu\\
    =T\int\log\frac{ P_T g}{P_T\bar g}\De\mu+T\int\log \frac{P_T f}{P_T\bar f}\De\nu-T\int\log P_T\bar g\,\De(\bar\mu-\mu)-T\int\log P_T\bar f\,\De(\bar\nu-\nu)\,.
\end{align*}
Applying Lemma \ref{disslemma} we deduce that
\begin{equation}\label{giappo}
\begin{aligned}
    \cS_T(\bar\mu,\bar\nu)-\cS_T(\mu,\nu) & \leq T\int\log\frac{ P_T g}{P_T\bar g} \De\mu+T\int\log \frac{ f}{\bar f}\De\mu\\
    &\qquad\qquad\qquad-T\int\log P_T\bar g\,\De(\bar\mu-\mu)-T\int\log P_T\bar f\,\De(\bar\nu-\nu)\\
    &=T\int\log\frac{\De\mu}{\De\bar\mu}\De\mu-T\int\log P_T\bar g\,\De(\bar\mu-\mu)-T\int\log P_T\bar f\,\De(\bar\nu-\nu)\\
    &=T\cH(\mu|\bar\mu)-T\int\log P_T\bar g\,\De(\bar\mu-\mu)-T\int\log P_T\bar f\,\De(\bar\nu-\nu)\,.
\end{aligned}\end{equation}
Notice that above we could have used Lemma \ref{disslemma} on the other integral $\int\log\frac{ P_T g}{P_T\bar g}\De\mu$  and therefore we would have got $\cH(\nu|\bar\nu)$ instead of $\cH(\mu|\bar\mu)$. Therefore we can state that
\begin{align*}
\cS_T(\bar\mu,\bar\nu)-\cS_T(\mu,\nu) & \leq T\Bigl[\cH(\mu|\bar\mu)\wedge \cH(\nu|\bar\nu)\Bigr]-T\int\log P_T\bar g\,\De(\bar\mu-\mu)-T\int\log P_T\bar f\,\De(\bar\nu-\nu)\\
& \leq T\Bigl[\cH(\mu|\bar\mu)\wedge \cH(\nu|\bar\nu)\Bigr]+T\norm{\nabla\log P_T\bar g}_{L^2(\bar\mu)}\norm{\bar\mu-\mu}_{\dot{H}^{-1}(\bar\mu)} \\
& \qquad\qquad +T\norm{\nabla\log P_T \bar f}_{L^2(\bar\nu)}\norm{\bar\nu-\nu}_{\dot{H}^{-1}(\bar\nu)}.
\end{align*}
By exchanging the roles between $(\mu,\nu)$ and $(\bar\mu,\bar\nu)$ and by arguing in the same way we also get
\begin{align*}\cS_T(\mu,\nu)-\cS_T(\bar\mu,\bar\nu)\leq T\Bigl[\cH(\bar\mu|\mu)\wedge \cH(\bar\nu|\nu)\Bigr] & + T\norm{\nabla\log P_T g}_{L^2(\mu)}\norm{\mu-\bar\mu}_{\dot{H}^{-1}(\mu)} \\
& + T\norm{\nabla\log P_T  f}_{L^2(\nu)}\norm{\nu-\bar\nu}_{\dot{H}^{-1}(\nu)},\end{align*}
and therefore we deduce 
\begin{align*}
    \abs{\cS_T(\bar\mu,\bar\nu)-\cS_T(\mu,\nu)}\leq&\, T\Bigl[\cH^\mathrm{sym}(\mu,\bar\mu)\wedge \cH^\mathrm{sym}(\nu,\bar\nu)\Bigr]
    \\
    +&\,T\norm{\nabla\log P_T g}_{L^2(\mu)}\norm{\mu-\bar\mu}_{\dot{H}^{-1}(\mu)}
    +T\norm{\nabla\log P_T  f}_{L^2(\nu)}\norm{\nu-\bar\nu}_{\dot{H}^{-1}(\nu)}\\
    +&\,T\norm{\nabla\log P_T\bar g}_{L^2(\bar\mu)}\norm{\bar\mu-\mu}_{\dot{H}^{-1}(\bar\mu)}
    +T\norm{\nabla\log P_T \bar f}_{L^2(\bar\nu)}\norm{\bar\nu-\nu}_{\dot{H}^{-1}(\bar\nu)}.
\end{align*}
Given the above, the thesis follows from the corrector estimates (cf.\ Proposition \ref{lemma:gia}).
\end{proof}

\begin{remark} Notice that in the cases where we change just one marginal (e.g. $\mu=\bar\mu$),  we can get rid of the relative entropy between the marginals in the above result. Therefore in order to get a bound without the symmetric relative entropy of the marginals it is enough to consider the case where one of the marginals is frozen by means of the trivial inequality
\begin{equation}
   \abs{\cS_T(\bar\mu,\bar\nu)-\cS_T(\mu,\nu)}\leq \abs{\cS_T(\bar\mu,\bar\nu)-\cS_T(\mu,\bar\nu)}+\abs{\cS_T(\mu,\bar\nu)-\cS_T(\mu,\nu)}\,.
\end{equation} 
\end{remark}

\medskip

Let us conclude by showing that, under the finite Fisher information assumption, we may write a stability estimate for the Schr\"odinger costs without involving the symmetric relative entropies on the right-hand side.

\begin{prop}\label{stab:fish:finite}
Let $(M,\sfd, \frm)$ satisfy \eqref{CD} and \eqref{I}. For any $(\mu,\nu)$ and $(\bar\mu,\bar\nu)$ satisfying \eqref{H1} and \eqref{H2} we have
\begin{equation*}\begin{aligned}
     \abs{\cC_T(\bar\mu,\bar\nu)-\cC_T(\mu,\nu)}\leq\,& \frac{1}{\sqrt{E_{2\kappa}(T)}}\biggl[\sqrt{\cI(\mu)}+\sqrt{\cC_T(\mu,\nu)-\cH(\mu|\frm)}\biggr]\norm{\mu-\bar\mu}_{\dot{H}^{-1}(\mu)}\\
+&\frac{1}{\sqrt{E_{2\kappa}(T)}}\biggl[\sqrt{\cI(\bar\mu)}+\sqrt{\cC_T(\bar\mu,\bar\nu)-\cH(\bar\mu|\frm)}\biggr]\norm{\bar\mu-\mu}_{\dot{H}^{-1}(\bar\mu)}\\
+& \frac{1}{\sqrt{E_{2\kappa}(T)}}\biggl[\sqrt{\cI(\nu)}+\sqrt{\cC_T(\mu,\nu)-\cH(\nu|\frm)}\biggr]\norm{\nu-\bar\nu}_{\dot{H}^{-1}(\nu)}\\
+&\frac{1}{\sqrt{E_{2\kappa}(T)}}\biggl[\sqrt{\cI(\bar\nu)}+\sqrt{\cC_T(\bar\mu,\bar\nu)-\cH(\bar\nu|\frm)}\biggr]\norm{\bar\nu-\nu}_{\dot{H}^{-1}(\bar\nu)}\,,
 \end{aligned}\end{equation*}
where $E_{2\kappa}$ is defined as in \eqref{eq:curvature-factor}.
\end{prop}

\begin{proof}
The proof runs like the one given in Proposition \ref{stab:entropic:cost} observing that, similarly to \eqref{giappo}, we can write
\[\cC_T(\bar\mu,\bar\nu)-\cC_T(\mu,\nu)\leq \int\log \bar f\,\De(\bar\mu-\mu)+\int \log\bar g\,\De(\bar\nu-\nu)-\biggl[\cH(\mu|\bar\mu)\vee\cH(\nu|\bar\nu)\biggr] \]
and therefore it holds
\[\begin{split}
    \abs{\cC_T(\bar\mu,\bar\nu)-\cC_T(\mu,\nu)}\leq \norm{\nabla\log f}_{L^2(\mu)}\norm{\mu-\bar\mu}_{\dot{H}^{-1}(\mu)}+\norm{\nabla\log\bar f }_{L^2(\bar\mu)}\norm{\bar\mu-\mu}_{\dot{H}^{-1}(\bar\mu)} \\
    +\norm{\nabla\log  g}_{L^2(\nu)}\norm{\nu-\bar\nu}_{\dot{H}^{-1}(\nu)}
    +\norm{\nabla\log \bar g}_{L^2(\bar\nu)}\norm{\bar\nu-\nu}_{\dot{H}^{-1}(\bar\nu)}\,.
\end{split}\]
\end{proof}
Notice that in the above result we are able to consider the stability directly between the Schr\"odinger costs $\cC_T(\mu,\nu)$ and $\cC_T(\bar\mu,\bar\nu)$. This is indeed due to the fact that we are working in a finite Fisher information setting.

\subsection{Application to EOT with quadratic cost}\label{app:EOT}
 In this section we  translate the stability results stated in Theorem \ref{stab:piani} and Proposition \ref{stab:entropic:cost} to the Euclidean EOT setting with quadratic cost
\begin{equation}\label{quadratic:EOT}\cS^\varepsilon(\mu,\nu)\coloneqq \inf_{\pi\in\Pi(\mu,\nu)}\int |x-y|^2\,\De\pi+\varepsilon\,\cH(\pi|\mu\otimes\nu)\,,\qquad\varepsilon>0\,.\end{equation}
We will manage to do that under sufficiently general conditions, so that our results will apply to any couple of marginals $\mu,\nu\in\cP_2(\RD)$ with finite relative entropy w.r.t.\ the Lebesgue measure $\cL_d$ on $\RD$ and with densities w.r.t.\ $\cL_d$ locally bounded away from $0$ on their support and such that $\mu(\partial \supp(\mu))=\nu(\partial \supp(\nu))=0$; or with bounded and compactly supported densities. 
For later reference, let $\EOTplan{\mu}{\nu}{\varepsilon}$ denote the minimizer (which exists and is unique since $\mu,\nu\in\cP_2(\RD)$, cf.\ \cite[Theorem 4.2]{Marcel:notes}).

In what follows we are going to consider a Schr\"odinger problem equivalent to the above \eqref{quadratic:EOT}, where the underlying stochastic dynamics is given by the law of the Ornstein–Uhlenbeck process
\begin{equation}\label{OU}
\De X_t=-\kappa X_t\De t+\sqrt{2}\,\De B_t\,,
\end{equation}
$(B_t)_t$ being a $d$-dimensional Brownian motion and $\kappa>0$ a curvature parameter (whose value will be specified later). The above SDE admits as unique invariant measure the Gaussian distribution $\frm\sim\cN\bigl(0,\kappa^{-1}\Id\bigr)$, i.e.\ $\De\frm(x)\sim e^{-\frac{\kappa}{2}\abs{x}^2}\De x$, which satisfies the curvature condition $\CD(\kappa,\oo)$ and hence a logarithmic Sobolev inequality with parameter $\kappa^{-1}$ \cite[Corollary 5.7.1]{bakry2013analysis}. Then, Herbst's argument \cite[Proposition 5.4.1]{bakry2013analysis} implies that $\frm$ satisfies the integrability condition \eqref{I} for any $r<\kappa/2 $. Notice that for any choice of $\kappa>0$ we have
  \begin{equation}\label{rel:entropie}\cH(\mu|\frm)=\cH(\mu|\cL_d)+\frac{\kappa}{2}\,M_2(\mu)+\frac{d}{2}\log\frac{2\pi}{\kappa}<+\oo \end{equation}
  and similarly $\cH(\nu|\frm)<+\oo$.

Now, let us consider a final time in SP such that $\frac{\varepsilon\kappa}{4}=\sinh(\kappa T)$, i.e.\
\begin{equation}\label{T:def}T\coloneqq \frac{1}{\kappa}\log\left(\frac{\varepsilon\kappa}{4}+\sqrt{\frac{\varepsilon^2\kappa^2}{16}+1}\right)\,\end{equation}
and let $\rmR_{0,T}=\cL(X_0,X_T)$ denote the joint law at times $0$ and $T$ of the Ornstein–Uhlenbeck process solving \eqref{OU} started at the invariant measure $\frm$, whose density is given by the transition kernel 
\[
\hp_T(x,y)=\frac{\De\rmR_{0,T}}{\De(\frm\otimes\frm)}(x,y)=\frac{1}{(1-e^{-2\kappa T})^{\frac{d}{2}}}\,\exp\left\{-\frac{\abs{x}^2-2e^{\kappa T}\,x\cdot y+\abs{y}^2}{\frac{2}{\kappa}(e^{2\kappa T}-1)}\right\}\,.
\]
Then, thanks to our choice of $T$, for any coupling $\pi\in\Pi(\mu,\nu)$ it holds
\begin{equation*}
\begin{aligned}
&\int_{\R^{2d}}\abs{x-y}^2\De\pi+\varepsilon\,\cH(\pi|\mu\otimes\nu)=M_2(\mu)+ M_2(\nu)-2\int_{\R^{2d}}x\cdot y\,\De\pi+\varepsilon\,\cH(\pi|\mu\otimes\nu)\\
&=-\frac{d\varepsilon}{2}\log(1-e^{-2\kappa T})+(1-e^{-\kappa T})\bigl(M_2(\mu)+ M_2(\nu)\bigr)-\varepsilon\int_{\R^{2d}}\log \hp_T(x,y)\,\De\pi+\varepsilon\,\cH(\pi|\mu\otimes\nu)\\
&=-\frac{d\varepsilon}{2}\log(1-e^{-2\kappa T})+(1-e^{-\kappa T})\bigl(M_2(\mu)+ M_2(\nu)\bigr)+\varepsilon\left[\cH(\pi|\mu\otimes\nu)-\int_{\R^{2d}}\log \frac{\De \rmR}{\De(\frm\otimes\frm)}\,\De\pi\right]\\
&=-\frac{d\varepsilon}{2}\log(1-e^{-2\kappa T})+(1-e^{-\kappa T})\bigl(M_2(\mu)+ M_2(\nu)\bigr)+\varepsilon\,\cH(\pi|\rmR_{0,T})-\varepsilon\,\cH(\mu|\frm)-\varepsilon\,\cH(\nu|\frm)\,.
\end{aligned}
\end{equation*}
Therefore the above EOT problem \eqref{quadratic:EOT} has the same minimizer of SP with reference given by the above-chosen $\rmR_{0,T}$ and their values are linked according to the following identity 
\begin{equation}\label{reason:not:ill}
\begin{aligned}\cS^\varepsilon(\mu,\nu)=&\,\varepsilon\,\cC_T(\mu,\nu)-\varepsilon\,\cH(\mu|\frm)-\varepsilon\,\cH(\nu|\frm)-\frac{d\varepsilon}{2}\log(1-e^{-2\kappa T})+(1-e^{-\kappa T})\bigl(M_2(\mu)+ M_2(\nu)\bigr)\\
=&\,\frac{\varepsilon}{T}\,\cS_T(\mu,\nu)-\frac{d\varepsilon}{2}\log(1-e^{-2\kappa T})+(1-e^{-\kappa T})\bigl(M_2(\mu)+ M_2(\nu)\bigr)\,.
\end{aligned}
\end{equation}
Since we are interested in stability results, take also $\bar\mu,\,\bar\nu\in \cP_2(\RD)$ such that
\[
\cH(\bar\mu|\cL_d),\,\cH(\bar\nu|\cL_d)<+\oo\quad\text{ and }\quad\cH^\mathrm{sym}(\mu,\bar\mu),\, \cH^\mathrm{sym}(\nu,\bar\nu)<+\oo\,.
\]
We are now ready to apply Proposition \ref{stab:entropic:cost} and get
\begin{equation}\label{troppi:parametri}
\begin{aligned}
    |\cS^\varepsilon(\bar\mu,\bar\nu)&-\cS^\varepsilon(\mu,\nu)|\\\leq& \frac{\varepsilon}{T}\abs{\cS_T(\bar\mu,\bar\nu)-\cS_T(\mu,\nu)}
    +\left(1-e^{-\kappa T}\right)\Bigl(\abs{M_2(\bar\mu)-M_2(\mu)}+\abs{ M_2(\bar\nu)-M_2(\nu)}\Bigr)\\
    \leq&\, \varepsilon\Bigl[\cH^\mathrm{sym}(\mu,\bar\mu)\wedge \cH^\mathrm{sym}(\nu,\bar\nu)\Bigr]+\left(1-e^{-\kappa T}\right)\Bigl(\abs{M_2(\bar\mu)-M_2(\mu)}+\abs{M_2(\bar\nu)-M_2(\nu)}\Bigr)\\
    +&\,\frac{\varepsilon}{\sqrt{E_{2\kappa}(T)}}\biggl[\sqrt{\cC_T(\mu,\nu)-\cH(\mu|\frm)}\norm{\mu-\bar\mu}_{\dot{H}^{-1}(\mu)}+\sqrt{\cC_T(\mu,\nu)-\cH(\nu|\frm)}\norm{\nu-\bar\nu}_{\dot{H}^{-1}(\nu)}\biggr]\\
   +&\,\frac{\varepsilon}{\sqrt{E_{2\kappa}(T)}}\biggl[\sqrt{\cC_T(\bar\mu,\bar\nu)-\cH(\bar\mu|\frm)}\norm{\bar\mu-\mu}_{\dot{H}^{-1}(\bar\mu)}+\sqrt{\cC_T(\bar\mu,\bar\nu)-\cH(\bar\nu|\frm)}\norm{\bar\nu-\nu}_{\dot{H}^{-1}(\bar\nu)}\biggr]
   \,.
\end{aligned}
\end{equation}
Since for any coupling $\pi\in\Pi(\bar\mu,\mu)$ we can write
\begin{align*}
    \abs{M_2(\bar\mu)-M_2(\mu)}=&\,\abs{\int \abs{x}^2-\abs{y}^2\,\De\pi}\leq \int\abs{x(x-y)}\De\pi+\int\abs{y(x-y)}\De\pi\\
    \leq&\, \left(\sqrt{M_2(\bar\mu)}+\sqrt{M_2(\mu)}\right)\left(\int\abs{x-y}^2\De\pi\right)^{\frac12}\,,
\end{align*}
 by minimizing the right-hand side over $\pi\in\Pi(\mu,\nu)$ we end up with
\begin{align*}
    \abs{M_2(\bar\mu)-M_2(\mu)}\leq \left(\sqrt{M_2(\bar\mu)}+\sqrt{M_2(\mu)}\right)\cW_2(\bar\mu,\mu)
\end{align*}
and \eqref{troppi:parametri} reads as
\begin{equation}\label{troppi:costi}\begin{aligned}
    |\cS^\varepsilon&(\bar\mu,\bar\nu)-\cS^\varepsilon(\mu,\nu)|\leq \varepsilon\Bigl[\cH^\mathrm{sym}(\mu,\bar\mu)\wedge \cH^\mathrm{sym}(\nu,\bar\nu)\Bigr]\\&+\left(1-e^{-\kappa T}\right)\left(\sqrt{M_2(\bar\mu)}+\sqrt{M_2(\mu)}\right)\cW_2(\bar\mu,\mu)+\left(1-e^{-\kappa T}\right)\left(\sqrt{M_2(\bar\nu)}+\sqrt{M_2(\nu)}\right)\cW_2(\bar\nu,\nu)\\
    &+\frac{\varepsilon}{\sqrt{E_{2\kappa}(T)}}\biggl[\sqrt{\cC_T(\mu,\nu)-\cH(\mu|\frm)}\norm{\mu-\bar\mu}_{\dot{H}^{-1}(\mu)}+\sqrt{\cC_T(\mu,\nu)-\cH(\nu|\frm)}\norm{\nu-\bar\nu}_{\dot{H}^{-1}(\nu)}\biggr]\\
   &+\frac{\varepsilon}{\sqrt{E_{2\kappa}(T)}}\biggl[\sqrt{\cC_T(\bar\mu,\bar\nu)-\cH(\bar\mu|\frm)}\norm{\bar\mu-\mu}_{\dot{H}^{-1}(\bar\mu)}+\sqrt{\cC_T(\bar\mu,\bar\nu)-\cH(\bar\nu|\frm)}\norm{\bar\nu-\nu}_{\dot{H}^{-1}(\bar\nu)}\biggr]
   \,.
\end{aligned}\end{equation}
Similarly, since $\schrplan{\mu}{\nu}{T}$ is also the optimizer of the EOT problem \eqref{quadratic:EOT}, we can translate Theorem \ref{stab:piani} into a stability result between the optimal plans for \eqref{quadratic:EOT}, that is, between $\EOTplan{\mu}{\nu}{\varepsilon}$ and $\EOTplan{\bar\mu}{\bar\nu}{\varepsilon}$:
\begin{equation}\label{troppi:piani}\begin{aligned}\varepsilon\,&\cH^\mathrm{sym}(\EOTplan{\mu}{\nu}{\varepsilon},\EOTplan{\bar\mu}{\bar\nu}{\varepsilon})=\frac{\varepsilon}{T}\,T\,\cH^\mathrm{sym}(\schrplan{\mu}{\nu}{T},\schrplan{\bar\mu}{\bar\nu}{T})\leq\, \varepsilon\,\cH^\mathrm{sym}(\mu,\bar\mu)+ \varepsilon\,\cH^\mathrm{sym}(\nu,\bar\nu)\\
  +&\,\frac{\varepsilon}{\sqrt{E_{2\kappa}(T)}}\biggl[\sqrt{\cC_T(\mu,\nu)-\cH(\mu|\frm)}\norm{\mu-\bar\mu}_{\dot{H}^{-1}(\mu)}+\sqrt{\cC_T(\mu,\nu)-\cH(\nu|\frm)}\norm{\nu-\bar\nu}_{\dot{H}^{-1}(\nu)}\biggr]\\
   +&\,\frac{\varepsilon}{\sqrt{E_{2\kappa}(T)}}\biggl[\sqrt{\cC_T(\bar\mu,\bar\nu)-\cH(\bar\mu|\frm)}\norm{\bar\mu-\mu}_{\dot{H}^{-1}(\bar\mu)}+\sqrt{\cC_T(\bar\mu,\bar\nu)-\cH(\bar\nu|\frm)}\norm{\bar\nu-\nu}_{\dot{H}^{-1}(\bar\nu)}\biggr]
   \,.
\end{aligned}\end{equation}
Given the above bounds one can now get explicit estimates by choosing the curvature parameter $\kappa>0$.

\begin{remark}[Small-noise limit] Let us point out that the above bounds are stable when $\varepsilon\to 0$. Indeed \eqref{T:def} guarantees us that the small-noise limit $\varepsilon\to 0$ is equivalent to the small-time limit $T\to0$ and moreover the ratio $\frac{\varepsilon}{T}$ stays finite since
\begin{align*}
    \frac{\varepsilon}{T}=\frac{\varepsilon\kappa}{\kappa T}=\frac{4\sinh(\kappa T)}{\kappa T}=2\,\frac{e^{\kappa T}-e^{-\kappa T}}{\kappa T}\,\to \,4\,,\quad \text{as $T\to 0$.}
\end{align*}
Therefore in the small-noise limit, from \eqref{eq:small:cost} we deduce for \eqref{quadratic:EOT} the following stability results
\begin{equation}\label{small:noise:EOT}\begin{aligned}
    \limsup_{\varepsilon\to 0}\abs{\cS^\varepsilon(\bar\mu,\bar\nu)-\cS^\varepsilon(\mu,\nu)}\leq& 2\,\cW_2(\mu,\nu)\,\biggl[\norm{\mu-\bar\mu}_{\dot{H}^{-1}(\mu)}+\norm{\nu-\bar\nu}_{\dot{H}^{-1}(\nu)}\biggr]\\
    +&2\,\cW_2(\bar\mu,\bar\nu)\,\biggl[\norm{\bar\mu-\mu}_{\dot{H}^{-1}(\bar\mu)}+\norm{\bar\nu-\nu}_{\dot{H}^{-1}(\bar\nu)}\biggr]\,,\\
    \limsup_{\varepsilon\to 0}\,\varepsilon\,\cH^\mathrm{sym}(\EOTplan{\mu}{\nu}{\varepsilon},\EOTplan{\bar\mu}{\bar\nu}{\varepsilon})\leq& 2\,\cW_2(\mu,\nu)\,\biggl[\norm{\mu-\bar\mu}_{\dot{H}^{-1}(\mu)}+\norm{\nu-\bar\nu}_{\dot{H}^{-1}(\nu)}\biggr]\\
    +&2\,\cW_2(\bar\mu,\bar\nu)\,\biggl[\norm{\bar\mu-\mu}_{\dot{H}^{-1}(\bar\mu)}+\norm{\bar\nu-\nu}_{\dot{H}^{-1}(\bar\nu)}\biggr]\,,
\end{aligned}\end{equation}
which agree, up to a scaling constant, with Remark \ref{small:time:rem} in the Schr\"odinger setting.
\end{remark}

Notice that the above small-noise limit is independent from the choice of $\kappa>0$. This suggests us to take the limit $\kappa\downarrow 0$ directly in \eqref{troppi:costi} and \eqref{troppi:piani}. Indeed from \eqref{T:def} it follows $ T\to \,\varepsilon/4$ as $\kappa\downarrow 0$.
Furthermore, by rearranging \eqref{reason:not:ill} we notice that
\[
\begin{split}
T\,\cC_T(\mu,\nu)-T\,\cH(\mu|\frm) & =\frac{T}{\varepsilon}\, \cS^\varepsilon(\mu,\nu)-\frac{T}{\varepsilon}(1-e^{-\kappa T}) \biggl(M_2(\mu)+M_2(\nu)\biggr) +T\,\cH(\nu|\frm)\\
& \qquad +\frac{d\,T}{2}\log\left(1-e^{-2\kappa T}\right)\\
& =\frac{T}{\varepsilon}\, \cS^\varepsilon(\mu,\nu)-\frac{T}{\varepsilon}(1-e^{-\kappa T})\biggl(M_2(\mu)+M_2(\nu)\biggr) +T\,\cH(\nu|\cL_d) \\
& \qquad +\frac{\kappa T}{2}M_2(\nu) +\frac{d\,T}{2}\log\left(4\pi T\,\frac{1-e^{-2\kappa T}}{2\kappa T}\right)
\end{split}
\]
and therefore we have
\[\lim_{\kappa\to 0}\, T\,\cC_T(\mu,\nu)-T\,\cH(\mu|\frm)=\frac{1}{4}\,\cS^\varepsilon(\mu,\nu)+\frac{\varepsilon}{4}\,\cH(\mu|\cL_d)+\frac14\,C_\varepsilon\,,\]
where $C_\varepsilon=\frac{d\varepsilon}{2}\log(4\pi\varepsilon)$.
As a consequence, by taking the limit in the right-hand sides of \eqref{troppi:costi} and \eqref{troppi:piani}, we deduce that
\begin{equation*}\begin{aligned}
    \abs{\cS^\varepsilon(\bar\mu,\bar\nu)-\cS^\varepsilon(\mu,\nu)}\leq \varepsilon\Bigl[\cH^\mathrm{sym}(\mu,\bar\mu)\wedge \cH^\mathrm{sym}(\nu,\bar\nu)&\Bigr]\\
    +2 \biggl[\sqrt{\cS^\varepsilon(\mu,\nu)+\varepsilon\,\cH(\nu|\cL_d)+C_\varepsilon}\norm{\mu-\bar\mu}_{\dot{H}^{-1}(\mu)}&+\sqrt{\cS^\varepsilon(\mu,\nu)+\varepsilon\,\cH(\mu|\cL_d)+C_\varepsilon}\norm{\nu-\bar\nu}_{\dot{H}^{-1}(\nu)}\biggr]\\
    +2\biggl[\sqrt{\cS^\varepsilon(\bar\mu,\bar\nu)+\varepsilon\,\cH(\bar\nu|\cL_d)+C_\varepsilon}\norm{\bar\mu-\mu}_{\dot{H}^{-1}(\bar\mu)}&+\sqrt{\cS^\varepsilon(\bar\mu,\bar\nu)+\varepsilon\,\cH(\bar\mu|\cL_d)+C_\varepsilon}\norm{\bar\nu-\nu}_{\dot{H}^{-1}(\bar\nu)}\biggr]
\end{aligned}\end{equation*}
and
\begin{equation*}\begin{aligned}
    \varepsilon\,\cH^\mathrm{sym}(\EOTplan{\mu}{\nu}{\varepsilon},\EOTplan{\bar\mu}{\bar\nu}{\varepsilon})\leq\, \varepsilon\,\cH^\mathrm{sym}(\mu,\bar\mu)+ \varepsilon\,\cH^\mathrm{sym}(\nu,\bar\nu)&\,\\
    +2 \biggl[\sqrt{\cS^\varepsilon(\mu,\nu)+\varepsilon\,\cH(\nu|\cL_d)+C_\varepsilon}\norm{\mu-\bar\mu}_{\dot{H}^{-1}(\mu)}&+\sqrt{\cS^\varepsilon(\mu,\nu)+\varepsilon\,\cH(\mu|\cL_d)+C_\varepsilon}\norm{\nu-\bar\nu}_{\dot{H}^{-1}(\nu)}\biggr]\\
    +2\biggl[\sqrt{\cS^\varepsilon(\bar\mu,\bar\nu)+\varepsilon\,\cH(\bar\nu|\cL_d)+C_\varepsilon}\norm{\bar\mu-\mu}_{\dot{H}^{-1}(\bar\mu)}&+\sqrt{\cS^\varepsilon(\bar\mu,\bar\nu)+\varepsilon\,\cH(\bar\mu|\cL_d)+C_\varepsilon}\norm{\bar\nu-\nu}_{\dot{H}^{-1}(\bar\nu)}\biggr]\,.
\end{aligned}\end{equation*}
Notice that from the above bounds,  we get the validity of \eqref{small:noise:EOT} in the small-noise limit $\varepsilon\downarrow 0$,.

\appendix

\section{A log-integrability Lemma}\label{app:log:int}

Let us start by recalling a few interesting facts about Orlicz spaces. 
Let $\frq$ be a probability measure on a measurable space $(\Omega,\Sigma)$ and consider the Young functions 
\begin{align*}
\theta(t) \coloneqq e^t-1\,, \qquad \theta^*(s) = \begin{cases}
s\,\log s-s+1\quad&\text{if }s>0,\\
1\quad&\text{if }s=0\,.
\end{cases}
\end{align*}
where $\theta^*(s)\coloneqq \sup_{t\in\R}\bigl\{st-\theta(t)\bigr\}$. The Orlicz space associated to the Young function $\theta$, denoted by $L_\theta(\frq)$, is defined as the space of measurable functions $f : \Omega \to \R$ with finite Luxemburg norm
\[
\norm{f}_\theta\coloneqq \inf \left\{b>0\colon\int\theta\left(\frac{\abs{f}}{b}\right)\De\frq\leq 1\right\}\,.
\]
To be precise we should make explicit reference to the underlying probability $\frq$, when talking about Luxemburg norms; however we will omit such reference when it is clear from the context.

When dealing with finite relative entropies, the Orlicz space associated to $\theta^*$ (defined analogously as for $\theta$) plays a natural role. Indeed, for any probability $\frp$ with $\cH(\frp|\frq) < +\infty$ we have
\begin{align*}
    \norm{\frac{\De\frp}{\De\frq}}_{\theta^*}= &\inf \left\{b>0\colon\int\theta^*\left(\frac1b\,\frac{\De\frp}{\De\frq}\right)\De\frq\leq 1\right\}=\inf \left\{b>0\colon\frac1b\cH(\frp|\frq)-\frac1b\left(1+\log b\right)+1\leq 1\right\}\\
    =&\inf \left\{b>0\colon \cH(\frp|\frq)-1\leq \log b\right\}=e^{\cH(\frp|\frq)-1}\,.
\end{align*}
Finally, let us recall that for any $f\in L_\theta(\frq)$ and $g\in L_{\theta^*}(\frq)$ we have
\begin{equation}\label{Orlicz:Young}
\int\abs{fg}\De\frq\leq 2\norm{f}_\theta\,\norm{g}_{\theta^*}\,,
\end{equation}
which is a consequence of the trivial inequality
\[
st\leq \theta(t)+\theta^*(s)\quad\forall s,t>0\,.
\]

After this digression, let us prove the following lemma, which is pivotal in our reasoning since it allows to deduce the log-integrability of any positive measurable function under the sole hypotheses of finite relative entropy and some positive and negative integrability of the same function. 

\begin{lemma}\label{log:integrability}
Let $h$ be a positive measurable function and $\frp,\frq \in \cP(\Omega)$ with $\cH(\frp|\frq)<\oo$. If there exist $p,\,q>0$ such that $h\in L^{q}(\frq)$ and $h^{-1}\in L^{p}(\frq)$, then $\log h\in L^1(\frp)$. More precisely, it holds
\begin{equation}\label{bound:L1}\int\abs{\log h}\De\frp\leq 2\,e^{\cH(\frp|\frq)-1}\,\left[ \frac{1}{1\wedge p\wedge q}\vee \log_2\left(\frq\{h\geq1\}^{1-\frac{1}{q}}\norm{h}_{L^q(\frq)}+\frq\{h<1\}^{1-\frac{1}{p}}\norm{h^{-1}}_{L^p(\frq)}\right)\right]\,,\end{equation}
and also
\begin{equation}\label{bound:L1:no:measure}
\int\abs{\log h}\De\frp\leq \frac{2\,e^{\cH(\frp|\frq)-1}}{p\wedge q}\,\left[1\vee\log_2\left(\norm{h}_{L^q(\frq)}^{p\wedge q}+\norm{h^{-1}}_{L^p(\frq)}^{p\wedge q}\right)\right]\,.
\end{equation}
\end{lemma}

\begin{proof}
By means of the Orlicz-Young inequality \eqref{Orlicz:Young} we have 
\begin{equation}\label{stima:ass}\int\abs{\log h}\De\frp=\int\abs{\frac{\De\frp}{\De\frq}\,\log h}\De\frq\leq 2 \norm{\log h}_\theta\,\norm{\frac{\De\frp}{\De\frq}}_{\theta^*}= 2  \norm{\log h}_\theta\, e^{\cH(\frp|\frq)-1}\,.\end{equation}
Therefore it suffices to show that the above Luxemburg norm $\norm{\log h}_\theta$ is finite.
For any $b>0$ we have
\begin{align}\label{b:theta}
    \int\theta\left(\frac{\abs{\log h}}{b}\right)\De\frq=\int\theta\left(\abs{\log h^{\frac1b}}\right)\De\frq=\int_{\{h\geq 1\}} h^{\frac1b}\De\frq+\int_{\{h<1\}}h^{-\frac1b}\De\frq -1\,.
\end{align}
Before proceeding, let us point out that the above right-hand side is always non-negative since it is greater or equal to $\frq\{h\geq 1\}+\frq\{h<1\}-1=0$.

Now consider any parameter $\delta\in(0,\oo)$ (to be fixed later)
and fix $b\geq\frac{1}{\delta\wedge p\wedge q}$. 
 For the sake of notation, let $P \coloneqq \{h\geq 1\}$ and its complementary $N \coloneqq \{h<1\}$. In what follows we are going to assume that both $\frq(P)>0$ and $\frq(N)>0$; the case $\frq(P)\in\{0,1\}$ can be treated in the same fashion. Now, introduce the probability measures on these sets induced by $\frq$, i.e.\
\[
\De \frq_{|_P}(x)\coloneqq\frac{\De\frq(x)}{\frq(P)}\qquad\text{and}\qquad\De\frq_{|_N}(x)\coloneqq\frac{\De\frq(x)}{\frq(N)}\,.
\]
Since $qb\geq 1$, from Jensen's inequality we deduce that
\begin{align*}
      \int_{\{h\geq 1\}} h^\frac1b\De\frq= \frq(P)\int_{P} h^\frac1b\De\frq_{|_P}=\frq(P)\int_{P} h^{q\,\frac{1}{qb}}\De\frq_{|_P}\leq \frq(P)\left(\int_{P} h^{q}\De\frq_{|_P}\right)^\frac{1}{qb}\\
      =\frq(P)^{1-\frac{1}{qb}}\left(\int_{P} h^{q}\De\frq\right)^\frac{1}{qb}\leq \frq(P)^{1-\frac{1}{qb}}\norm{h}_{L^q(\frq)}^\frac1b\,.
\end{align*}
Similarly, from $pb\geq 1$ we deduce that 
\begin{align*}
     \int_{\{h<1\}}\left(\frac1h\right)^{\frac1b}\De\frq=\frq(N)\int_{N} h^{-\frac1b}\De\frq_{|_N}=\frq(N)\int_{N} h^{-p\,\frac{1}{pb}}\De\frq_{|_N}\leq \frq(N)\left(\int_{N} h^{-p}\De\frq_{|_N}\right)^\frac{1}{pb}\\
      =\frq(N)^{1-\frac{1}{pb}}\left(\int_{N} h^{-p}\De\frq\right)^\frac{1}{pb}\leq \frq(N)^{1-\frac{1}{pb}}\norm{h^{-1}}_{L^p(\frq)}^\frac1b\,.
\end{align*}
For the sake of brevity let $\lambda \coloneqq \frq(P)$. Since also $\delta b\geq 1$, from \eqref{b:theta} and the concavity of $x^\frac{1}{\delta b}$ we deduce
\[
\begin{split}
\int\theta\left(\frac{\abs{\log h}}{b}\right)\De\frq & \leq\lambda^{1-\frac{1}{qb}}\norm{h}_{L^q(\frq)}^\frac1b+(1-\lambda)^{1-\frac{1}{pb}}\norm{h^{-1}}_{L^p(\frq)}^\frac1b-1\\
& =\lambda\left(\frac{\norm{h}_{L^q(\frq)}^\delta}{\lambda^\frac{\delta}{q}}\right)^\frac{1}{\delta b}+(1-\lambda)\left(\frac{\norm{h^{-1}}_{L^p(\frq)}^\delta}{(1-\lambda)^\frac{\delta}{p}}\right)^\frac{1}{\delta b}-1\\
& \leq \left(\lambda^{1-\frac{\delta}{q}}\norm{h}_{L^q(\frq)}^\delta+(1-\lambda)^{1-\frac{\delta}{p}}\norm{h^{-1}}_{L^p(\frq)}^\delta\right)^\frac{1}{\delta b}-1\leq 1\,,
\end{split}
\]
where the last inequality holds as soon as $b\geq \frac{1}{\delta}\,\log_2\left(\lambda^{1-\frac{\delta}{q}}\norm{h}_{L^q(\frq)}^\delta+(1-\lambda)^{1-\frac{\delta}{p}}\norm{h^{-1}}_{L^p(\frq)}^\delta\right)$. Therefore, if we take
\[
\bar b\coloneqq\frac{1}{\delta\wedge p\wedge q}\vee \frac{1}{\delta}\,\log_2\left(\lambda^{1-\frac{\delta}{q}}\norm{h}_{L^q(\frq)}^\delta+(1-\lambda)^{1-\frac{\delta}{p}}\norm{h^{-1}}_{L^p(\frq)}^\delta\right)\,,
\]
we deduce that $\int\theta\left(\frac{\abs{\log h}}{\bar b}\right)\De\frq\leq 1$,
and hence 
\begin{equation}\label{delta:norm}
\norm{\log h}_\theta\leq\bar b=\frac{1}{\delta\wedge p\wedge q}\vee \frac{1}{\delta}\,\log_2\left(\lambda^{1-\frac{\delta}{q}}\norm{h}_{L^q(\frq)}^\delta+(1-\lambda)^{1-\frac{\delta}{p}}\norm{h^{-1}}_{L^p(\frq)}^\delta\right)\,. 
\end{equation}
From this last result and \eqref{stima:ass} it follows that $\log h\in L^1(\frq)$ and that
\begin{equation}\label{delta:L1}
     \int\abs{\log h}\De\frp\leq 2\,e^{\cH(\frp|\frq)-1}\,\left[\frac{1}{\delta\wedge p\wedge q}\vee \frac{1}{\delta}\,\log_2\left(\lambda^{1-\frac{\delta}{q}}\norm{h}_{L^q(\frq)}^\delta+(1-\lambda)^{1-\frac{\delta}{p}}\norm{h^{-1}}_{L^p(\frq)}^\delta\right)\right]\,.
\end{equation}
By taking $\delta=1$ we immediately get \eqref{bound:L1}. In order to get the second bound, take $\delta\coloneqq p\wedge q$. Then \eqref{delta:norm} reads as
 \begin{align}\label{minimi:appo}
     \norm{\log h}_\theta\leq \frac{1}{ p\wedge q}\left[1\vee \log_2\left(\lambda^{1-\frac{p\wedge q}{q}}\norm{h}_{L^q(\frq)}^{p\wedge q}+(1-\lambda)^{1-\frac{p\wedge q}{p}}\norm{h^{-1}}_{L^p(\frq)}^{p\wedge q}\right)\right]\,.
 \end{align}
 Since $\lambda=\frq(P)\in[0,1]$ we deduce that
 \begin{align*}
     \log_2\left(\frq(P)^{1-\frac{p\wedge q}{q}}\norm{h}_{L^q(\frq)}^{p\wedge q}+(1-\frq(P))^{1-\frac{p\wedge q}{p}}\norm{h^{-1}}_{L^p(\frq)}^{p\wedge q}\right)\\
     \leq \log_2\sup_{\lambda\in[0,1]}\left(\lambda^{1-\frac{p\wedge q}{q}}\norm{h}_{L^q(\frq)}^{p\wedge q}+(1-\lambda)^{1-\frac{p\wedge q}{p}}\norm{h^{-1}}_{L^p(\frq)}^{p\wedge q}\right)\\
     =\begin{cases}
     \log_2 \left(\norm{h}_{L^q(\frq)}^{p}+\norm{h^{-1}}_{L^p(\frq)}^{p}\right)\quad\text{if }p\leq q\\
      \log_2 \left(\norm{h}_{L^q(\frq)}^{q}+\norm{h^{-1}}_{L^p(\frq)}^{q}\right)\quad\text{if }q< p
     \end{cases}\\
     =\log_2\left(\norm{h}_{L^q(\frq)}^{p\wedge q}+\norm{h^{-1}}_{L^p(\frq)}^{p\wedge q}\right)\,,
 \end{align*}
 and hence 
 \begin{equation}\label{norma:theta}\norm{\log h}_\theta\leq \frac{1}{ p\wedge q}\left[1\vee\log_2\left(\norm{h}_{L^q(\frq)}^{p\wedge q}+\norm{h^{-1}}_{L^p(\frq)}^{p\wedge q}\right)\right]\,.\end{equation}
 As a byproduct of the above bound and \eqref{stima:ass} we get \eqref{bound:L1:no:measure}.
\end{proof}

\begin{remark}
When $\frq\{h\geq 1\}\in\{0,1\}$, the bound \eqref{bound:L1:no:measure} can be improved. Indeed from \eqref{minimi:appo} we straightforwardly deduce that
\begin{align}\label{extreme:no:measure}
    \int\abs{\log h}\De\frp\leq\,
    \begin{cases} 2\,e^{\cH(\frp|\frq)-1}\,\displaystyle{\bigg(\frac{1}{p\wedge q}\vee\log_2\norm{h}_{L^q(\frq)}\bigg)} \qquad&\text{if }\frq\{h\geq 1\}=1\,,\\
    \\
    2\,e^{\cH(\frp|\frq)-1}\,\displaystyle{\bigg(\frac{1}{p\wedge q}\vee\log_2\norm{h^{-1}}_{L^p(
\frq)}\bigg)} \qquad&\text{if }\frq\{h\geq 1\}=0\,.
    \end{cases}
\end{align}
\end{remark}

\begin{corollary}\label{cor:final:log}
If $p\wedge q\leq1$ in the previous Lemma, then it holds
\begin{equation}\label{final:bound:L1}
    \int\abs{\log h}\De\frp\leq 2\,e^{\cH(\frp|\frq)-1}\,\left[\frac{1}{p\wedge q}+\left(\log_2\,\frac{\norm{h}_{L^q(\frq)}+\norm{h^{-1}}_{L^p(\frq)}}{2}\right)^+\,\right]\,.
\end{equation}
\end{corollary}

\bibliographystyle{plain}
\bibliography{revised2}

\begin{thebibliography}{10}

\bibitem{mark}
Stefan Adams, Nicolas Dirr, Mark~A. Peletier, and Johannes Zimmer.
\newblock From a {L}arge-{D}eviations {P}rinciple to the {W}asserstein
  {G}radient {F}low: {A} {N}ew {M}icro-{M}acro {P}assage.
\newblock {\em Communications in Mathematical Physics}, 307(3):791--815, 2011.

\bibitem{AGS14}
Luigi Ambrosio, Nicola Gigli, and Giuseppe Savar{\'e}.
\newblock Calculus and heat flow in metric measure spaces and applications to
  spaces with ricci bounds from below.
\newblock {\em Inventiones mathematicae}, 195(2):289--391, 2014.

\bibitem{AGS14b}
Luigi Ambrosio, Nicola Gigli, and Giuseppe Savar{\'e}.
\newblock Metric measure spaces with {R}iemannian {R}icci curvature bounded
  from below.
\newblock {\em Duke Mathematical Journal}, 163(7):1405--1490, 2014.

\bibitem{backhoff2020mean}
Julio Backhoff, Giovanni Conforti, Ivan Gentil, and Christian L{\'e}onard.
\newblock The mean field {S}chr{\"o}dinger problem: ergodic behavior, entropy
  estimates and functional inequalities.
\newblock {\em Probability Theory and Related Fields}, 178(1):475--530, 2020.

\bibitem{bakry2013analysis}
Dominique Bakry, Ivan Gentil, and Michel Ledoux.
\newblock {\em Analysis and geometry of Markov diffusion operators}, volume
  348.
\newblock Springer Science \& Business Media, 2013.

\bibitem{barthe2008}
Franck Barthe and Alexander~V. Kolesnikov.
\newblock {M}ass {T}ransport and {V}ariants of the {L}ogarithmic {S}obolev
  {I}nequality.
\newblock {\em Journal of Geometric Analysis}, 18:921--979, 2008.

\bibitem{benamou2021optimal}
Jean-David Benamou.
\newblock Optimal transportation, modelling and numerical simulation.
\newblock {\em Acta Numerica}, 30:249--325, 2021.

\bibitem{BerntonGhosalNutz}
Espen Bernton, Promit Ghosal, and Marcel Nutz.
\newblock {E}ntropic {O}ptimal {T}ransport: {G}eometry and {L}arge
  {D}eviations.
\newblock {\em Duke Mathematical Journal}, forthcoming.

\bibitem{brenier1991polar}
Yann Brenier.
\newblock Polar factorization and monotone rearrangement of vector-valued
  functions.
\newblock {\em Communications on pure and applied mathematics}, 44(4):375--417,
  1991.

\bibitem{CarlierLaborde}
Guillaume Carlier and Maxime Laborde.
\newblock A {D}ifferential {A}pproach to the {M}ulti-{M}arginal
  {S}chr{\"o}dinger {S}ystem.
\newblock {\em SIAM Journal on Mathematical Analysis}, 52(1):709--717, 2020.

\bibitem{CPT22}
Guillaume Carlier, Paul Pegon, and Luca Tamanini.
\newblock Convergence rate of general entropic optimal transport costs.
\newblock {\em Preprint, arXiv:2206.03347}, 2022.

\bibitem{chen2021stochastic}
Yongxin Chen, Tryphon~T Georgiou, and Michele Pavon.
\newblock Stochastic {C}ontrol {L}iaisons: {R}ichard {S}inkhorn {M}eets
  {G}aspard {M}onge on a {S}chrodinger {B}ridge.
\newblock {\em SIAM Review}, 63(2):249--313, 2021.

\bibitem{CCG21}
Alberto Chiarini, Giovanni Conforti, Giacomo Greco, and Zhenjie Ren.
\newblock Entropic turnpike estimates for the kinetic schr{\" o}dinger problem.
\newblock {\em Preprint, arXiv:2108.09161}, 2021.

\bibitem{chizat2020faster}
Lenaic Chizat, Pierre Roussillon, Flavien L{\'e}ger, Fran{\c{c}}ois-Xavier
  Vialard, and Gabriel Peyr{\'e}.
\newblock Faster {W}asserstein distance estimation with the {S}inkhorn
  divergence.
\newblock {\em Advances in Neural Information Processing Systems},
  33:2257--2269, 2020.

\bibitem{clerc2020variational}
Gauthier Clerc, Giovanni Conforti, and Ivan Gentil.
\newblock On the variational interpretation of local logarithmic {S}obolev
  inequalities.
\newblock {\em forthcoming in {A}nnales de la facult{\'e} des Sciences de
  {T}olouse}, 2020.

\bibitem{conforti2019second}
Giovanni Conforti.
\newblock A second order equation for {S}chr{\"o}dinger bridges with
  applications to the hot gas experiment and entropic transportation cost.
\newblock {\em Probability Theory and Related Fields}, 174(1):1--47, 2019.

\bibitem{conforti2021formula}
Giovanni Conforti and Luca Tamanini.
\newblock A formula for the time derivative of the entropic cost and
  applications.
\newblock {\em Journal of Functional Analysis}, 280(11):964--1008, 2021.

\bibitem{DePFig2014}
Guido De~Philippis and Alessio Figalli.
\newblock The {M}onge--{A}mp{\`e}re equation and its link to optimal
  transportation.
\newblock {\em Bulletin of the American Mathematical Society}, 51(4):527--580,
  2014.

\bibitem{Deligiannidis2021}
George Deligiannidis, Valentin De~Bortoli, and Arnaud Doucet.
\newblock Quantitative {U}niform {S}tability of the {I}terative {P}roportional
  {F}itting {P}rocedure.
\newblock {\em Preprint, arXiv:2108.08129}, 2021.

\bibitem{Marcel:quant:stability}
Stephan Eckstein and Marcel Nutz.
\newblock Quantitative {S}tability of {R}egularized {O}ptimal {T}ransport and
  {C}onvergence of {S}inkhorn's {A}lgorithm.
\newblock {\em Preprint, arXiv:2110.06798}, 2021.

\bibitem{ErbarMassRenger}
Matthias Erbar, Jan Maas, and Michiel Renger.
\newblock From large deviations to {W}asserstein gradient flows in multiple
  dimensions.
\newblock {\em Electronic Communications in Probability}, 20:1--12, 2015.

\bibitem{Figalli07}
Alessio Figalli.
\newblock Existence, uniqueness, and regularity of optimal transport maps.
\newblock {\em SIAM Journal on Mathematical Analysis}, 39(1):126--137, 2007.

\bibitem{FigGigli2011}
Alessio Figalli and Nicola Gigli.
\newblock Local semiconvexity of {K}antorovich potentials on non-compact
  manifolds.
\newblock {\em ESAIM: Control, Optimisation and Calculus of Variations},
  17(3):648--653, 2011.

\bibitem{gentil2020dynamical}
Ivan Gentil, Christian L{\'e}onard, and Luigia Ripani.
\newblock Dynamical aspects of the generalized {S}chr{\"o}dinger problem via
  {O}tto calculus--{A} heuristic point of view.
\newblock {\em Revista Matem{\'a}tica Iberoamericana}, 36(4):1071--1112, 2020.

\bibitem{Marcel:EOT-deviations}
Promit Ghosal, Marcel Nutz, and Espen Bernton.
\newblock Entropic {O}ptimal {T}ransport: {G}eometry and {L}arge {D}eviations.
\newblock {\em Preprint, arXiv:2102.04397}, 2021.

\bibitem{ghosal2021stability}
Promit Ghosal, Marcel Nutz, and Espen Bernton.
\newblock Stability of {E}ntropic {O}ptimal {T}ransport and {S}chr{\"o}dinger
  {B}ridges.
\newblock {\em Preprint, arXiv:2106.03670}, 2021.

\bibitem{GigTam19}
Nicola Gigli and Luca Tamanini.
\newblock Benamou-{B}renier and duality formulas for the entropic cost on
  ${RCD}^*({K},{N})$ spaces.
\newblock {\em Probability Theory and Related Fields}, pages 1--34, 2019.

\bibitem{GigTam21}
Nicola Gigli and Luca Tamanini.
\newblock Second order differentiation formula on ${RCD}^*({K},{N})$ spaces.
\newblock {\em Journal of the European Mathematical Society}, 23(5):1727--1795,
  2021.

\bibitem{greene1989s}
Robert~E Greene.
\newblock S. gallot, d. hulin and j. lafontaine, riemannian geometry.
\newblock {\em Bulletin (New Series) of the American Mathematical Society},
  21(1):157--162, 1989.

\bibitem{Grigoryan09}
Alexander Grigoryan.
\newblock {\em Heat kernel and analysis on manifolds}, volume~47.
\newblock American Mathematical Soc., 2009.

\bibitem{Hamilton93}
Richard~S. Hamilton.
\newblock A matrix {H}arnack estimate for the heat equation.
\newblock {\em Comm. Anal. Geom.}, 1(1):113--126, 1993.

\bibitem{JLZ14}
Renjin Jiang, Huaiqian Li, and Huichun Zhang.
\newblock Heat {K}ernel {B}ounds on {M}etric {M}easure {S}paces and {S}ome
  {A}pplications.
\newblock {\em Potential Analysis}, 44:601--627, 2016.

\bibitem{JiangZhang16}
Renjin Jiang and Huichun Zhang.
\newblock Hamilton’s gradient estimates and a monotonicity formula for heat
  flows on metric measure spaces.
\newblock {\em Nonlinear Analysis}, 131:32--47, 2016.

\bibitem{Kotschwar07}
Brett~L. Kotschwar.
\newblock Hamilton's gradient estimate for the heat kernel on complete
  manifolds.
\newblock {\em Proc. Amer. Math. Soc.}, 135(9):3013--3019, 2007.

\bibitem{Leo12}
Christian L{\'e}onard.
\newblock From the {S}chr\"odinger problem to the {M}onge–{K}antorovich
  problem.
\newblock {\em Journal of Functional Analysis}, 262(4):1879--1920, 2012.

\bibitem{Leo2012Girsanov}
Christian L{\'e}onard.
\newblock {\em Girsanov Theory Under a Finite Entropy Condition}, pages
  429--465.
\newblock Springer Berlin Heidelberg, Berlin, Heidelberg, 2012.

\bibitem{LeoSch}
Christian L{\'e}onard.
\newblock A survey of the {S}chrödinger problem and some of its connections
  with optimal transport.
\newblock {\em Discrete and Continuous Dynamical Systems}, 34(4):1533--1574,
  2014.

\bibitem{Mikami04}
Toshio Mikami.
\newblock Monge’s problem with a quadratic cost by the zero-noise limit of
  $h$-path processes.
\newblock {\em Probability Theory and Related Fields}, 129:245--260, 2004.

\bibitem{Norris97}
James~R Norris.
\newblock Heat kernel asymptotics and the distance function in {L}ipschitz
  {R}iemannian manifolds.
\newblock {\em Acta Mathematica}, 179(1):79--103, 1997.

\bibitem{Marcel:notes}
Marcel Nutz.
\newblock Introduction to {E}ntropic {O}ptimal {T}ransport.
\newblock {\em
  \url{http://www.math.columbia.edu/~mnutz/docs/EOT_lecture_notes.pdf}}, 2021.

\bibitem{nutz2021entropic}
Marcel Nutz and Johannes Wiesel.
\newblock Entropic optimal transport: {C}onvergence of potentials.
\newblock {\em Probability Theory and Related Fields}, pages 1--24, 2021.

\bibitem{Marcel:stab:potentials}
Marcel Nutz and Johannes Wiesel.
\newblock Stability of {S}chr\"odinger {P}otentials and {C}onvergence of
  {S}inkhorn's {A}lgorithm.
\newblock {\em Preprint, arXiv:2201.10059}, 2022.

\bibitem{soumik19}
Soumik Pal.
\newblock On the difference between entropic cost and the optimal transport
  cost.
\newblock {\em Preprint, arXiv:1905.12206}, 2019.

\bibitem{peyre2019computational}
Gabriel Peyr{\'e} and Marco Cuturi.
\newblock Computational optimal transport.
\newblock {\em Foundations and Trends in Machine Learning}, 11(5-6):355--607,
  2019.

\bibitem{Comp:H-1:W}
R\'emi Peyre.
\newblock Comparison between $\mathcal{W}_2$ distance and $\dot{H}^{-1}$ norm,
  and {L}ocalisation of {W}asserstein dinstance.
\newblock {\em ESAIM: Control, Optimisation and Calculus of Variations},
  24(4):1489--1501, 2018.

\bibitem{EstimationEOTmap}
Aram-Alexandre Pooladian and Jonathan Niles-Weed.
\newblock Entropic estimation of optimal transport maps.
\newblock {\em Preprint, arXiv:2109.12004}, 2021.

\bibitem{Schr32}
Erwin Schr\"odinger.
\newblock La th\'eorie relativiste de l'\'electron et l' interpr\'etation de la
  m\'ecanique quantique.
\newblock {\em Ann. Inst Henri Poincar\'e}, 2:269 -- 310, 1932.

\bibitem{Sturm06I}
Karl-Theodor Sturm.
\newblock On the geometry of metric measure spaces. {I}.
\newblock {\em Acta mathematica}, 196(1):65--131, 2006.

\bibitem{Sturm06II}
Karl-Theodor Sturm.
\newblock On the geometry of metric measure spaces. {II}.
\newblock {\em Acta Math.}, 196(1):133--177, 2006.

\bibitem{luca2022costa}
Luca Tamanini.
\newblock A generalization of {C}osta's entropy power inequality.
\newblock {\em IEEE Transactions on Information Theory}, 68(7):4224--4229,
  2022.

\bibitem{Villani:TopicsOT}
C{\'e}dric Villani.
\newblock {\em Topics in Optimal Transportation}, volume~58.
\newblock American Mathematical Society, 2003.

\bibitem{Wang11}
Feng-Yu Wang.
\newblock Equivalent semigroup properties for the curvature-dimension
  condition.
\newblock {\em Bulletin des sciences mathematiques}, 135(6-7):803--815, 2011.

\end{thebibliography}

\end{document}